\documentclass[12pt]{amsart}

\usepackage{amssymb}
\usepackage{esint}
\usepackage{mathtools}
\usepackage[english]{babel}
\usepackage{epsfig}
\usepackage{mathptmx}
\usepackage{amsmath,amsfonts,amssymb}
\usepackage{mathtools}
\usepackage{enumitem}
\usepackage{mathrsfs}
\usepackage[all]{xy}
\usepackage{graphicx}
\usepackage{latexsym}
\usepackage{verbatim}
\usepackage{xcolor}
\usepackage[normalem]{ulem}
\usepackage{cancel}
\usepackage[hidelinks]{hyperref}
\makeatletter
\@namedef{subjclassname@2020}{\textup{2020} Mathematics Subject Classification}
\makeatother

\usepackage[text={16.5cm, 22.25cm}, asymmetric]{geometry}
\numberwithin{equation}{section}

\def\Re{\mathop{\mathrm{Re}}}
\def\Im{\mathop{\mathrm{Im}}}

\def\D{\mathbb D}

\def\H{\mathbb{H}}

\def\R{\mathbb R}
\def\Z{\mathbb Z}
\def\C{\mathbb C}
\def\Hol{{\sf Hol}}
\def\Aut{{\sf Aut}(\mathbb D)}
\def\Autm{\mathop{\mathsf{Aut}}}

\def\N{\mathbb N}

\def\const{{\rm const}}
\def\id{{\sf id}}

\DeclareMathOperator{\sgn}{\mathrm{sgn}}
\newcommand{\widesim}[1]{
  \mathrel{\lower.36ex\hbox{$\overset{#1}{\scalebox{1.25}[1]{$\sim$}}$}}
}

\newtheorem{theorem}{Theorem}[section]
\newtheorem{lemma}[theorem]{Lemma}
\newtheorem{proposition}[theorem]{Proposition}
\newtheorem{corollary}[theorem]{Corollary}

\theoremstyle{definition}
\newtheorem{definition}[theorem]{Definition}
\newtheorem{example}[theorem]{Example}

\theoremstyle{remark}
\newtheorem{remark}[theorem]{Remark}

\numberwithin{equation}{section}

\newcommand{\interior}{\mathop{\mathsf{int}}}

\newcommand{\Zen}{\mathcal{Z}}

\newcommand{\U}{\mathcal{U}}
\newcommand{\UH}{\mathbb{H}}
\newcommand{\UD}{\mathbb{D}}
\newcommand{\UC}{\partial\UD}
\newcommand{\Complex}{\mathbb{C}}
\newcommand{\Real}{\mathbb{R}}
\newcommand{\Natural}{\mathbb{N}}
\newcommand{\di}{\mathop{\mathrm{d}}\nolimits}
\newcommand{\anglim}{\angle\lim}

\renewcommand{\emptyset}{\varnothing}
\renewcommand{\ge}{\geqslant}
\renewcommand{\le}{\leqslant}
\renewcommand{\geq}{\geqslant}
\renewcommand{\leq}{\leqslant}
\newcommand{\proofof}[1]{{\fontseries{bx}\fontshape{it}\selectfont Proof of #1}}

\newcommand{\StepP}[1]{\medskip\noindent{\textsc{Proof of #1.}}}

\newcommand{\StepC}[2]{\medskip\noindent{\textsc{Case}~#1: #2.}}

\newcommand{\dff}[1]{\textsl{#1}}

\newcommand{\Tmap}{\mathfrak T}
\newcommand{\Smap}{{\mathfrak S}}

\newcommand{\PGdlc}{\color[rgb]{0.67,0.5,0.00}}
\newcommand{\PGaddc}{\color[rgb]{1,0.1,0.3}}

\newcommand{\PGC}[1]{{\color[rgb]{.25,0.1,0.9} PG: #1}}
\newcommand{\PGCm}[1]{{\color[rgb]{.2,0.7,0.3} PG: #1}}
\newcommand{\PGadd}[1]{{\PGaddc{}#1}}
\newcommand{\PGdl}[1]{{\PGdlc\ifmmode\xcancel{#1}\else\sout{#1}\fi}}
\newcommand{\PGdlm}[1]{{\PGdlc#1}}
\newcommand{\PGcr}[2]{\PGdl{#1}\,\,\PGadd{#2}}
\newcommand{\PGQ}[1]{\textbf{\textcolor[rgb]{0.82,0.27,0.00}{#1}}}
\newcommand{\REM}[1]{\relax}

\newcommand{\CRR}[1]{\marginpar{\def\mrkcol{0,0,0} \ifcase#1\relax%
                                                      \or\def\mrkcol{0.00,0.50,1.00}%
                                                      \or\def\mrkcol{1.00,0.00,1.00}%
                                                      \or\def\mrkcol{0.00,0.00,0.00}%
                                                      \or\def\mrkcol{0.25,0.25,1.00}%
                                                      \or\def\mrkcol{0.75,0,0.75}%
                                                      \or\def\mrkcol{0.99,0.44,0.02}%
                                                      \or\def\mrkcol{0.00,0.80,0.91}%
                                                      \or\def\mrkcol{1,0.0,0.25}%
                                                      \else\relax%
                                                      \fi%
                                   \expandafter\textcolor[rgb]{\mrkcol}{%
                                          \ifcase#1\tt CRR%
                                          \or\tt CRR\\Pavel Oct.\,15-18%
                                          \or{\tt CRR2}\\{\small by Pavel for\\ the Examples}
                                          \or{\tt CRR3}\\{\small by Manolo\\ and Santiago\\ Oct.\, 28}
                                          \or{\tt CRR4}\\{\small $\prod.\Gamma.$ Oct.\,30}
                                          \or{\tt CRR5}\\{\small en SEVILLA}
                                          \or{\tt CRR6}\\{últimas}
                                          \or{\tt CRR7}\\{bis}
                                           \or{\tt CRR8-$\Pi$.}
                                          \else\relax\fi}}}

\newenvironment{ourlist}{\begin{enumerate}[label={\rm (\arabic*)}, ref={\rm (\arabic*)}, left=.7em]}{\end{enumerate}}
\newenvironment{Ourlist}{\begin{enumerate}[label={\bf (\Alph*)}, ref={\rm (\Alph*)}, left=0em]%
\everydisplay{\makeatletter\def\@eqnum{\normalfont(\theequation)}\makeatother}}{\end{enumerate}}
\newenvironment{OurlistM}{\begin{enumerate}[label={\bf (\hskip.07em\Roman*\hskip.07em)}, ref={\rm (\hskip.05em\Roman*\hskip.05em)}, left=0em]%
\everydisplay{\makeatletter\def\@eqnum{\normalfont(\theequation)}\makeatother}}{\end{enumerate}}
\newenvironment{OurL}{\begin{enumerate}[label={\rm(\Alph*:}, ref={\rm \Alph*}]}{\end{enumerate}}

\newcommand{\SeC}[1]{{\color{violet} SE: #1}}
\newcommand{\Seadd}[1]{{\color{cyan} #1}}

\newcommand{\HideComments}{%
\renewcommand{\PGC}[1]{\relax}
\renewcommand{\PGCm}[1]{\relax}
\renewcommand{\PGadd}[1]{##1}
\renewcommand{\PGdl}[1]{\relax}
\renewcommand{\PGdlm}[1]{\relax}
\renewcommand{\PGcr}[2]{##2}
\renewcommand{\PGQ}[1]{\relax}
\renewcommand{\Seadd}[1]{##1}
\renewcommand{\SeC}[1]\relax
\renewcommand{\CRR}[1]\relax
}

\begin{document}\HideComments
\title[Centralizers and the embedding problem]{Centralizers of non-elliptic univalent self-maps  and the embeddability problem in the unit disc}
\author[M. D. Contreras]{Manuel D. Contreras $^\dag$}

\author[S. D\'{\i}az-Madrigal]{Santiago D\'{\i}az-Madrigal $^\dag$}
\address{Camino de los Descubrimientos, s/n\\
	Departamento de Matem\'{a}tica Aplicada II and IMUS\\
	Universidad de Sevilla\\
	Sevilla, 41092\\
	Spain.}\email{contreras@us.es} \email{madrigal@us.es}

\author[P. Gumenyuk]{Pavel Gumenyuk$^\ddag$}\address{Pavel Gumenyuk: Department of
Mathematics\\ Milano Politecnico, via E. Bonardi 9\\ Milan 20133, Italy.}
\email{pavel.gumenyuk@polimi.it}

\date\today

\thanks{This research was supported in part by Ministerio de Innovaci\'on y Ciencia, Spain, project PID2022-136320NB-I00, and Junta de Andaluc\'ia, project P20\_00664.}

\maketitle

\begin{abstract}
The embeddability problem is a very old and hard problem in discrete holomorphic iteration which deals with determining general conditions on a given univalent self-map $\varphi$ of the unit disc $\D$ in order to be contained in a continuous one-parameter semigroup. In this paper, we tackle this embedding problem by establishing different dichotomy results about the centralizer of $\varphi$ (i.e. the set of all univalent self-maps commuting with $\varphi$) which depend strongly on the dynamical character of $\varphi$. Our approach is, in part, based on a new technique to obtain simultaneous linearizations of two non-elliptic  univalent self-maps of the unit disc, which might be interesting on their own. We also introduce and study several closed additive subsemigroups of the complex plane that collect the main features of the centralizer of~$\varphi$ and which play a prominent position in those dichotomy results.
\end{abstract}

\tableofcontents

%%%%%%%%%%%%%%%%%%%%%%%%%%%%%%%%%%%%%%%%%%%%%%%%%%%%%%%%%%%%%%%%%%%%%%%%%%%%%%%%%%%%%%%%%%%%%%%%%%
%%%%%%%%%%%%%%%%%%%%%%%%%%%%%%%%%%%%%%%%%%%%%%%%%%%%%%%%%%%%%%%%%%%%%%%%%%%%%%%%%%%%%%%%%%%%%%%%%%
%%%%%%%%%%%%%%%%%%%%%%%%%%%%%%%%%%%%%%%%%%%%%%%%%%%%%%%%%%%%%%%%%%%%%%%%%%%%%%%%%%%%%%%%%%%%%%%%%%
%%%%%%%%%%%%%%%%%%%%%%%%%%%%%%%%%%%%%%%%%%%%%%%%%%%%%%%%%%%%%%%%%%%%%%%%%%%%%%%%%%%%%%%%%%%%%%%%%%
%%%%%%%%%%%%%%%%%%%%%%%%%%%%%%%%%%%%%%%%%%%%%%%%%%%%%%%%%%%%%%%%%%%%%%%%%%%%%%%%%%%%%%%%%%%%%%%%%%
%%%%%%%%%%%%%%%%%%%%%%%%%%%%%%%%%%%%%%%%%%%%%%%%%%%%%%%%%%%%%%%%%%%%%%%%%%%%%%%%%%%%%%%%%%%%%%%%%%
\section{Introduction}
%%%%%%%%%%%%%%%%%%%%%%%%%%%%%%%%%%%%%%%%%%%%%%%%%%%%%%%%%%%%%%%%%%%%%%%%%%%%%%%%%%%%%%%%%%%%%%%%%%
%%%%%%%%%%%%%%%%%%%%%%%%%%%%%%%%%%%%%%%%%%%%%%%%%%%%%%%%%%%%%%%%%%%%%%%%%%%%%%%%%%%%%%%%%%%%%%%%%%
%%%%%%%%%%%%%%%%%%%%%%%%%%%%%%%%%%%%%%%%%%%%%%%%%%%%%%%%%%%%%%%%%%%%%%%%%%%%%%%%%%%%%%%%%%%%%%%%%%

The embeddability problem is the problem to determine whether a given univalent (i.e. injective and holomorphic) self-map of the unit disc $\D:={\{z\in\C:|z|<1\}}$ is contained in a continuous one-parameter semigroup, i.e. in the image of a continuous homomorphism from the semigroup formed by the non-negative reals and endowed with the standard Euclidean topology to the semigroup of all univalent self-maps of~$\D$ with the topology of locally uniform convergence and with the composition as the semigroup operation.  This problem has been satisfactorily solved only for a small collection of univalent functions. The embedding problem for linear-fractional self-maps of the unit disc was solved in \cite[Theorem~3.4]{BCD-LFM} (see also \cite[Theorems~9.5.2 and 9.5.3]{BCD-Book}). In particular, it was shown that for every non-elliptic linear-fractional self-map~$\varphi$ of the unit disc there is a continuous one-parameter semigroup (necessarily also of lineal-fractional maps) $(\phi_t)$ such that ${\phi_{1}=\varphi}$, while if $\varphi$ is elliptic, its embeddability depends on the relative position of the fixed points of $\varphi$ on the Riemann sphere and the spectral value. For the rest of univalent self-maps, this problem is usually extremely hard, unless the self-map luckily fails to have one of the nice properties known for elements of one-parameter semigroups, such as existence of angular limits everywhere on the unit circle.

Despite of its difficulty, the embeddability problem is an area of continued interest. Closely related problems appear in different contexts, e.g. in the study of Markov chains, in Operator Theory, and for real-analytic functions; see e.g. \cite{EmbOp}, \cite[Sect.\,6]{Goryainov-survey}, \cite[Sect.\,4]{Bonet}, and references therein.

In this paper, we tackle the  embeddability problem by studying the centralizer $\Zen(\varphi)$  of a univalent self-map of the unit disc $\varphi$ (i.e. the set of all univalent self-maps commuting with $\varphi$). As a byproduct of the central results of the paper, we are able to relate the embeddability of a univalent self-map with the structure of its centralizer.

The embeddability problem for homeomorphisms and diffeomorphisms of finite-dimensional manifolds, with different regularity assumptions, has been studied since 1970s, see e.g. \cite{Bonomo-Varandas} and references therein. In this framework, the centralizer is directly related to the symmetries of a homeomorphism. Recently,
Damjanovi\'c, Wilkinson, and  Xu~\cite{DWX} have given an explicit description of the centralizers in some classes of volume-preserving ergodic partially hyperbolic $C^{\infty}$-diffeomorphisms and applied this description to determine when such diffeomorphisms can be embedded into  volume-preserving $C^\infty$-flows. Surely, our setting is quite different and requires completely different methods, but the very idea to relate the embeddability to the structure of the centralizer, borrowed from~\cite{DWX}, proved to be fruitful, as our results illustrate, also in the framework of univalent self-maps.

For two explicitly given univalent self-maps it is almost straightforward to check whether they commute or not. However, it is still very hard to describe explicitly the whole centralizer of a univalent self-map. Some interesting preceding studies in this direction can be cited. For example, in the paper~\cite{Pranger}, Pranger studied the centralizer of an elliptic self-map of the unit disc. In the setting of non-elliptic functions (the backdrop of this paper), it is worth mentioning the results of Behan \cite{Behan}, Cowen \cite{Cowen-comm},  Vlacci \cite{Vlacci}, and the joint paper of Bisi and Gentili~\cite{GB}. In a certain sense, the two former papers might be considered the closest ones to the topic of this work.  Most of the results obtained in~\cite{Behan}, \cite{Cowen-comm},  \cite{Vlacci}, and~\cite{GB}  appeared collected in the recent book of Abate \cite[Section~4.10]{Abate2}, where an extensive and very detailed list of results on properties shared by commuting holomorphic self-maps of the unit disc can be found in the endnotes of Section~4.10. We also point out that the centralizer of a parabolic automorphism was described in \cite[Theorem~1.2.27]{Abate}.

In our approach to study the centralizer, we develop a technique to obtain simultaneous linearizations of two commuting non-elliptic univalent self-maps of the unit disc. Namely, in Theorem~\ref{TH_simultaneous}, we show that given two such self-maps $\varphi$ and $\psi$, there exist a univalent function $h_{\varphi,\psi}:\UD\to\C$ and a constant ${c_{\varphi,\psi}\in\C}$ satisfying the following system of Abel's equations:
\begin{equation}\label{EQ_simultaneousintroduction}
h_{\varphi,\psi}\circ\varphi=h_{\varphi,\psi}+1\quad\text{and}\quad
h_{\varphi,\psi}\circ\psi=h_{\varphi,\psi}+c_{\varphi,\psi}.
\end{equation}
An aspect which to our best knowledge is completely new, is that to any abelian part of the centralizer, i.e. a subset~$\Delta$ whose elements commute with each other, there corresponds a \textit{common solution} $h_{\varphi,\Delta}$ to~\eqref{EQ_simultaneousintroduction} independent from the choice of~$\psi$ within~$\Delta$. We also show that the constant $c_{\varphi,\psi}$  can be computed using an explicit formula, see~\eqref{EQ_formula-for-c}, which does not require knowing $h_{\varphi,\Delta}$.

Clearly, simultaneous linearization of commuting self-maps is a quite natural idea and has been already studied in the framework of holomorphic dynamics. Probably and currently, the canonical work on this topic (covering higher complex dimensions) is the paper~\cite{Simultaneous} by Arosio and Bracci; see also the references therein. In fact, in the hyperbolic case, our Theorem~\ref{TH_simultaneous} can be deduced from~\cite[Theorem~3.11]{Simultaneous}.  At the same time, as explained in Section~\ref{S_simultaneous}, the results of~\cite{Simultaneous} do not imply Theorem~\ref{TH_simultaneous} for parabolic self-maps.

The proof of Theorem~\ref{TH_simultaneous} is given in three different sections depending on the properties of~$\varphi$. For parabolic self-maps of zero hyperbolic step and for hyperbolic self-maps, Theorem~\ref{TH_simultaneous} follows from Proposition~\ref{PR_abelian}, which is proved for these two cases in Sections~\ref{S_para-zero} and~\ref{S_hyperbolic}, respectively. Further, in Section~\ref{S_para-positive} we prove Proposition~\ref{PR_simultaneous-para-positive}, which implies Theorem~\ref{TH_simultaneous} for parabolic self-maps of positive hyperbolic step. This last case is by far the most interesting and deepest one. \smallskip

To proceed to our main results, we introduce two topologically closed additive subsemigroups of the complex plane, associated to a given non-elliptic univalent self-map $\varphi$, that collect the main features of its centralizer $\Zen(\varphi)$: $\mathcal A_{\varphi}$ (see Page~\pageref{PropA_varphi} for its definition) and $\mathcal A_{\varphi}^*$ (see Page \pageref{Thm:PHS-S}). With these sets in mind, we can state our results relating embeddability with the structure of the centralizer gathered as follows.

Denote by $\U(\UD)$ the set of all univalent (i.e. injective holomorphic) self-maps of the unit disc~$\UD$.
\begin{theorem}
Let $\varphi\in\U(\UD)$ be a non-elliptic  self-map.
\begin{OurL}\def\itm#1{\item {\rm#1{)}\,}}
\itm{Theorem \ref{TH_dichotomy-nonelliptic}} If $~\id_\UD$ is not isolated in~$\mathcal Z(\varphi)$, then $\mathcal Z(\varphi)$ contains a non-trivial continuous one-parameter semigroup $(\phi_t)$.
    \medskip%
\itm{Proposition \ref{Pro:hyp-dichotomy}}
	If $\varphi $ is hyperbolic and not an automorphism,  then exactly one of the following alternatives hold:\smallskip
		\begin{enumerate}[left=1.2em]
			\item[\rm (i)] either $\id_\UD$ is an isolated point of~$\Zen(\varphi)$ and, in this case, $\varphi$ is not embeddable,
			\item[\rm (ii)] or there exists a unique continuous one-parameter semigroup $(\phi_t)\subset\U(\UD)$ with ${\phi_1=\varphi}\,$ and, in this case $\Zen(\varphi)=\{\phi_t:t\ge0\}$.
		\end{enumerate}
     \medskip%
\itm{Theorem \ref{PR_emb-dichotomy-P0HS}} If $\varphi$ is parabolic of zero hyperbolic step, then exactly one of the following alternatives holds:\smallskip
	\begin{enumerate}[left=1.2em]
		\item[\rm (i)] either there exist $\rho>0$ and $\delta>0$ such that
		$$\mathcal A_{\varphi }\cap\{c\in\C: 0<|c|<\rho, |\mathrm{Arg}(c)|<\delta\}=\emptyset,$$
		\item[\rm (ii)] or there exists a unique one-parameter semigroup $(\phi_t)\subset\U(\UD)$ with ${\phi_1=\varphi}$.
	\end{enumerate}
     \medskip%
     \newpage
\itm{Theorem \ref{TH_emb-dichotomy-P-PHS}} If $\varphi$ is parabolic of positive hyperbolic step, then exactly one of the following alternatives holds:\smallskip
	\begin{enumerate}[left=1.2em]
		\item[\rm (i)] either there exist $\rho>0$ and $\delta>0$ such that
		$$\mathcal A^*_{\varphi }\,\cap\,\{c\in\C: 0<|c|<\rho, |\mathrm{Arg}(c)|<\delta\}~=~\emptyset,$$
		\item[\rm (ii)] or there exists a unique continuous one-parameter semigroup $(\phi_t)\subset\U(\UD)$ with~${\phi_1=\varphi}$.
	\end{enumerate}
\end{OurL}
\end{theorem}

Beyond the above theorem, there are many differences between the centralizers of hyperbolic and parabolic self-maps. For example, Proposition~\ref{Prop:equalcentralizer} shows that if $\varphi\in\U(\UD)$ is hyperbolic and $\psi\in\Zen(\varphi)$, with both of them different from~$\id_\UD$,  then $\Zen(\psi)=\Zen(\varphi)$. Quite the opposite, there are examples of commuting univalent parabolic self-maps $\varphi_{1}, \varphi_{2}$ such that $\mathcal Z(\varphi_{1})\neq \mathcal Z(\varphi_{2})$ (see e.g. Examples~\ref{EX_parab-autom} and~\ref{EX_Z-non-abelian}).

The paper is organized as follows. In Section~\ref{S_preliminaries}, we recall some preliminaries from holomorphic dynamics that will be needed to follow the paper. In Section~\ref{S_first-results}, we relate the centralizer of a univalent self-map with the properties of its holomorphic model.  In  Section~\ref{S_simultaneous}, our result about simultaneous solutions to Abel's equation for two commuting univalent non-elliptic self-maps of the unit disc (Theorem~\ref{TH_simultaneous}) is stated together with some  initial consequences. The next three sections deal, respectively, with parabolic self-maps of zero hyperbolic steps, hyperbolic self-maps, and parabolic self-maps of positive hyperbolic steps. In Section~\ref{S_examples}, we provide a number of examples to show both the scope of our results and different behaviour of the centralizers of non-elliptic self-maps.  We end the paper with an Appendix to show how an arbitrary univalent solution to Abel's equation (which may differ essentially from the Koenigs function) provides the classification of the given univalent self-map. This result has been used in Section~\ref{S_para-positive}.
\medskip

\noindent {\bf Acknowledgement.} We thank Prof. Filippo Bracci for providing us some useful references related to the study of centralizers of holomorphic self-maps in the unit disc. We also appreciate his valuable suggestions that have substantially improved the paper.

%%%%%%%%%%%%%%%%%%%%%%%%%%%%%%%%%%%%%%%%%%%%%%%%%%%%%%%%%%%%%%%%%%%%%%%%%%%%%%%%%%%%%%%%%%%%%%%%%%
%%%%%%%%%%%%%%%%%%%%%%%%%%%%%%%%%%%%%%%%%%%%%%%%%%%%%%%%%%%%%%%%%%%%%%%%%%%%%%%%%%%%%%%%%%%%%%%%%%
%%%%%%%%%%%%%%%%%%%%%%%%%%%%%%%%%%%%%%%%%%%%%%%%%%%%%%%%%%%%%%%%%%%%%%%%%%%%%%%%%%%%%%%%%%%%%%%%%%
%%%%%%%%%%%%%%%%%%%%%%%%%%%%%%%%%%%%%%%%%%%%%%%%%%%%%%%%%%%%%%%%%%%%%%%%%%%%%%%%%%%%%%%%%%%%%%%%%%
%%%%%%%%%%%%%%%%%%%%%%%%%%%%%%%%%%%%%%%%%%%%%%%%%%%%%%%%%%%%%%%%%%%%%%%%%%%%%%%%%%%%%%%%%%%%%%%%%%
%%%%%%%%%%%%%%%%%%%%%%%%%%%%%%%%%%%%%%%%%%%%%%%%%%%%%%%%%%%%%%%%%%%%%%%%%%%%%%%%%%%%%%%%%%%%%%%%%%
\section{Preliminaries} \label{S_preliminaries}
%%%%%%%%%%%%%%%%%%%%%%%%%%%%%%%%%%%%%%%%%%%%%%%%%%%%%%%%%%%%%%%%%%%%%%%%%%%%%%%%%%%%%%%%%%%%%%%%%%
%%%%%%%%%%%%%%%%%%%%%%%%%%%%%%%%%%%%%%%%%%%%%%%%%%%%%%%%%%%%%%%%%%%%%%%%%%%%%%%%%%%%%%%%%%%%%%%%%%
%%%%%%%%%%%%%%%%%%%%%%%%%%%%%%%%%%%%%%%%%%%%%%%%%%%%%%%%%%%%%%%%%%%%%%%%%%%%%%%%%%%%%%%%%%%%%%%%%%
Below we introduce some notation and basic theory used further in the paper. For a more detail and for the proofs of the results presented in this section, we refer the interested readers to the recent monographs \cite{Abate2,BCD-Book}.

\subsection{Notation}
As usual, we denote the unit disc by
${\UD:=\{z\in\C:|z|<1\}}$, $\D^*$ will be the punctured disc $\D\setminus \{0\}$ and we write $\UH:={\{w\in\C:\Im w>0\}}$ for the upper half-plane and $\UH_{r}:={\{z\in\C:\Re w>0\}}$ for the right half-plane.  As usual, let $\Natural_0:={\Natural\cup\{0\}}$.

Furthermore, denote by $\Hol(D,E)$ the class of all holomorphic mappings of a domain $D\subset\C$ into a set $E\subset\C$,
and let $\U(D,E)$ stand for the class of all \textit{univalent} (i.e. injective holomorphic) mappings from $D$ to~$E$.  As usual, we endow $\Hol(D,E)$ and $\U(D,E)$ with the topology of locally uniform convergence. In case $E=D$, we will write  $\Hol(D)$ and $\U(D)$ instead of $\Hol(D,D)$ and $\U(D,D)$, respectively.
Note also that endowed with the binary operation ${(\varphi,\psi)\mapsto \varphi\circ\psi}$,  the class~$\U(D)$ becomes a topological semigroup with the neutral element~$\id_D$.

For a self-map $\varphi:D\to D$ of a domain $D\subset\C$ and ${n\in\Natural}$ we denote by $\varphi^{\circ n}$ the $n$-th iterate of~$\varphi$, and let $\varphi^{\circ0}:=\id_D$. Moreover, if $\varphi$ is an automorphism of~$D$ and, then for every $n\in\N$, we denote by $\varphi^{\circ-n}$ the $n$-th iterate of~$\varphi^{-1}$.

\subsection{Holomorphic self-maps of the unit disc}
The study of the dynamics of an arbitrary holomorphic self-map $\varphi$ of the unit disc $\mathbb{D}$ is a classical and well-established branch of Complex Analysis. The central result in the area is the Denjoy\,--\,Wolff Theorem, which states that if $\varphi$ is different from an elliptic automorphism (i.e. not an automorphism of~$\UD$ possessing a fixed point in~$\UD$), then the sequence of the iterates $(\varphi ^{\circ n})$ converges locally uniformly in~$\UD$ to a certain point~${\tau\in\overline{\mathbb{D}}}$.  This point  is called the \textit{Denjoy\,--\,Wolff point\/} of $\varphi$. Moreover, if $\tau\in \partial \D$, it is the unique boundary fixed point at which the angular derivative $\varphi'(\tau)$ exists and belongs to $(0,1]$.

According to the position of the Denjoy\,--\,Wolff point~$\tau$ and the value of~$\varphi'(\tau)$, holomorphic self-maps $\varphi\in\Hol(\UD)$ different from elliptic automorphisms are divided into three categories. Namely, $\varphi$ is called:
\begin{itemize}
\item[(a)] \textit{elliptic\/} if $\tau\in\UD$,

\item[(b)] \textit{hyperbolic\/} if $\tau\in \partial \D$ and $\varphi'(\tau )<1$, and

\item[(c)] \textit{parabolic\/} if $\tau
\in \partial \D$ such that $\varphi'(\tau )=1$.
\end{itemize}
The identity mapping~$\id_\UD$ and all elliptic automorphisms of~$\UD$ are conventionally included in the category~(a) of elliptic self-maps. By the Denjoy\,--\,Wolff point  of an elliptic automorphism different from~$\id_\UD$ we mean its unique fixed point in~$\UD$. Finally, we do not define the Denjoy\,--\,Wolff point for ${\varphi=\id_\UD}$.

\begin{definition}
By the \textit{multiplier} of $\varphi\in\Hol(\UD)\setminus\{\id_\UD\}$ we mean the number $\varphi'(\tau)$, where $\tau$ is the Denjoy\,--\,Wolff point of~$\varphi$.
\end{definition}
Note that in the case $\tau\in\UC$, $\varphi'(\tau)$ in the above definition (and in what follows) stands for the angular derivative at~$\tau$.\medskip

As we will see, parabolic self-maps can have very different properties depending on the so-called \textit{hyperbolic step}.
Denote by $\rho_\D$ the hyperbolic distance in $\mathbb{D}$, and let $\varphi\in\Hol(\UD)$ be non-elliptic. Thanks to the Schwarz\,--\,Pick Lemma, for the orbit $\big(z_n\big):=\big(\varphi^{\circ n}(z_0)\big)$ of any point ${z_0\in\UD}$, there exists a finite limit $q(z_0):=\lim_{n\to+\infty} \rho_\D(z_{n},z_{n+1})$. It is known, see e.g. \cite[Corollary\,4.6.9]{Abate2}, that  either $q(z_0)>0$ for all~${z_0\in\UD}~$ or $q\equiv0$ in~$\UD$.  The self-map~$\varphi$ is said to be of \textit{positive} or of \textit{ zero hyperbolic step} depending whether the former or the latter alternative occurs.
If $\varphi$ is
hyperbolic, then it is always of positive hyperbolic step. However, there exist parabolic self-maps of zero as well as of positive hyperbolic step.

\subsection{Commuting holomorphic self-maps}
It is clear that if two  holomorphic self-maps $\varphi,\psi\in\Hol(\UD)\setminus\{\id_\UD\}$ commute, i.e. ${\varphi\circ\psi}={\psi\circ\varphi}$, and if one of them is elliptic, then the other is also elliptic and they share the Denjoy\,--\,Wolff point. The situation is not so evident when we consider non-elliptic self-maps. In 1973, Behan \cite{Behan}, see also \cite[Section 4.10]{Abate2}, proved that if $\varphi, \psi$ are non-elliptic self-maps of~$\UD$ with Denjoy\,--\,Wolff points $\tau_{\varphi}$ and $\tau_{\psi}$, respectively, then:
\begin{enumerate}
\item[(i)] if $\varphi$ is not a hyperbolic automorphism, then $\varphi$ and~$\psi$ share the  Denjoy\,--\,Wolff point, i.e.~$\tau_{\varphi}=\tau_{\psi}$;
\item[(ii)]  if $\varphi$ is a hyperbolic automorphism, then $\psi$ is a hyperbolic automorphism as well and it has the same fixed points as~$\varphi$.
\end{enumerate}

Later, Cowen proved that if $\varphi$ and $\psi$ are two non-elliptic commuting holomorphic self-maps of~$\D$ and if $\varphi$  is hyperbolic, then $\psi$ is also hyperbolic (and thus if $\varphi$ is parabolic, then $\psi$ is parabolic) \cite[Corollary~4.1]{Cowen-comm}, see also \cite[Theorem~1.3]{Simultaneous}.

It is worth mentioning that parabolic self-maps of positive hyperbolic step can commute with parabolic self-maps of zero hyperbolic step, see e.g. Remark~\ref{RM_PHS-commute-with-PHS}.\label{PPP} Moreover, in contrast to the hyperbolic case, the fact that one of them is an automorphism would not imply that the other must be also an automorphism. To see this, one can consider the following example: ${\varphi:=h^{-1}\circ(h+1)}$, ${\psi:=h^{-1}\circ(h+i)}$, where $h$ is a conformal map of~$\UD$ onto~$\UH$. We describe the centralizer of this automorphism~$\varphi$ in Example~\ref{EX_parab-autom}.

\begin{remark}
It is worth mentioning that commuting continuous one-parameter semigroups were studied in~\cite{conReich,conTauraso}; see also \cite[Sect.\,6.3]{EliShobook10}.
\end{remark}

\subsection{Holomorphic models for univalent self-maps}

An indispensable role in our study is played by the concept of a holomorphic model, which goes back to Pommerenke~\cite{Pom79}, Baker and Pommerenke~\cite{BakerPommerenke}, and Cowen~\cite{Cowen} and which is discussed below for the special case of a univalent self-map. The terminology we use is mainly borrowed from~\cite{Canonicalmodel}.

\begin{definition}\label{DF_holomorphic-model} A \textit{holomorphic model} of $\varphi\in\U(\UD)$ is any triple $\mathcal M:=(S,h,\alpha)$, where $S$ is a Riemann surface, $\alpha$ is an automorphism of $S$, and $h$ is a univalent map from $\UD$ into $S$ satisfying the following two conditions:
	\begin{enumerate}[left=2.5em]
		\item[(HM1)] $h\circ \varphi=\alpha \circ h$, {}~and
		\item[(HM2)] $S\,=\,\bigcup_{n\geq0} \alpha^{\circ \, -n}(h(\D))$.
	\end{enumerate}
The Riemann surface $S$ is called the \textit{base space}, and the map $h$ is called the \textit{intertwining map} of the holomorphic model~$\mathcal M$.
\end{definition}

\begin{remark}\label{RM_invariance_wrt_alpha}
Condition (HM1) in the above definition implies that $$\alpha\big(h(\UD)\big)=h\big(\varphi(\UD)\big)\subset h(\UD).$$
\end{remark}

Every $\varphi\in\U(\UD)\setminus\{\id_\UD\}$ admits an essentially unique holomorphic model. More precisely, the following fundamental theorem holds.

\begin{theorem}[\protect{\cite[Theorem~1.1]{Canonicalmodel}}] \label{Thm:uniqness} Every $\varphi\in\U(\UD)$ admits a holomorphic model. Moreover such a model is unique up to a model isomorphism; i.e., if $(S_1,h_1,\alpha_1)$ and $(S_2,h_2,\alpha_2)$ are holomorphic models for $\varphi$, then there exists a biholomorphic map $\eta$ of~$S_1$ onto~$S_2$ such that
	$$
	h_2=\eta\circ h_1,\quad \alpha_2=\eta\circ\alpha_1\circ \eta^{-1}.
	$$
\end{theorem}

\begin{remark}\label{RM_factorization}
The uniqueness of a holomorphic model up to an isomorphism extends to a slightly more general statement.
Let $(S,h,\alpha)$ be a holomorphic model for $\varphi\in\U(\UD)$ and let $S_1$ be a Riemann surface. If
\begin{equation}\label{EQ_intertw}
 h_1\circ\varphi=\alpha_1\circ h_1
\end{equation}
holds for some holomorphic map ${h_1\in\Hol(\UD,S_1)}$ and some ${\alpha_1\in\Autm(S_1)}$, then the intertwining map in~\eqref{EQ_intertw} factorizes via~$h$, i.e. ${h_1=\eta\circ h}$ for a suitable ${\eta\in\Hol(S,S_1)}$.  Observe that $h$ is univalent and hence ${\eta:=h_1\circ h^{-1}}$ is a well-defined holomorphic map from~$h(\UD)$ to~$S_1$. The main point is that $\eta$ extends to a holomorphic map from the whole of~$S$ to~$S_1$. Indeed, combining~\eqref{EQ_intertw} with (HM1), we get ${\eta\circ\alpha=\alpha_1\circ \eta}$. Taking into account the absorption property~(HM2) and  Remark~\ref{RM_invariance_wrt_alpha}, one can use a standard argument, see e.g. \cite[Proof of Proposition~3.10]{Canonicalmodel}, to extend~$\eta$ to all~${z\in S}$ by setting $\eta(z):=\alpha_1^{\circ-n}\big(\eta\big(\alpha^{\circ n}(z)\big)\big)$, where ${n=n(z)\in\Natural}$ is large enough so that ${\alpha^{\circ n}(z)\in h(\UD)}$.
\end{remark}

The type of a univalent self-map (elliptic, hyperbolic, or parabolic) is reflected in, and actually can be fully determined from the kind of holomorphic model $\varphi$ admits.  For an open interval $I\subset\Real$, we define
$$
 S_I:=\Real\times I=\{x+iy:x\in\Real,\,y\in I\}.
$$

\begin{theorem}[\cite{Cowen}, see also \cite{Canonicalmodel}]\label{Thm:model} Let $\varphi\in\U(\UD)\setminus\{\id_\UD\}$. The following statements hold.
\begin{ourlist}
	\item\label{IT_HM-ell-auto} $\varphi$ is an elliptic automorphism with multiplier $\lambda\in\partial\UD\setminus\{1\}$ if and only if $\varphi$ admits a holomorphic model of the form ${\mathcal M_\varphi:=(\UD,h,z\mapsto \lambda z)}$, where ${h\in\Aut}$.
	\item\label{IT_HM-ell-non-auto} $\varphi$ is an elliptic self-map with multiplier $\lambda\in\UD^*$ (and hence it is not an automorphism) if and only if $\varphi$ admits a holomorphic model of the form ${\mathcal M_\varphi:=(\C,h,z\mapsto \lambda z)}$.
	\item\label{IT_HM-hyp} $\varphi$ is a hyperbolic self-map with multiplier $\lambda\in(0,1)$ if and only if $\varphi$ admits a holomorphic model of the form $\mathcal M_\varphi:=(S_{I},h,z\mapsto z+1)$, where $I=(a,b)$ is a bounded open interval of length ${b-a=\pi/|\log\lambda|}$.
	\item\label{IT_HM-para-PHS} $\varphi$ is a parabolic self-map of positive hyperbolic step if and only if $\varphi$ admits a holomorphic model of the form ${\mathcal M_\varphi:=(S_{I},h,z\mapsto z+1)}$, where $I$ is an open unbounded interval different from the whole~$\Real$.
	\item\label{IT_HM-para-0HS} $\varphi$ is a parabolic self-map of zero hyperbolic step if and only if $\varphi$ admits a holomorphic model of the form ${\mathcal M_\varphi:=(\C,h,z\mapsto z+1)}$.	
\end{ourlist}	
\end{theorem}

\begin{remark}\label{RM_normalization}
In the above theorem, we may assume that:
\begin{itemize}
 \item[-]in case~\ref{IT_HM-ell-auto}, $h'(\tau)>0$, where $\tau$ is the Denjoy\,--\,Wolff point of $\varphi$;
 \item[-]in case~\ref{IT_HM-ell-non-auto}, $h'(\tau)=1$, where $\tau$ is the Denjoy\,--\,Wolff point of $\varphi$;
 \item[-]in cases~\ref{IT_HM-hyp} and~\ref{IT_HM-para-0HS}, $h(0)=0$;
 \item[-]in case~\ref{IT_HM-para-PHS}, $\Re h(0)=0$ and ${S_I=S_{(0,+\infty)}=\UH}$ or ${S_I=S_{(-\infty,0)}=-\UH}$.
\end{itemize}
Using the uniqueness part of Theorem~\ref{Thm:uniqness}, one can show (see e.g. \cite[Corollary 4.6.12]{Abate2} for details) that the above assumptions play the role of a normalization under which the holomorphic model $\mathcal M_\varphi$ for a given $\varphi\in\U(\UD)\setminus\{\id_\UD\}$ is unique. Note that the normalization for cases \ref{IT_HM-hyp} and~\ref{IT_HM-para-0HS} would also work in case~\ref{IT_HM-para-PHS}, but we prefer to use another normalization, so that for parabolic self-maps of positive hyperbolic step,  the base space~$S_I$ of~$\mathcal M_\varphi$ coincides with $\UH$ or~$-\UH$. Moreover, replacing, if necessary, $\varphi$ with $z\mapsto\overline{\varphi(\bar z)}$ we may assume that ${S_I=\UH}$.
\end{remark}

\begin{definition}\label{DF_canonical}
 The unique holomorphic model $\mathcal M_\varphi$ of a self-map  $\varphi\in\U(\UD)\setminus\{\id_\UD\}$ defined in Theorem~\ref{Thm:model} and normalized as in Remark~\ref{RM_normalization} is called \dff{canonical (holomorphic) model} for~$\varphi$. The intertwining map~$h$ of the canonical model~$\mathcal M_\varphi$ is called the \dff{Koenigs function}, and ${\Omega:=h(\UD)}$ is called the \textit{Koenigs domain} of~$\varphi$.
\end{definition}

\begin{remark}\label{RM_absorption-continuous}
Let $\varphi\in\U(\UD)$ be a non-elliptic self-map with the canonical model ${(S,h,z\mapsto z+1)}$.  Bearing  in mind Remark~\ref{RM_invariance_wrt_alpha}, it is easy to see that thanks to the absorption property~(HM2),  for any compact set ${K\subset S}$ there exists ${n_K\in\Natural}$ such that for all ${n\in\Natural}$ with ${n\ge n_K}$ we have ${K+n\subset\Omega}:=h(\UD)$. Furthermore, if ${w\in S}$ then ${K(w):=\{w+s:s\in[0,1]\}}$ is a compact subset of~$S$ and we can easily conclude that ${w+t\in \Omega}$ for all ${t\in\Real}$ with ${t\ge n_{K(w)}}$. In particular, it follows that for any $b>0$ the following ``generalized absorption property'' holds:
$$
  \bigcup_{n\in\Natural}\Omega-nb~=~S.
$$
\end{remark}

\subsection{One-parameter semigroups in the unit disc and embeddability problem}\label{SS_one-param-semigr} Semigroups of holomorphic functions in  the unit disc $\D$ have been a subject of study since the early 1900s. In~1978, Berkson and Porta \cite{BP} studied continuous semigroups of holomorphic self-maps of the unit disc in connection with composition operators. This paper meant the resurgence of this area. The current state of the art is presented in the monograph~\cite{BCD-Book}; see also~\cite{EliShobook10}.
We recall the definition straightaway.
\begin{definition} \label{def:semigroup}
We say that a family $(\phi_t)_{t\geq0}$ (or to simplify,  $(\phi_t)$)
of holomorphic functions ${\phi_t:\D\to \D}$ is a \dff{one-parameter semigroup} if it verifies the following two algebraic properties:
\begin{itemize}
    \item[(i)] $\phi_0=\id_{\D}$;
    \item[(ii)] $\phi_t\circ\phi_s=\phi_{t+s}~$ for every $t,s\geq0$.
\end{itemize}
If, in addition, $\phi_t\to\phi_0$ uniformly on compact subsets of $\D$, as $t\to0^+$, we say that the one-parameter semigroup $(\phi_t)$ is \dff{continuous}.
\end{definition}

\smallskip\noindent{\bf Convention.}
From now on, unless explicitly indicated otherwise, by a continuous one-parameter semigroup we will mean a \textit{non-trivial} one, i.e. containing at least one element different from the identity map.
\smallskip

It is worth recalling that all elements of any continuous one-parameter  semigroup are univalent functions (see, e.g., \cite[Theorem~8.1.17]{BCD-Book}). Moreover, all of them, except for the identity map, have the same Denjoy\,--\,Wolff point, so one can talk about elliptic and non-elliptic continuous one-parameter semigroups (see, e.g., \cite[Theorem~8.3.1]{BCD-Book}).

Given a continuous one-parameter semigroup $(\phi_{t})$, it is possible to show that all the functions of the semigroup \textit{essentially} share their canonical model.
Indeed,  for non-elliptic semigroups, if $(S,h,z\mapsto z+1)$ is the canonical holomorphic model for $\phi_{1}$ given in Theorem~\ref{Thm:model}, then the triple $\big(S,h,(z\mapsto z+t)_{t\ge0}\big)$ is the canonical holomorphic model for~$(\phi_t)$ (see~\cite[Theorem~9.3.5]{BCD-Book} for the precise definition and further details). In particular, we have that
\begin{equation}\label{EQ_Abel-eq-for-semigroup}
h\circ \phi_{t}=h+t\qquad\text{for all $~t\geq 0$.}
\end{equation}
It follows that $\phi_t=h^{-1}\circ(h+t)$, we can differentiate this equality w.r.t.~$t$ and get ${\di\phi_t/\di t=G\circ\phi_t}$, where $G:=1/h'$ is called the \dff{infinitesimal generator} of~$(\phi_t)$.

The function $h$ and the domain $\Omega:=h(\UD)$ defined above are called the \dff{Koenigs function} and \dff{Koenigs domain} of the non-elliptic semigroup~$(\phi_t)$.

In a similar way, one can introduce the infinitesimal generator for an elliptic continuous one-parameter semigroup (see, e.g., \cite[Theorem~10.1.4\,(2)]{BCD-Book}). A detailed description of the properties of infinitesimal generators can be found in \cite[Chapter~10]{BCD-Book}.

\begin{definition}
A holomorphic self-map $\varphi:\UD\to\UD$ is said to be \dff{embeddable}, if there exists a continuous one-parameter semigroup $(\phi_t)$ such that ${\phi_1=\varphi}$.
\end{definition}

Note that in the above definition, the trivial semigroup $(\phi_t)$ with ${\phi_t=\id_\UD}$ for all ${t\ge0}$ is allowed, so that the identity map is embeddable.

As we mentioned above, all elements of an non-elliptic one-parameter semigroup $(\phi_t)$ can be expressed using the Koenigs function~$h$ of the self-map~$\phi_1$ via the formula $\phi_t={h^{-1}\circ(h+t)}$. As an immediate consequence we get the following:

\begin{corollary} \label{CR_uniquesemigroup}
Let $\varphi\in\U(\UD)$ be a non-elliptic self-map. Then there is at most one continuous one-parameter semigroup $(\phi_{t})$ such that $\phi_{1}=\varphi$.
\end{corollary}

\begin{remark}\label{RM_emb-ty_via_K-dom}
The fact that the Koenigs map of $\phi_1$ satisfies Abel's equation~\eqref{EQ_Abel-eq-for-semigroup} for the whole semigroup~$(\phi_t)$ has another remarkable corollary. Namely, we have  the following criterion of embeddability: \textit{a non-elliptic self-map $\varphi\in\U(\UD)$ is embeddable if and only if
\begin{equation}\label{EQ_starlike-at-infty}
  \Omega+t\subset\Omega\quad\text{for any~$~t\ge0$},
\end{equation}
where $\Omega$ stands for the Koenigs domain of~$\varphi$.}

This criterion is slightly stronger than the necessary and sufficient condition in~\cite[Theorem~2]{Goryainov-et-al}, which in view of the analytic characterization of hyperbolic simply connected domains~$\Omega$ satisfying~\eqref{EQ_starlike-at-infty}, see e.g. \cite[Sect.\,6]{HengSch1970}, can be stated as follows: a non-elliptic self-map $\varphi$ is embeddable if and only if Abel's equation ${g\circ\varphi}={g+1}$ admits a univalent solution whose image ${\Omega:=g(\UD)}$ satisfies~\eqref{EQ_starlike-at-infty}. Note that in general, Abel's equation can have univalent solutions which are \textit{not} of the form $g=h+c$, where $h$ is the Koenings function of~$\varphi$ and $c\in\C$ is a constant.
\end{remark}

\begin{definition}
A domain $\Omega\subset\C$ satisfying~\eqref{EQ_starlike-at-infty} is said to be \dff{starlike at infinity}.
\end{definition}

It is easy to see that if $\Omega$ is starlike at infinity, then it is simply connected. If, in addition, ${\Omega\neq\C}$, then there exists a non-elliptic continuous one-parameter semigroup for which $\Omega$, up to a translation, is the Koenigs domain. (This semigroup is unique up to conjugating by disc automorphisms.) We will use this fact for constructing examples in Sect.\,\ref{S_examples}.

%%%%%%%%%%%%%%%%%%%%%%%%%%%%%%%%%%%%%%%%%%%%%%%%%%%%%%%%%%%%%%%%%%%%%%%%%%%%%%%%%%%%%%%%%%%%%%%%%%
%%%%%%%%%%%%%%%%%%%%%%%%%%%%%%%%%%%%%%%%%%%%%%%%%%%%%%%%%%%%%%%%%%%%%%%%%%%%%%%%%%%%%%%%%%%%%%%%%%
%%%%%%%%%%%%%%%%%%%%%%%%%%%%%%%%%%%%%%%%%%%%%%%%%%%%%%%%%%%%%%%%%%%%%%%%%%%%%%%%%%%%%%%%%%%%%%%%%%
%%%%%%%%%%%%%%%%%%%%%%%%%%%%%%%%%%%%%%%%%%%%%%%%%%%%%%%%%%%%%%%%%%%%%%%%%%%%%%%%%%%%%%%%%%%%%%%%%%
%%%%%%%%%%%%%%%%%%%%%%%%%%%%%%%%%%%%%%%%%%%%%%%%%%%%%%%%%%%%%%%%%%%%%%%%%%%%%%%%%%%%%%%%%%%%%%%%%%
%%%%%%%%%%%%%%%%%%%%%%%%%%%%%%%%%%%%%%%%%%%%%%%%%%%%%%%%%%%%%%%%%%%%%%%%%%%%%%%%%%%%%%%%%%%%%%%%%%
\section{First results on the centralizer}\label{S_first-results}
%%%%%%%%%%%%%%%%%%%%%%%%%%%%%%%%%%%%%%%%%%%%%%%%%%%%%%%%%%%%%%%%%%%%%%%%%%%%%%%%%%%%%%%%%%%%%%%%%%
%%%%%%%%%%%%%%%%%%%%%%%%%%%%%%%%%%%%%%%%%%%%%%%%%%%%%%%%%%%%%%%%%%%%%%%%%%%%%%%%%%%%%%%%%%%%%%%%%%
%%%%%%%%%%%%%%%%%%%%%%%%%%%%%%%%%%%%%%%%%%%%%%%%%%%%%%%%%%%%%%%%%%%%%%%%%%%%%%%%%%%%%%%%%%%%%%%%%%
Fix $\varphi\in\U(D)$, where $D$ is a domain of $\C$. The \dff{centralizer} of~$\varphi$ is defined as
$$
\Zen_D(\varphi):=\{\psi\in\U(D):\varphi\circ\psi=\psi\circ\varphi\}.
$$
Note that $\Zen_{D}(\varphi)$ is a subsemigroup of~$\U(D)$. In case $D=\UD$ and as long as it does not lead to any confusion, we will drop the subscript and write simply $\Zen(\varphi)$.

Sometimes we consider self-maps $g$ of a Riemann  surface $S$. In such cases, $\U(S)$, the centralizer $\Zen_S(g)$ and the iterates of $g$ are defined in a similar way.

Based on the canonical holomorphic model, an early harvest is the following theorem, in which we construct a continuous injective homomorphism from a subsemigroup~$\mathcal A_\varphi$ of $[\C,+]$ to the centralizer $[\Zen(\varphi),\,\circ\,]$. Later we will show that for hyperbolic self-maps~$\varphi$, as well as for parabolic self-maps of zero hyperbolic step, this homomorphism is actually an isomorphism of topological semigroups; see Theorems~\ref{Thm:0HS} and~\ref{TH_isomorphism}. This is however not the case in general for parabolic self-maps of \textit{positive} hyperbolic step, see Remark~\ref{RM_examples-not-iso}.

\begin{theorem} \label{PropA_varphi}
Let $\varphi \in \U(\UD)$ a non-elliptic self-map and $(S,h,z\mapsto z+1)$ its canonical holomorphic model.  Write $\Omega:=h(\D)$ and take $\mathcal A_{\varphi}:=\{b\in\C:\Omega+b\subset\Omega\}$. Then:
\begin{Ourlist}
\item\label{IT_closed-subsemigroup}  $\mathcal A_{\varphi}$ is a closed additive subsemigroup of $[\C,+]$ containing~$\N_0$.
\item\label{PropA_notauto} If $\varphi$ is not an automorphism, then $\mathcal A_{\varphi}\cap (-\infty,0)=\emptyset$.
\item\label{IT_homomorphism} For each $b\in \mathcal A_{\varphi}$, the function $\psi_{b}(z):=h^{-1}(h(z)+b)$ is well-defined in~$\UD$ and belongs to $\mathcal Z(\varphi)$. The map
     $$
      \Tmap_{\varphi}:\mathcal A_{\varphi} \to  \mathcal Z(\varphi);\quad b\mapsto \Tmap_\varphi(b):=\psi_{b},
     $$
     is injective, continuous, and satisfies $\psi_{b+c}=\psi_{b}\circ \psi_{c}$ for any $b,c\in \mathcal A_\varphi$.
\item\label{IT_Tmap-closed} Moreover, the map $\Tmap_\varphi$ is closed, i.e. for any closed set $X\subset\mathcal A_\varphi$ the image $\Tmap(X)$ is relatively closed in~$\Hol(\UD)$.
\end{Ourlist}
\end{theorem}
%\begin{proof}
\StepP{\ref{IT_closed-subsemigroup}} Clearly, $0\in\mathcal A_\varphi$ and if $b_1,b_2\in\mathcal A_\varphi$, then ${b_1+b_2\in\mathcal A_\varphi}$. Moreover, $1\in\mathcal A_\varphi$ because ${\Omega+1=h(\UD)+1=h(\varphi(\UD))}$.
It remains to see that the set $\mathcal A_\varphi$ is topologically closed. Let $b_0\in\C\setminus\mathcal A_\varphi$.
Then there exists $z_0\in\Omega$ such that $z_0+b_0\not\in\Omega$. Since $\Omega$ is open, there exists ${\varepsilon>0}$ such that $z_0+\zeta\in\Omega$ whenever $|\zeta|<\varepsilon$. It follows that $b:=b_0-\zeta\not\in\mathcal A_\varphi$ for any such $\zeta$ because ${z_0+\zeta+b=z_0+b_0\not\in\Omega}$ while ${z_0+\zeta\in\Omega}$. Hence, $\mathcal A_\varphi$ is a closed set.

\StepP{\ref{PropA_notauto}}
 Choose arbitrary $a_1\le a_2$ such that ${U:=\{w:a_1\le \Im w\le a_2\}\subset S}$. Consider the compact set $${K:=\{w\in U:0\le \Re w\le 2\}\subset S}.$$ The absorption property~(HM2) implies that there is a natural number $m$ such that $K+m\subset \Omega$. Since ${\Omega+1\subset\Omega}$, it follows that
\begin{equation}\label{EQ_half-strip-in-Omega}
 \{w\in U:\, m\le \Re w\}~=~\bigcup_{n=m}^{+\infty}K+n~\subset~\Omega.
\end{equation}

If $\mathcal A_{\varphi}$ contained an element $b<0$, then using~\eqref{EQ_half-strip-in-Omega} and taking into account that $kb\in  \mathcal A_{\varphi}$ for all ${k\in \N}$, we could deduce that ${U\subset\Omega}$.
Since the closed horizontal strip $U\subset S$ in this argument is arbitrary, this would further imply that ${S\subset \Omega=h(\UD)\subset S}$, which is possible only if~$S\neq\C$ and ${\varphi\in\Aut}$.

\medskip\noindent The \textsc{proof of~\ref{IT_homomorphism}} is straightforward and it is, therefore, omitted.

\StepP{\ref{IT_Tmap-closed}} Let $X\subset\mathcal A_\varphi$ be a closed subset. If $X$ is bounded, then $\Tmap_\varphi(X)$ is compact and hence closed, because $\Tmap_\varphi$ is continuous. To show that $\Tmap_\varphi(X)$ is closed in~$\Hol(\UD)$ even if $X$ is unbounded, we write $$X=\bigcup_{n\in\Natural} X_n,\quad X_n:=X\cap\{z:n-1\le |z|\le n\}.$$
Since $\Tmap_\varphi(X_n)$ are all closed, it is sufficient to check that the family $\big(\Tmap_\varphi(X_n)\big)_{n\in\Natural}$ is locally finite in~$\Hol(\UD)$. Suppose on the contrary that there exists a sequence ${(b_n)\subset \C}$ with ${b_n\in X_n}$ for each~${n\in\Natural}$ such that the sequence $\big(\Tmap_\varphi(b_n)\big)$ converges to an element of $\Hol(\UD)$. Passing to the limit in the equality ${h\circ\Tmap_\varphi(b_n)=h+b_n}$, we arrive to a contradiction because ${b_n\to\infty}$ as ${n\to+\infty}$ while the left hand side converges to a holomorphic function. The proof is now complete.
\qed%\end{proof}

\newcommand{\thprop}{Theorem~\ref{PropA_varphi}}

\begin{remark}\label{RM_A_varphi}
The absorption property (HM2) in the definition of a holomorphic model implies that if ${\Omega+b\subset\Omega}$, then ${S+b\subset S}$. It follows that ${\mathcal A_\varphi\subset\Real}$ if $\varphi$ is hyperbolic, and that $\mathcal A_\varphi$ is contained in one of the half-planes bounded by~$\Real$ if $\varphi$ is parabolic of positive hyperbolic step.
\end{remark}

The semigroup $[\mathcal A_{\varphi},+]$  introduced in Theorem~\ref{PropA_varphi} will play a significant role in the study of the centralizer. As we have already mentioned, $\mathcal A_{\varphi}$ and $\Zen(\varphi)$ are isomorphic as topological semigroups (and, in particular, $\Zen(\varphi)={\{\psi_{b}:\, b\in \mathcal A_{\varphi}\}}$)\, if $\varphi$ is a hyperbolic self-map or a parabolic self-map of zero hyperbolic step, but not necessarily in case of parabolic self-maps of positive hyperbolic step. As it becomes clear with the help of our next theorem (applied to the canonical model of~$\varphi$) the reason behind this difference is that if $S$ is either a horizontal strip or the whole plane~$\C$ and if ${g:S\to S}$ is a univalent holomorphic map that commutes with ${z\mapsto z+1}$, then $g$ is an automorphism of~$S$ of the form ${g(z):=z+b}$, where ${b\in\C}$ is a constant.  For the case of a horizontal strip, the latter statement follows from  \cite[Lemma~2.1]{Heins}, while for ${S=\C}$ it is a consequence of the elementary fact that an entire function can be univalent in the whole plane only if it is affine. The situation is more complicated if $S$ is a horizontal half-plane: indeed, there are plenty of univalent holomorphic mappings~$g$ from a half-plane onto its proper subsets satisfying the identity ${g(z+1)=g(z)+1}$; see e.g. Lemma~\ref{Lem:intertwining-exp}\,\ref{IT_periodic-univalent-converse}.

From an abstract point view, the centralizer of a univalent self-map admits a description in terms of holomorphic models as follows.
\begin{remark}\label{RM_relation-to-uniqueness_of-the-model} Let $(S,h,\alpha)$ be a holomorphic model of $\varphi\in\U(\UD)$. Then there is a bijective relationship between $\Zen(\varphi)$ and  the family of all holomorphic models $\mathcal M:=(\tilde S,\tilde h,\widetilde\alpha)$ for $\varphi$ such that ${\tilde S\subset S}$, ${\tilde h(\UD)\subset h(\UD)}$ and $\widetilde\alpha=\alpha|_{\tilde S}$.

Indeed, given  $\psi\in \Zen(\varphi)$, then the univalent function ${h_\psi:=h\circ\psi}$ satisfies the functional equation
\begin{equation}\label{EQ_func-eq-for-h_psi}
 h_\psi\circ\varphi~=~\alpha\circ h_\psi.
\end{equation}
Moreover, $\alpha$ maps ${\Omega_\psi:=h_\psi(\UD)}$ into itself. Indeed, thanks to~\eqref{EQ_func-eq-for-h_psi}, we have
$$
\alpha(\Omega_\psi)=\alpha\big(h_\psi(\UD)\big)=h_\psi\big(\varphi(\UD)\big)\subset\Omega_\psi.
$$
It follows that the restriction of $\alpha$ to $S_\psi:={\bigcup_{n\geq0}\alpha^{\circ-n}(\Omega_\psi)}$ is an automorphism of~$S_\psi$ and that the domain $\Omega_\psi$ is an absorbing set for ${\alpha|_{S_\psi}:S_\psi\to S_\psi}$. Thus, ${(\tilde S,\tilde h,\widetilde\alpha)}=\mathcal M_\psi:={(S_\psi,h_\psi,\alpha|_{S_\psi})}$ is a holomorphic model for~$\varphi$. Clearly, ${S_\psi\subset S}$ and ${h_\psi(\UD)\subset h(\UD)}$.

 Conversely, if for some ${\tilde S\subset S}$ and some univalent function ${\tilde h:\UD\to h(\UD)}$  the triple $\mathcal M:={(\tilde S,\tilde h,\alpha|_{\tilde S})}$ is a holomorphic model for~$\varphi$, then  $\psi:=h^{-1}\circ \tilde h\in\U(\UD)$ and
$$
  \psi\circ \varphi=h^{-1}\circ\big(\tilde h\circ \varphi\big)=h^{-1}\circ\big(\alpha|_{\tilde S}\circ \tilde h\big)=h^{-1}\circ\big(\alpha\circ h\circ \psi\big)=\big(h^{-1}\circ\alpha\circ h\big)\circ \psi=\varphi\circ \psi,
$$
i.e. $\psi\in \Zen(\varphi)$, with $\mathcal M_\psi=\mathcal M$.	
\end{remark}

The above remark together with the uniqueness of the holomorphic model implies a more explicit description of~$\Zen(\varphi)$.

\begin{theorem} \label{Thm:general} Let $(S,h,\alpha)$ be a holomorphic model of some $\varphi\in\U(\UD)$ and denote $\Omega:=h(\D)$. Then
	$$
	\Zen(\varphi)=\{h^{-1}\circ g\circ h: g\in\Zen_{S}(\alpha),\ g(\Omega)\subset\Omega\}.
	$$	
\end{theorem}
\begin{proof} Let $g\in\Zen_{S}(\alpha)$ and suppose that $g(\Omega)\subset\Omega$. Then $(\tilde S,\tilde h,\alpha|_{\tilde S})$, where ${\tilde S:=g(S)}$ and ${\tilde h:=g\circ h}$, is a holomorphic model for~$\varphi$ satisfying ${\tilde S\subset S}$ and ${\tilde h(\UD)\subset h(\UD)}$. Therefore, by Remark~\ref{RM_relation-to-uniqueness_of-the-model}, $\psi:={h^{-1}\circ \tilde h=h^{-1}\circ g\circ h\in\U(\UD)}.$

Conversely, given $\psi\in\Zen_\D(\varphi)$,  consider the corresponding holomorphic model $\mathcal M_\psi={(S_\psi,\,h_\psi:=h\circ\psi,\,\alpha|_{S_\psi})}$ defined in Remark~\ref{RM_relation-to-uniqueness_of-the-model}. By Theorem~\ref{Thm:uniqness}, there exists a conformal map~$g$ of $S$ onto~${\tilde S\subset S}$ such that ${\alpha|_{S_\psi}\circ g}={g\circ \alpha}$ and ${h\circ\psi}={g\circ h}$. The former equality means that $g\in\Zen_{S}(\alpha)$, while the latter one implies that $g(\Omega)=h\big(\psi(\UD)\big)\subset \Omega$ and that $\psi={h^{-1}\circ g\circ h}$.
\end{proof}

Below, roughly speaking, we show that $\Zen(\varphi)$ is a relatively closed subset of~$\U(\UD)$. More precisely, we will now prove the following statement.

\begin{theorem} \label{Thm:general2} Let $\varphi\in\U(\UD)\setminus\{\id_\UD\}$  and let $(\psi_k)\subset\Zen(\varphi)$ be a sequence converging locally uniformly in~$\UD$ to some holomorphic function ${\psi:\UD\to\C}$.
\begin{Ourlist}
	\item\label{IT_2_not_auto} If $\varphi$ is not a hyperbolic automorphism, then either $\psi\in\Zen(\varphi)$ or ${\psi\equiv\tau}$, where $\tau$ is the Denjoy\,--\,Wolff point of~$\varphi$.
\item\label{IT_2_auto} If $\varphi$ is a hyperbolic automorphism, then every $\psi_k$ is an automorphism, and either ${\psi\in\Zen(\varphi)}$ or ${\psi\equiv\sigma}$, where ${\sigma\in\UC}$ is one of the two fixed points of~$\varphi$.
\end{Ourlist}
\end{theorem}
\begin{proof}
If $\psi=\id_\UD$, then there is nothing to prove. If ${\psi\neq\id_\UD}$, then removing a finite number of terms, we may assume that ${\psi_k\neq\id_\UD}$ for all ${k\in\Natural}$.

\StepP{\ref{IT_2_not_auto}} 
Assume first that ${\psi(\UD)\subset\UD}$. Then tending ${k\to+\infty}$ in ${\psi_k\circ \varphi}=\varphi\circ\psi_k$, we get ${\psi\circ\varphi}={\varphi\circ\psi}$.
If ${\psi=\const=:\sigma}$, it follows that ${\varphi(\sigma)=\sigma}$, which is only possible if ${\sigma=\tau}$. If $\psi$ is not constant, then it follows that $\psi$ is univalent and belongs to~$\Zen(\varphi)$.

Assume now that ${\psi(\UD)\not\subset\UD}$. Then $\psi(z)=\sigma$ for all ${z\in\UD}$ and some constant~${\sigma\in\partial\UD}$. We have to show that ${\sigma=\tau}$. First we note that  $\tau\in\UC$, because otherwise the relation ${\psi_k\circ\varphi}={\varphi\circ\psi_k}$ combined with the fact that $\varphi$ cannot have more than one fixed point in~$\UD$ would imply that $\psi_k(\tau)=\tau\in\UD$ for all ${k\in\Natural}$. Now
consider a horodisc $E(\tau,R)$ of center $\tau$ and some radius~${R>0}$. Fix a point ${z_0\in E(\tau,R)}$. Recall that, thanks to Behan's Theorem~\cite[Theorem~6]{Behan}, $\tau$~is the Denjoy\,--\,Wolff point  also for each~$\psi_{k}$. Therefore, by Julia's Lemma, we have $\varphi\big(\psi_{k}(z_0)\big)\in E(\tau,R)$ for all $k\in\Natural$. At the same time,
$$
 \varphi\big(\psi_{k}(z_0)\big)= \psi_{k}\big(\varphi(z_0)\big)~\to~\sigma\in\partial\UD\quad\text{as~$~k\to\infty$.}
$$
The only common point of $\partial\UD$ and the closure of the horodisc $E(\tau,R)$ is $\tau$. Thus,  $\sigma=\tau$.

\StepP{\ref{IT_2_auto}} If $\varphi$ is a hyperbolic automorphism, then by \cite[Lemma~2.1]{Heins}, see also \cite[Theorem~4.10.3]{Abate2}, any non-identity self-map in $\mathcal Z(\varphi)$ is also a hyperbolic automorphism having the same fixed points on~$\partial\UD$ as~$\varphi$. Passing to the upper half-plane~$\UH$ (and suppressing the language by denoting the corresponding self-maps of~$\H$ in the same way as the original self-maps of~$\UD$), we may assume that ${\varphi, \psi_{k}:\H\to \H}$ are given by ${\varphi(w)=\lambda w}$ and ${\psi_{k}(w)=\lambda_{k} w}$, respectively, for all ${z\in\UH}$ and some ${\lambda>1}$ and ${\lambda_{k}>0}$. By hypothesis, $(\lambda_k)_{k\in\Natural}$ converges to some ${\mu\in [0,+\infty]}$. If ${\mu\in (0,+\infty)}$, then $\psi(w)=\mu w$ is holomorphic and commutes with $\varphi $. Otherwise, either $\psi\equiv0$ or $\psi\equiv\infty$, with $0$ and~$\infty$ being exactly the two fixed points of~$\varphi$.
\end{proof}

%%%%%%%%%%%%%%%%%%%%%%%%%%%%%%%%%%%%%%%%%%%%%%%%%%%%%%%%%%%%%%%%%%%%%%%%%%%%%%%%%%%%%%%%%%%%%%%%%%
%%%%%%%%%%%%%%%%%%%%%%%%%%%%%%%%%%%%%%%%%%%%%%%%%%%%%%%%%%%%%%%%%%%%%%%%%%%%%%%%%%%%%%%%%%%%%%%%%%
%%%%%%%%%%%%%%%%%%%%%%%%%%%%%%%%%%%%%%%%%%%%%%%%%%%%%%%%%%%%%%%%%%%%%%%%%%%%%%%%%%%%%%%%%%%%%%%%%%
%%%%%%%%%%%%%%%%%%%%%%%%%%%%%%%%%%%%%%%%%%%%%%%%%%%%%%%%%%%%%%%%%%%%%%%%%%%%%%%%%%%%%%%%%%%%%%%%%%
%%%%%%%%%%%%%%%%%%%%%%%%%%%%%%%%%%%%%%%%%%%%%%%%%%%%%%%%%%%%%%%%%%%%%%%%%%%%%%%%%%%%%%%%%%%%%%%%%%
%%%%%%%%%%%%%%%%%%%%%%%%%%%%%%%%%%%%%%%%%%%%%%%%%%%%%%%%%%%%%%%%%%%%%%%%%%%%%%%%%%%%%%%%%%%%%%%%%%
\section{Simultaneous solutions to Abel's equation}\label{S_simultaneous}
%%%%%%%%%%%%%%%%%%%%%%%%%%%%%%%%%%%%%%%%%%%%%%%%%%%%%%%%%%%%%%%%%%%%%%%%%%%%%%%%%%%%%%%%%%%%%%%%%%
%%%%%%%%%%%%%%%%%%%%%%%%%%%%%%%%%%%%%%%%%%%%%%%%%%%%%%%%%%%%%%%%%%%%%%%%%%%%%%%%%%%%%%%%%%%%%%%%%%
%%%%%%%%%%%%%%%%%%%%%%%%%%%%%%%%%%%%%%%%%%%%%%%%%%%%%%%%%%%%%%%%%%%%%%%%%%%%%%%%%%%%%%%%%%%%%%%%%%

The next theorem is one of the main technical results of the paper. It plays a crucial role at several points of our study.

\begin{theorem}\label{TH_simultaneous}
Suppose $\varphi\in\U(\UD)\setminus\{\id_\UD\}$ is non-elliptic. Let $\Delta$ be any abelian subset of $\Zen(\varphi)$, i.e. ${\psi_1\circ\psi_2}={\psi_2\circ\psi_1}$ for all ${\psi_1,\psi_2\in\Delta}$. Then there exists a univalent map $h_{\varphi,\Delta}:\UD\to\C$ with the following property: to each ${\psi \in \Delta}$ there corresponds a constant  $c_{\varphi,\psi}\in\C$ such that
\begin{equation}\label{EQ_simultaneous}
	h_{\varphi,\Delta}\circ\varphi=h_{\varphi,\Delta}+1\quad\text{and}\quad
	h_{\varphi,\Delta}\circ\psi=h_{\varphi,\Delta}+c_{\varphi,\psi}.
\end{equation}
\end{theorem}

In the most complicated and interesting case of parabolic self-maps~$\varphi$
having positive hyperbolic step, Theorem~\ref{TH_simultaneous} up to our best knowledge is completely new, although related results for a \textit{fixed pair} of commuting self-maps have been previously known, see e.g. \cite[Theorem~3.1]{Cowen-comm}, \cite[Theorem~2.1]{Vlacci}, \cite[Theorem~6]{GB}, \cite[Lemma~3]{conReich} (see also \cite[p.\,146]{EliShobook10}). The dependence of the solutions to~\eqref{EQ_simultaneous} on the choice of~$\Delta$ is essential, as the whole centralizer $\Zen(\varphi)$ is not necessarily abelian in this case; see Examples~\ref{EX_Z-non-abelian} and~\ref{EX_again-non-abelian}.

In the context of holomorphic (not necessarily injective) self-maps of the unit ball in~$\C^n$, a strong and very general result by Arosio and Bracci \cite[Theorem~1.1]{Simultaneous} treats finite abelian subsets of the centralizer, but it does not imply Theorem~\ref{TH_simultaneous} either, because a parabolic self-map ${\varphi\in\U(\UD)}$ of positive hyperbolic step can commute with a parabolic self-map ${\psi\in\U(\UD)}$ having zero hyperbolic step, see Sect.\,\ref{S_examples} for examples. In such a case, the canonical Kobayashi hyperbolic semi-model for the pair ${(\varphi,\psi)}$ provided by \cite[Theorem~1.1]{Simultaneous} is trivial, i.e. its dimension vanishes.

Later in the paper, we will be able to identify the value of the constant $c_{\varphi,\psi}$ that appears in~\eqref{EQ_simultaneous}. Namely, given a non-elliptic $\varphi\in\U(\UD)$ and $\psi\in\Zen(\varphi)$, the constant $c_{\varphi,\psi}$ for which the system of functional equations
\begin{equation}\label{EQ_two-Abel-in-remark}
 h_*\circ \varphi\,=\, h_*+1,\qquad h_*\circ\psi\,=\,h_*+c_{\varphi,\psi}
\end{equation}
admits a univalent solution $h_*:\UD\to\C$ is unique and it is given by
\begin{equation}\label{EQ_formula-for-c}
c_{\varphi,\psi}=\begin{cases}
 \displaystyle \angle\lim_{z\to\tau}\big(\psi(z)-z\big)/\big(\varphi(z)-z\big),& \text{if $\varphi$ is parabolic,}\\[1.5ex]
 \hphantom{\angle}\frac{\log{\psi^\prime(\tau)}}{\log{\varphi^\prime(\tau)}}, & \text{if $\varphi$ is hyperbolic.}
 \end{cases}
\end{equation}
For parabolic self-maps of positive hyperbolic step, this formula will be deduced in Remark~\ref{Rem:cinPHP}.
For hyperbolic self-maps and for parabolic self-maps of zero hyperbolic step, \eqref{EQ_formula-for-c} will be established in the proofs of Theorems~\ref{TH_isomorphism} and~\ref{Thm:0HS}, respectively; see \eqref{Eq:TH_isomorphism} and \eqref{Eq:0HS}, see also Remark~\ref{RM_c-phi-psi_unique} for the uniqueness in the hyperbolic case.
\medskip

The proof of Theorem~\ref{TH_simultaneous} for parabolic self-maps of positive hyperbolic step is given in Sect.\,\ref{S_para-positive}. More precisely, we will prove the following proposition.
\begin{proposition}\label{PR_simultaneous-para-positive}
Under conditions of Theorem~\ref{TH_simultaneous}, suppose additionally that $\varphi$ is parabolic of positive hyperbolic step and let ${(S,h,z\mapsto z+1)}$, $S\in\{-\UH,\UH\}$, stand for its canonical holomorphic model. Then
there exists a univalent map $\beta_\Delta:S\to\C$  such that:
		\begin{enumerate}[label={\rm(\roman*)}, ref={\rm(\roman*)}]
			\item ${\beta_\Delta(z+1)}={\beta_\Delta(z)+1}$ for all~${z\in S}$ and\smallskip
			\item $h_{\varphi,\Delta}:=\beta_\Delta\circ h$ satisfies~\eqref{EQ_simultaneous} for each ${\psi\in\Delta}$.
		\end{enumerate}
\end{proposition}

In the remaining two cases, Theorem~\ref{TH_simultaneous} is an immediate corollary of the following statement.

\begin{proposition}\label{PR_abelian}
Let $\varphi\in\U(\UD)$ be hyperbolic or parabolic of zero hyperbolic step and let $h$ stand for the Koenigs function of~$\varphi$. Then a self-map $\psi:\UD\to\UD$ belongs to $\Zen(\varphi)$ if and only if~\eqref{EQ_two-Abel-in-remark} holds with~$h_*:=h$ and a suitable constant ${c_{\varphi,\psi}\in\C}$.
In particular, the centralizer $\Zen(\varphi)$ is abelian.
\end{proposition}
One of the implications in the above proposition is almost a banality. Indeed, if \eqref{EQ_two-Abel-in-remark} admits a univalent solution, then
$$
 \psi=h_*^{-1}\circ\big(h_*+c_{\varphi,\psi}\big),\quad \varphi=h_*^{-1}\circ\big(h_*+1\big),
$$
and it immediately follows that $\psi\in\Zen(\varphi)$. The other implication is rather non-trivial. For parabolic self-maps of zero hyperbolic step, it follows from Theorem~\ref{Thm:general} as explained in Sect.\,\ref{proofTH_simultaneousparaboliczero} on page~\pageref{proofTH_simultaneousparaboliczero}. Note that the univalence of $\varphi$ and~$\psi$ plays in this case an essential role, in contrast to the hyperbolic case, in which Proposition~\ref{PR_abelian}, and hence Theorem~\ref{TH_simultaneous}, can be deduced from \cite[Theorem~3.1]{Cowen-comm} combined with some details in its proof. Cowen's proof in~\cite{Cowen-comm} makes use of Behan's Theorem~\cite[Theorem~6]{Behan} (see also \cite[Section~4.10]{Abate2}) stating that commuting holomorphic self-maps, aside from the case of hyperbolic automorphisms, share the Denjoy\,--\,Wolff point. In order to clarify why the situation here differs so much from the case of parabolic self-maps of positive hyperbolic step, we give a direct proof of Proposition~\ref{PR_abelian} for hyperbolic self-maps in Sect.\,\ref{S_hyperbolic}.  It is based on a result of Heins~\cite[Lemma~2.1]{Heins}, which is considerably more elementary than Behan's Theorem.

Finally, it is worth mentioning that in the hyperbolic case, the simultaneous solution to Abel's equations, given in Proposition~\ref{PR_abelian} and Theorem~\ref{TH_simultaneous}, has a generalization to several complex variables, see \cite[Theorem~3.11]{Simultaneous}.

A first consequence of the simultaneous solution to Abel's two equations~\eqref{EQ_simultaneous} is the next result that characterizes when the identity map is isolated  in the centralizer.
\begin{theorem}\label{TH_dichotomy-nonelliptic}
Let $\varphi\in\U(\UD)\setminus\{\id_\UD\}$ be non-elliptic. If $~\id_\UD$ is not isolated in~$\mathcal Z(\varphi)$, then $\mathcal Z(\varphi)$ contains a non-trivial continuous one-parameter semigroup $(\phi_t)$.
\end{theorem}
\begin{proof}
The hypothesis of the theorem implies that ${\mathcal Z(\varphi)\setminus\{\id_\UD\}}$ contains a sequence $(\psi_n)$ converging locally uniformly in~$\UD$ to the identity map. By Theorem~\ref{TH_simultaneous}, for each ${n\in\Natural}$ there exists a univalent function $h_n$ in~$\UD$ and a complex number~$c_n\neq0$ (because $\psi_n \neq \id_\D$ for all $n$) such that
\begin{equation}\label{EQ_simult-Abel-equations}
h_n\circ \varphi=h_n+1\quad\text{and}\quad {h_n\circ\psi_n=h_n+c_n}.
\end{equation}
Note that the functions $h_n$ are defined up to an additive constant; hence, we may suppose that $h_n(0)=0$ for all ${n\in\Natural}$. Then ${h_n(\varphi(0))=h_n(0)+1=1}$. As a consequence, by the Growth Theorem, see e.g. \cite[Theorem~2.6 on p.\,33]{Duren}, we have
$$
\frac{(1-|\varphi(0)|)^{2}}{|\varphi(0)|}\leq |h'_{n}(0)|\leq \frac{(1+|\varphi(0)|)^{2}}{|\varphi(0)|},
$$
for all~$n\in\Natural$.
It follows that $h_n$'s form a normal family in~$\UD$. Therefore, passing if necessary to a subsequence we may suppose that $(h_n)$ converges uniformly on compacta in~$\UD$ to some function~$h$. Note that ${h(0)=0}$, ${h(\varphi(0))=1}$. In particular, $h$ is not constant, and as a consequence of Hurwitz Theorem, it is univalent in~$\UD$.

Clearly, ${c_n\to0}$ as ${n\to+\infty}$. Passing to a subsequence we may suppose that also ${\arg c_n\to\theta_0}$ as ${n\to+\infty}$ for some ${\theta_0\in\Real}$. Now fix some ${t>0}$ and write $k(n):=\big\lfloor t/|c_n|\big\rfloor$, to denote the integer part of ${t/|c_n|}$. Since
$$
k(n)|c_n|\leq t<(k(n)+1)|c_n| \quad \textrm{ for all } \ n,
$$
it is easy to see that
\begin{equation}\label{EQ_numbers-converge}
k(n)c_n\to t e^{i\theta_0}\quad\text{as~$~n\to+\infty$}.
\end{equation}
We claim that
\begin{equation}\label{EQ_remains-to-show-inv-Omega}
h(\UD)=:\Omega\supset\Omega + te^{i\theta_0}.
\end{equation}
If the above inclusion holds for all ${t\ge0}$, then the formula $\phi_t(z):=h^{-1}\big(h(z)+te^{i\theta_0}\big)$ defines a non-trivial one-parameter semigroup in~$\UD$ and moreover, $(\phi_t)\subset\mathcal Z(\varphi)$.

So it remains to prove~\eqref{EQ_remains-to-show-inv-Omega}. To this end,
we consider the set
$$
 U:=\bigcup_{m\in\Natural}\interior\Big(\bigcap_{n\ge m}\Omega_n\Big),\quad\Omega_n:=h_n(\UD),
$$
where $\interior(\cdot)$ stands for the interior of a set in~$\C$.  Let us check that
\begin{equation}\label{EQ_U-inv}
U+te^{i\theta_0}\subset U.
\end{equation}
Fix an arbitrary~$w_0\in U$. Since $U$ is open, it contains some closed disc~$B$ centred at~$w_0$. The expanding sequence of sets ${U_m:=\interior(\cap_{n\ge m}\Omega_n)}$ is an open cover of~$B$. Therefore,  by compactness of~$B$, there exists ${m_0\in\Natural}$ such that ${B\subset U_{m_0}}$ and hence, ${B\subset \Omega_n}$ for any ${n\ge m_0}$. Further, let $B_1$ be an open disc centred at~$w_0$ of radius strictly smaller than that of the disc~$B$. Then, according to~\eqref{EQ_numbers-converge}, there exists $m_1\ge m_0$ such that ${\omega_n(w):=w+\big(te^{i\theta_0}-k(n)c_n\big)\in B}$ for any~${n\ge m_1}$ and any~$w\in B_1$. Finally, from~\eqref{EQ_simult-Abel-equations} it follows that for each $n\in\Natural$, ${\Omega_n+ c_n\subset\Omega_n}$. Therefore, since $k(n)\in \N_{0}$, for any ${w\in B_1}$ and all ${n\ge m_1}$,
$$
 w+te^{i\theta_0}~=~\omega_n(w)+k(n)c_n~\in~B+k(n)c_n\subset\Omega_n+k(n)c_n\subset\Omega_n,
$$
i.e. $B_1+te^{i\theta_0}\subset\Omega_n$ for all $n\in\Natural$ large enough. It follows that ${w_0+te^{i\theta_0}\in B_1+te^{i\theta_0}\subset U}$. This proves~\eqref{EQ_U-inv}.

Since $t>0$ in the above argument is arbitrary, \eqref{EQ_U-inv} implies that together with any point ${w_0\in U}$, the set~$U$ contains the whole ray ${R(w_0):=\{w_0+te^{i\theta_0}:t\ge0\}}$. Obviously, $R(w_0)$ is a connected set and ${w_0\in R(w_0)}$. It follows that ${w_0+te^{i\theta_0}}$ and~$w_0$ belong to the same connected component of~$U$. Thus, in order to complete the proof of~\eqref{EQ_remains-to-show-inv-Omega} it remains to notice that by Carath\'eodory's kernel convergence theorem, $\Omega$ is a connected component of~$U$.
\end{proof}

\begin{remark}
Clearly, if in the above proof $\theta_0=0$, then $\phi_1=\varphi$ and hence $\varphi$ is embeddable. At the same time, if $0$ is not in the limit set of $(\arg c_n)$, then $\varphi$ does not have to be embeddable, even though $\Zen(\varphi)$ contains a non-trivial one-parameter semigroup, see e.g. Example~\ref{EX_non-non}.
\end{remark}

%%%%%%%%%%%%%%%%%%%%%%%%%%%%%%%%%%%%%%%%%%%%%%%%%%%%%%%%%%%%%%%%%%%%%%%%%%%%%%%%%%%%%%%%%%%%%%%%%%
%%%%%%%%%%%%%%%%%%%%%%%%%%%%%%%%%%%%%%%%%%%%%%%%%%%%%%%%%%%%%%%%%%%%%%%%%%%%%%%%%%%%%%%%%%%%%%%%%%
%%%%%%%%%%%%%%%%%%%%%%%%%%%%%%%%%%%%%%%%%%%%%%%%%%%%%%%%%%%%%%%%%%%%%%%%%%%%%%%%%%%%%%%%%%%%%%%%%%
%%%%%%%%%%%%%%%%%%%%%%%%%%%%%%%%%%%%%%%%%%%%%%%%%%%%%%%%%%%%%%%%%%%%%%%%%%%%%%%%%%%%%%%%%%%%%%%%%%
%%%%%%%%%%%%%%%%%%%%%%%%%%%%%%%%%%%%%%%%%%%%%%%%%%%%%%%%%%%%%%%%%%%%%%%%%%%%%%%%%%%%%%%%%%%%%%%%%%
%%%%%%%%%%%%%%%%%%%%%%%%%%%%%%%%%%%%%%%%%%%%%%%%%%%%%%%%%%%%%%%%%%%%%%%%%%%%%%%%%%%%%%%%%%%%%%%%%%
\section{Parabolic self-maps of zero hyperbolic step}\label{S_para-zero}
%%%%%%%%%%%%%%%%%%%%%%%%%%%%%%%%%%%%%%%%%%%%%%%%%%%%%%%%%%%%%%%%%%%%%%%%%%%%%%%%%%%%%%%%%%%%%%%%%%
%%%%%%%%%%%%%%%%%%%%%%%%%%%%%%%%%%%%%%%%%%%%%%%%%%%%%%%%%%%%%%%%%%%%%%%%%%%%%%%%%%%%%%%%%%%%%%%%%%
%%%%%%%%%%%%%%%%%%%%%%%%%%%%%%%%%%%%%%%%%%%%%%%%%%%%%%%%%%%%%%%%%%%%%%%%%%%%%%%%%%%%%%%%%%%%%%%%%%
In case of a parabolic self-map of zero hyperbolic step, Theorem~\ref{Thm:general} implies simultaneous solvability of Abel's equations as follows.
\begin{proof}[\proofof{Proposition~\ref{PR_abelian} for parabolic self-maps of zero hyperbolic step}] \label{proofTH_simultaneousparaboliczero}
By Theorem~\ref{Thm:model}, the canonical holomorphic model of $\varphi$ is of the form ${(\C,h,z\mapsto z+1)}$. Let $\psi\in\mathcal Z(\varphi)$. Then, by Theorem~\ref{Thm:general}, ${\psi=}{h^{-1}\circ g\circ h}$ for a suitable univalent function ${g:\C\to \C}$ satisfying ${g(w+1)}={g(w)+1}$ for all ${w\in \C}$. Since $g$ is univalent and entire, we have that it is affine, and the commutativity with the translation ${w\mapsto w+1}$ implies that there is ${b\in \C}$ such that ${g(w)=w+b}$ for all ${w\in \C}$.  Therefore, we get
$h(\psi (z))=g(h(z))=h(z)+b$,  i.e. the Koenigs function~$h$ of~$\varphi$ satisfies~\eqref{EQ_two-Abel-in-remark} with $c_{\varphi,\psi}:=b$.
\end{proof}

The above proof shows that if $\varphi\in\U(\UD)$ is a parabolic self-map of zero hyperbolic step, then the map we introduced in Theorem~\ref{PropA_varphi},
$$
  \Tmap_\varphi:\mathcal A_\varphi\to\Zen(\varphi);\quad b\,\mapsto\,\psi_b:=h^{-1}\circ(w\mapsto w+1)\circ h,
$$
is a bijection.  We identify its inverse ${\Smap_\varphi:\Zen(\varphi)\to\mathcal A_\varphi}$ in Theorem~\ref{Thm:0HS} below. Moreover, the method we use yields immediately that for any $\psi\in\Zen(\varphi)$, the number $\Smap_\varphi(\psi)$  is the unique value of~$c_{\varphi,\psi}$ for which the system of Abel's equations~\eqref{EQ_two-Abel-in-remark} admits a univalent solution.

It is known (see \cite[Theorem~5.2]{CDP}) that for any parabolic self-map ${\varphi:\UD\to\UD}$  with Denjoy\,--\,Wolff point $\tau\in \partial \D$ and Koenigs function $h$,
\begin{equation}\label{Eq:CDP}
\angle\lim _{z\to \tau} (\varphi(z)-z)h'(z)=1.
\end{equation}
This result is a key to finding an explicit expression for $\Smap_\varphi=\Tmap_\varphi^{-1}$. In fact, we will need a more general version, which we prove using ideas similar to those  given in \cite[Theorem~5.2]{CDP}.

\begin{proposition}\label{positivo2} Let $\varphi$ be a parabolic self-map of the unit disc (not necessarily univalent) with Denjoy\,--\,Wolff point $\tau\in\partial\D$ and let ${h:\UD\to\C}$ be univalent. Then
	$$
	\angle \lim_{z\to\tau}\dfrac{h(\varphi(z))-h(z)}{h^\prime(z)(\varphi(z)-z)}=1.
	$$
\end{proposition}
\begin{proof} Consider $\Phi:=C\circ\varphi\circ C^{-1}$, where $C$ is the usual Cayley map with a pole at~$\tau$; i.e. $C(z)=\frac{\tau+z}{\tau-z},\ z\in\D$. Then $\Phi$ is a parabolic self-map of the right half-plane $\H_r$ with Denjoy\,--\,Wolff point at~$\infty$. In particular,
\begin{equation}\label{Eq:positivo2}
\angle\lim_{w\to\infty}\frac{\Phi(w)}w=1.
\end{equation}
	
	Fix $\xi\in\H_r$ and denote $z:=C^{-1}(\xi)$ as well as $H:=h\circ C^{-1}$. Then, $H$ is a univalent function in~$\UH_r$ and ${h(\varphi(z))-h(z)}={H(\Phi(\xi))-H(\xi)}$. Moreover,
	$$
	h^\prime(z)=H^\prime(C(z))C^\prime(z)=H^\prime(\xi)\frac{2\tau}{(\tau-C^{-1}(\xi))^2}= H^\prime(\xi)\frac{(1+\xi)^2}{2\tau}.
	$$
	On the other hand,
	$$
	\varphi(z)-z=C^{-1}(\Phi(\xi))-C^{-1}(\xi)=2\tau\frac{\Phi(\xi)-\xi}
	{(1+\xi)(1+\Phi(\xi))}.
	$$
	Hence,
	$$
	\frac{h(\varphi(z))-h(z)}{h^\prime(z)(\varphi(z)-z)}= \frac{H(\Phi(\xi))-H(\xi)}{H^\prime(\xi)(\Phi(\xi)-\xi)}\,\frac{1+\Phi(\xi)}{1+\xi}.
	$$

\medskip
\noindent	Therefore, it is enough to check that
	$$	\angle\lim_{\xi\to\infty}\frac{H(\Phi(\xi))-H(\xi)}{H^\prime(\xi)(\Phi(\xi)-\xi)}=1.
	$$
	Consider the function
	$$
	g_\xi(w):=\frac{H(\xi+w\Re\xi)-H(\xi)}{H^\prime(\xi)\Re\xi},\quad  w\in \D.
	$$
	Note that $g\in\U(\D,\C)$ and  $g_\xi(0)=g_\xi^\prime(0)-1=0$. Moreover, we claim that
\begin{equation}\label{EQ_claim}
	|g_\xi(w)-w|\leq 54|w|^2\quad  \text{when~$~|w|<1/2$.}
\end{equation}

Assume for a moment that~\eqref{EQ_claim} holds and consider
	 an arbitrary sequence $(\xi_n)\subset\H_r$ converging non-tangentially to $\infty$. Denote ${w_n:=(\Phi(\xi_n)-\xi_n)/\Re \xi_n}$. By \eqref{Eq:positivo2},
\begin{equation}\label{EQ_to0}
	|w_n|=\left|\frac{\Phi(\xi_n)}{\xi_n}-1\right|\,\cdot\,\left|\frac{\xi_n}
	{\Re \xi_n}\right|
		~\longrightarrow~0\quad\text{as~$~n\to+\infty$}.
\end{equation}
	Therefore, omitting a finite number of terms, we may assume that  ${|w_{n}|<1/2}$ for all ${n\geq1}$.
	Using~\eqref{EQ_claim} and~\eqref{EQ_to0}, we therefore obtain
	$$ \lim_{n\to+\infty}\frac{H(\Phi(\xi_n))-H(\xi_n)}{H^\prime(\xi_n)(\Phi(\xi_n)-\xi_n)}=\lim_{n\to+\infty}\frac{g_{\xi_n}(w_n)}{w_n}=1.
	$$

It remains to prove~\eqref{EQ_claim}. Fix an arbitrary~$w$ with $|w|<1/2$. Using Cauchy's integral formula with the contour $\Gamma:=\{z:|z|=2/3\}$ oriented counterclockwise and the Growth Theorem, see e.g. \cite[Theorem~2.6 on p.\,33]{Duren}, we find that
\begin{align*}
	\big|g_{\xi}(w)-w\big| &=\big|g_{\xi}(w)-g_{\xi}(0)-g_{\xi}^{\prime }(0)w\big|= \\
		&=\displaystyle\left\vert \frac{1}{2\pi i}\int_\Gamma g_{\xi}(z)\Big( \frac{1
		}{z-w}-\frac{1}{z }-\frac{w}{z^{2}}\Big) \di z \right\vert  \\
		&=\displaystyle\left\vert \frac{1}{2\pi i}\int_\Gamma g_{\xi}(z)\frac{w^{2}}{%
			z^{2}(z-w)}\di z \right\vert  \\
		&\leq \displaystyle\frac{1}{2\pi }\int_\Gamma\frac{|z|}{(1-|z|)^{2}}\,
		\frac{|w|^{2}}{|z|^{2}\,|z-w|}\,|\di z |\leq 54|w|^{2}. \qedhere
\end{align*}	
\end{proof}

\begin{theorem} \label{Thm:0HS} Let $\varphi\in\U(\UD)$ be parabolic self-map of zero hyperbolic step. Then for any $\psi\in \Zen(\varphi)$ there exists finite angular limit
	$$
	 \Smap_{\varphi}(\psi):=\angle\lim_{z\to\tau}\frac{\psi(z)-z}{\varphi(z)-z},
	$$
where $\tau$ is the Denjoy\,--\,Wolff point of~$\varphi$. The map $\psi\mapsto\Smap_{\varphi}(\psi)$ is an isomorphism of the topological semigroup $[\Zen(\varphi),\circ]$ onto  $[\mathcal A_{\varphi},+]$. Moreover, ${\Smap_\varphi^{-1}=\Tmap_\varphi}$.
\end{theorem}
\begin{proof}
Let $h$ be the Koenigs function of $\varphi$  and  let  $\psi\in {\mathcal Z(\varphi)\setminus\{\id_{\D}}\}$. By Proposition~\ref{PR_abelian}, there is $b\in \C$, $b\neq 0$, such that $h\circ \psi=h+b$. In the notation introduced in Theorem~\ref{PropA_varphi}, this means that $b\in \mathcal A_{\varphi}$ and $\psi=\psi_{b}=\Tmap_\varphi (b)$. As a consequence, the injective map $\Tmap_{\varphi}:\mathcal {A}_{\varphi}\to \Zen(\varphi)$ is, in fact, bijective.
Now, applying Proposition~\ref{positivo2} twice:  to~$\varphi$ and to~$\psi$, we deduce that
$$
\angle\lim _{z\to \tau} (\varphi(z)-z)h'(z)=1\quad\text{and}\quad\angle\lim _{z\to \tau} (\psi(z)-z)\frac{h'(z)}{b}=1.
$$
Therefore,
\begin{equation}\label{Eq:0HS}
b=\angle \lim_{z\to \tau}\frac{\psi(z)-z}{\varphi(z)-z},
\end{equation}
i.e.  $\Smap_{\varphi}(\psi):=\angle \lim_{z\to \tau}\frac{\psi(z)-z}{\varphi(z)-z}~$ is the inverse of $\Tmap_\varphi$.

It remains to recall that by Theorem~\ref{PropA_varphi}, the map $\Tmap_\varphi$ is continuous and closed, and moreover, it is a semigroup homomorphism. Thus, $\Tmap_\varphi$ and it inverse~$\Smap_\varphi$ are homeomorphic isomorphisms between the topological semigroups ${[\Zen(\varphi),\circ]}$ and ${[\mathcal A_\varphi,+]}$.
\end{proof}

\begin{remark}\label{RM_0HS}
 Theorem~\ref{Thm:0HS} implies the following characterization of the centralizers of a parabolic self-maps with zero hyperbolic step via their canonical models ${(\C,h,z\mapsto z+1)}$:
 $$
   \Zen(\varphi)=\big\{h^{-1}\circ (z\mapsto z+c)\circ h: c\in\C\setminus(-\infty,0),~ \Omega+c\subset\Omega\big\},\quad \Omega:=h(\UD).
 $$
(Here we took into account that according to Theorem~\ref{PropA_varphi}, if ${c\in(-\infty,0)}$, then $\Omega+c\not\subset\Omega$.)

A very similar characterization holds for hyperbolic self-maps, see Remark~\ref{RM_hyperb-character-in-term-of-the-model}.
\end{remark}

Concluding this section we prove our dichotomy relating embeddability with the local structure of the centralizer near the identity (for the case of a parabolic self-map of zero hyperbolic step). As usual, for ${w\in \C^*:=\C\setminus \{0\}}$ we denote by $\mathrm{Arg}(w)$ the principal argument, i.e. the value of the argument of~$w$ contained in~$(-\pi,\pi]$. In the notation introduced in Theorem~\ref{PropA_varphi}, the concluding result of this section can be stated as follows.

\begin{theorem}\label{PR_emb-dichotomy-P0HS} Let $\varphi\in\U(\UD)$ be parabolic of zero hyperbolic step. Exactly one of the following alternatives holds:
	\begin{itemize}
		\item[\rm (i)] either there exist $\rho>0$ and $\delta>0$ such that
		$$\mathcal A_{\varphi }\cap\{c\in\C: 0<|c|<\rho, |\mathrm{Arg}(c)|<\delta\}=\emptyset,$$
		\item[\rm (ii)] or there exists a unique one-parameter semigroup $(\phi_t)\subset\U(\UD)$ with $\phi_1=\varphi$.
	\end{itemize}
\end{theorem}
\begin{proof} If (ii) holds, then $[0,+\infty)\subset \mathcal{A}_{\varphi}$ and, clearly, (i) does not hold.

Conversely, assume that (i) does not hold.
Then there is a sequence $\{r_{n}\}$ in $(0,+\infty)$ and a sequence $\{\theta_{n}\}$, both converging to zero, such that ${r_{n}e^{i\theta_{n}}\in \mathcal{A}_{\varphi}}$ for all ${n\in\Natural}$. Fix an arbitrary ${t>0}$.
For ${n\in\Natural}$, denote ${m_{n}:=\lfloor t/r_{n}\rfloor}$, the integer part of $t/r_{n}$.
By Theorem~\ref{PropA_varphi}, $\mathcal A_\varphi$ is a subsemigroup of ${[\C,+]}$. Hence, $m_{n}r_{n}e^{i\theta_{n}}\in \mathcal{A}_{\varphi}$ for all ${n\in\Natural}$. Moreover,
 $ m_{n}r_{n}\leq t<m_n r_n+r_{n} $ so that $\lim_{n\to+\infty}m_{n}r_{n}=t$.
It follows that ${t\in \mathcal{A}_{\varphi}}$, because by Theorem~\ref{PropA_varphi}, $\mathcal{A}_{\varphi}$ is a closed  set in~$\C$. Since $t>0$ is arbitrary in our argument, this proves that ${[0,+\infty)\subset\mathcal A_\varphi}$. As a consequence, the formula  ${\phi_{t}:=h^{-1}\circ (z\mapsto z+t)\circ h}$ defines a one-parameter semigroup in~$\UD$ with $\phi_1=\varphi$. The uniqueness of~$(\phi_t)$ is provided by Corollary~\ref{CR_uniquesemigroup}.
\end{proof}

It is worth mentioning that in Section~\ref{S_para-positive}, we obtain one more result valid for parabolic self-maps of zero hyperbolic step (and, at the same time, also for those of positive hyperbolic step); see Corollary~\ref{CR_parabolic-second-der} and remarks after it.

%%%%%%%%%%%%%%%%%%%%%%%%%%%%%%%%%%%%%%%%%%%%%%%%%%%%%%%%%%%%%%%%%%%%%%%%%%%%%%%%%%%%%%%%%%%%%%%%%%
%%%%%%%%%%%%%%%%%%%%%%%%%%%%%%%%%%%%%%%%%%%%%%%%%%%%%%%%%%%%%%%%%%%%%%%%%%%%%%%%%%%%%%%%%%%%%%%%%%
%%%%%%%%%%%%%%%%%%%%%%%%%%%%%%%%%%%%%%%%%%%%%%%%%%%%%%%%%%%%%%%%%%%%%%%%%%%%%%%%%%%%%%%%%%%%%%%%%%
%%%%%%%%%%%%%%%%%%%%%%%%%%%%%%%%%%%%%%%%%%%%%%%%%%%%%%%%%%%%%%%%%%%%%%%%%%%%%%%%%%%%%%%%%%%%%%%%%%
%%%%%%%%%%%%%%%%%%%%%%%%%%%%%%%%%%%%%%%%%%%%%%%%%%%%%%%%%%%%%%%%%%%%%%%%%%%%%%%%%%%%%%%%%%%%%%%%%%
%%%%%%%%%%%%%%%%%%%%%%%%%%%%%%%%%%%%%%%%%%%%%%%%%%%%%%%%%%%%%%%%%%%%%%%%%%%%%%%%%%%%%%%%%%%%%%%%%%
\section{Hyperbolic self-maps}\label{S_hyperbolic}
%%%%%%%%%%%%%%%%%%%%%%%%%%%%%%%%%%%%%%%%%%%%%%%%%%%%%%%%%%%%%%%%%%%%%%%%%%%%%%%%%%%%%%%%%%%%%%%%%%
%%%%%%%%%%%%%%%%%%%%%%%%%%%%%%%%%%%%%%%%%%%%%%%%%%%%%%%%%%%%%%%%%%%%%%%%%%%%%%%%%%%%%%%%%%%%%%%%%%
%%%%%%%%%%%%%%%%%%%%%%%%%%%%%%%%%%%%%%%%%%%%%%%%%%%%%%%%%%%%%%%%%%%%%%%%%%%%%%%%%%%%%%%%%%%%%%%%%%

In this section we study the centralizers of hyperbolic self-maps. Although it is not explicitly stated in~\cite{Cowen-comm}, combining \cite[Theorem~3.1]{Cowen-comm} with some part of its proof, it is possible to deduce Theorem~\ref{TH_simultaneous} for hyperbolic self-maps. A similar remark concerns \cite{Vlacci} and~\cite{GB}. Moreover, Theorem~\ref{TH_simultaneous} restricted to the hyperbolic case follows also from \cite[Theorem~3.11]{Simultaneous}. Nevertheless, for the sake of completeness, below we provide a short direct proof (of Proposition~\ref{PR_abelian}, which implies Theorem~\ref{TH_simultaneous}) also for the hyperbolic case.

\begin{proof}[\proofof{Proposition~\ref{PR_abelian} for hyperbolic self-maps}] \label{proofTH_simultaneoushyp}
If ${(S,h,z\mapsto z+1)}$ is the canonical model for~$\varphi$, then  by Theorem~\ref{Thm:general} every ${\psi\in\Zen(\varphi)}$ is of the form $\psi=h^{-1}\circ g\circ h$, where $g\in\Hol(S)$ satisfies the functional equation ${g(w+1)=g(w)+1}$, ${w\in S}$. Since $\varphi$ is hyperbolic, $S$ is a horizontal strip. Hence, according to \cite[Lemma~2.1]{Heins}, ${g(w)=w+b}$ for all ${w\in S}$ and some constant ${b\in\Real}$. Then we are done.
\end{proof}

Before moving further, we make several remarks.

\begin{remark}\label{RM_h-varphi-psi-not-unique}
Unlike the case of a parabolic self-map of zero hyperbolic step, in the hyperbolic case, not every univalent solution to~\eqref{EQ_two-Abel-in-remark}  is of the form $h_*=h+c_0$, where $c_0\in\C$ is a constant and $h$ is the Koenigs map of~$\varphi$. For example, take a hyperbolic univalent self-map $\varphi$ such that $1/2\in \mathcal A_{\varphi}$ and consider  $\psi:=h^{-1}\circ(z\mapsto z+1/2)\circ h$, where $h$ stands, as before, for the Koenigs function of~$\varphi$. Clearly, $\psi\in\Zen(\varphi)$. Consider a conformal mapping $F$ of the base space $S$ of the canonical model ${(S,h,z\mapsto z+1)}$ for~$\varphi$. Assume that $F$ satisfies the functional equation ${F(z+1/2)=F(z)+1/2}$, ${z\in S}$, and that $F$ is different from a translation. For $\kappa$ small enough, ${F(z):=z+\kappa \sin (4\pi z)}$ is an example of such a mapping. (The univalence of~$F$ can be established using the Noshiro\,--\,Warschawski criterion.) It is immediate that both $h$ and $F\circ h$ solve~\eqref{EQ_two-Abel-in-remark} with ${c_{\varphi,\psi}=1/2}$.

At the same time, as the above proof of Proposition~\ref{PR_abelian} for the hyperbolic case shows, ${c_{\varphi,\psi}\in\Real}$. If $c_{\varphi,\psi}$ is irrational, then every holomorphic solution to~\eqref{EQ_two-Abel-in-remark} coincides with the Koenigs map~$h$ up to an additive constant. Indeed, by Remark~\ref{RM_factorization}, every solution to the first equation in~\eqref{EQ_two-Abel-in-remark} factorizes via the canonical model of~$\varphi$, i.e. $h_*=\eta\circ h$ for some holomorphic map of~$S$ satisfying ${\eta(z+1)=\eta(z)+1}$ for all ${z\in S}$. Moreover, both $h$ and $h_*$ satisfy the second equation in~\eqref{EQ_two-Abel-in-remark}. Hence, we also have ${\eta(z+c_{\varphi,\psi})=\eta(z)+c_{\varphi,\psi}}$  for all ${z\in S}$. As a consequence, ${\eta\circ(z\mapsto z+c)}={(z\mapsto z+c)\circ\eta}$ for all~$c$ belonging to the set $E:={\{k_1\,+\,k_2\,c_{\varphi,\psi}:k_1,k_2\in\Z\}}.$
Clearly, $E$ is dense in~$\Real$ (because irrational rotations have dense orbits).
With the help of the Identity Theorem for holomorphic functions, it follow that $\eta$ is a translation, as claimed.
\end{remark}

\begin{remark}\label{RM_c-phi-psi_unique}
Taking into account the previous remark, the uniqueness of the constant $c_{\varphi,\psi}$ for which the system of functional equations~\eqref{EQ_two-Abel-in-remark} admits a holomorphic solution requires some words of justification. Suppose that ${h_j\circ\varphi=h_j+1}$ and ${h_j\circ\psi=h_j+c_j}$, ${j=1,2}$, where ${h_1=h}$ is the Koenigs function of~$\varphi$ and ${h_2:\UD\to\Complex}$ is another holomorphic function. Then by the same argument as in Remark~\ref{RM_h-varphi-psi-not-unique}, we have that ${h_2=\eta\circ h_1}$ for some holomorphic function in~$S$ satisfying ${\eta(z+1)=\eta(z)+1}$ and ${\eta(z+c_1)=\eta(z)+c_2}$ for all ${z\in S}$.
These two identities can hold simultaneously only if ${c_2=c_1}$. Indeed, fix a straight line $L\subset S$ parallel to~$\R$. By the former identity, ${\mu(z):=\eta(z)-z}$ is periodic and hence, it is bounded on~$L$, but the latter identity implies that $\mu$ is unbounded on~$L$ unless ${c_2=c_1}$.
\end{remark}

Recall that in \thprop{} we introduced the homomorphic map $\Tmap_{\varphi}$ of the topologically closed subsemigroup $$\mathcal A_\varphi:= {\{b\in\Complex:\Omega+b\}}~\subset~[\C,+],$$
where $\Omega$ is the Koenigs domain of~$\varphi$, into the centralizer $\Zen(\varphi)$. In the previous section we have proved, see Theorem~\ref{Thm:0HS}, that for  parabolic self-maps of zero hyperbolic step, this map is a semigroup isomorphism. We will now show that this is the case for hyperbolic self-maps as well. We will also identify the inverse of $\Tmap_{\varphi}$.

\begin{theorem}\label{TH_isomorphism}
Let $\varphi\in\U(\UD)$ be a hyperbolic self-map with Denjoy\,--\,Wolff point ${\tau\in\partial\UD}$. Then the map $\Tmap_\varphi$ is an isomorphism of the topological semigroup ${[\mathcal A_\varphi,+]}$ onto $[\Zen(\varphi),\circ]$. Moreover, the inverse ${\Smap_\varphi:\Zen(\varphi)\to\mathcal A_\varphi}\,$ of $~\Tmap_\varphi$ is given by
	\begin{equation}\label{Eq:TH_isomorphism}
	\psi\in \Zen(\varphi)\mapsto \Smap_{\varphi}(\psi):=\frac{\log{\psi^\prime(\tau)}}{\log{\varphi^\prime(\tau)}}.
	\end{equation}
Furthermore, ${\mathcal A_\varphi=\Real}$ if ${\varphi\in\Aut}$, and ${\mathcal A_\varphi\subset[0,+\infty)}$ otherwise.
\end{theorem}

\begin{proof} Let $h$ be the Koenigs function of $\varphi$ and let ${\psi\in\Zen(\varphi)\setminus\{\id_\UD\}}$.  By Proposition~\ref{PR_abelian}, there is a constant ${c_{\varphi,\psi}\neq 0}$ such that
\begin{equation}\label{EQ_Abel-for-psi}
h\circ \psi=h+c_{\varphi,\psi}.
\end{equation}
Moreover, from the proof of Proposition~\ref{PR_abelian} (for hyperbolic self-maps) given at the beginning of this section, it is evident that $c_{\varphi,\psi}\in\Real$.
Recall that the map ${\Tmap_{\varphi}:\mathcal A_{\varphi} \to \mathcal Z(\varphi)}$ is given by  ${\Tmap_{\varphi}(b)=\psi_{b}}:=h^{-1}\circ(h+b)$.
The functional equation~\eqref{EQ_Abel-for-psi} implies that ${c_{\varphi,\psi}\in \mathcal A_{\varphi}}$ and that ${\psi=\Tmap_{\varphi}(c_{\varphi,\psi})}$. Hence the map~$\Tmap_\varphi$ is surjective. Combining this with Theorem~\ref{PropA_varphi}, we easily conclude that $\Tmap_\varphi$ is an isomorphism of topological semigroups.

If $\varphi$ is an automorphism, then clearly, ${z\mapsto z+1}$ is an automorphism of~$\Omega:=h(\UD)$. As a consequence, the absorption property~(HM2) in Definition~\ref{DF_holomorphic-model} implies that ${\Omega=S}$. In particular, it immediately follows that $\mathcal A_\varphi=\Real$. It further follows that ${h(z)=a\,\log f(z)\,+\,ib}$ for all ${z\in\UD}$, where $f$ is a conformal mapping of~$\UD$ onto~$\H_r$ with ${f(\tau)=\infty}$ and ${a>0}$, ${b\in\Real}$ are suitable constants. Taking into account~\eqref{EQ_Abel-for-psi}, the proof of~\eqref{Eq:TH_isomorphism} in this case is elementary and hence we omit it.

Suppose that $\varphi\not\in\Aut$. Then by Theorem~\ref{PropA_varphi}\,\ref{PropA_notauto} and Remark~\ref{RM_A_varphi}, ${\mathcal A_\varphi\subset[0,+\infty)}$. In view of the formula for the width of the strip $S_I$ in Theorem~\ref{Thm:model}, to prove~\eqref{Eq:TH_isomorphism}  it sufficient to see that for any ${\psi\in\Zen(\varphi)\setminus\{\id_\UD\}}$, the triple ${(S_1,h_1,z\mapsto z+1)}$, where $S_1:={\{w/c_{\varphi,\psi}:w\in S\}}$ and $h_1:={h/c_{\varphi,\psi}}$, is the canonical model for~$\psi$.  To this end we only need to check that ${\bigcup_{n\in\Natural} \Omega-n\,c_{\varphi,\psi}}\,=\,S$. In turn, this equality follows from the inclusion ${\Omega\subset S}$ and from Remark~\ref{RM_absorption-continuous}. The proof is now complete.
\end{proof}

\begin{remark}\label{RM_hyperb-character-in-term-of-the-model}
The above theorem implies a description of the centralizer for a hyperbolic self-map $\varphi\in\U(\UD)$ in terms of its canonical model $(S,h,z\mapsto z+1)$. Let  $\Omega:=h(\UD)$. The following statements hold. (Compare with \cite[Corollary~2]{GB}.)
\begin{enumerate}
	\item[(A)] If $\varphi$ is an automorphism, then $\Omega$ is the horizontal strip $S$ and
	$$
	\Zen(\varphi)=\big\{h^{-1}\circ (z\mapsto z+c)\circ h: c\in\R\big\};
	$$
	\item[(B)] If $\varphi$ is not an automorphism, then
	$$
	\Zen(\varphi)=\big\{h^{-1}\circ (z\mapsto z+c)\circ h: c\in[0,+\infty),~ \Omega+c\subset\Omega\big\}.
	$$
\end{enumerate}
\end{remark}

\begin{remark}
Theorem~\ref{TH_isomorphism} implies the following rigidity result due to Bracci, Tauraso and Vlacci \cite[Theorem~2.4\,\,(2)\,\,and\,\,(5)]{BTV} in the special case of \textit{univalent} self-maps: if $\varphi\in\U(\UD)$ is hyperbolic with Denjoy\,--\,Wolff point~$\tau$ and if $\psi\in\Zen(\varphi)$ satisfies $\psi'(\tau)=\varphi'(\tau)$, then ${\psi\equiv\varphi}$.
\end{remark}

Yet another consequence of the above results is  that any two commuting univalent hyperbolic self-maps share the centralizer.

\begin{proposition} \label{Prop:equalcentralizer} Let $\varphi\in\U(\UD)$ be hyperbolic and let $\psi\in\Zen(\varphi)$, both of them different from~$\id_\UD$. Then $\Zen(\psi)=\Zen(\varphi)$.
\end{proposition}
\begin{proof}
As we have seen in the proof of Theorem~\ref{TH_isomorphism},  if $(S,h,z\mapsto z+1)$ is the canonical model for~$\varphi$ and ${b:=\Smap_{\varphi}(\psi)}$,  then  $(S_1,h_1,z\mapsto z+1)$, where $S_1:=b^{-1}S$ and $h_1:=b^{-1}h$, is the canonical model for~$\psi$. (According to Theorem~\ref{TH_isomorphism}, ${b\neq0}$ because ${\psi\neq\id_\UD}$.)

Since $\varphi$ is hyperbolic, $S$ is a horizontal strip and hence, so is $S_1$. Thus, $\psi$ is also hyperbolic. As a consequence, by Proposition~\ref{PR_abelian}, for any ${\phi\in \mathcal Z(\psi)}$ there is $c\in \R$ such that
$$
\phi=h_1^{-1}\circ (z\mapsto z+c)\circ h_1= h^{-1}\circ (z\mapsto z+bc)\circ h.
$$
It follows that $\phi\in \mathcal Z(\varphi)$. So we may conclude that $\Zen(\psi)\subset\Zen(\varphi)$. Interchanging $\psi$ and $\varphi$ in this above argument, we see that also $\Zen(\varphi)\subset\Zen(\psi)$.
\end{proof}
\begin{remark} The proof of above proposition shows that if $\varphi$ and $\psi$ are hyperbolic self-maps that commute, both different from~$\id_\UD$, then $\mathcal A_\varphi=\Smap_\varphi(\Zen(\varphi))\subset\Real$ is homothetic to~$\mathcal A_\psi$. More precisely,
$$
 \mathcal A_\varphi=\Smap_\varphi(\psi)\mathcal A_\psi.
$$
\end{remark}

\begin{remark}
For parabolic self-maps, the conclusion of Proposition~\ref{Prop:equalcentralizer} may fail, see e.g. Examples~\ref{EX_parab-autom} and~\ref{EX_Z-non-abelian}. In both examples, the semigroup ${[\Zen(\varphi),\circ]}$ is not commutative. This means that there exist ${\psi_1,\psi_2\in\Zen(\varphi)}$ such that $\psi_1\not\in\Zen(\psi_2)$.  Note also that in the pairs of commuting parabolic self-maps $\varphi,\psi\in\U(\UD)$ with $\Zen(\varphi)\neq\Zen(\psi)$ that arise in these examples, at least one of the self-maps is of positive hyperbolic step. If $\varphi,\psi\in{\U(\UD)\setminus\{\id_\UD\}}$ are commuting parabolic self-maps and \textit{both} are of zero hyperbolic step, then it follows from Theorem~\ref{TH_simultaneous} combined with the essential uniqueness of univalent solutions to Abel's equation (which in turn can be established easily by examining the argument of Remark~\ref{RM_factorization}; see also \cite[Sect.\,3]{CDP-Abel}) that if $h$ is the Koenigs function of~$\varphi$, then $\Smap_\varphi(\psi)^{-1}h$ is the Koenigs function of~$\psi$. In this situation, Theorem~\ref{Thm:0HS} implies, as explained in Remark~\ref{RM_0HS}, that ${\Zen(\psi)=\Zen(\varphi)}$ and ${\mathcal A_\varphi=\Smap_\varphi(\psi)\mathcal A_\psi}$.
\end{remark}

As the next result shows, for a hyperbolic univalent self-map $\varphi$, embeddability is equivalent to $\id_\UD$ being an accumulation point of the centralizer, and in this aspect, the situation is simpler than in the parabolic case; compare with Theorems~\ref{TH_dichotomy-nonelliptic} and~\ref{PR_emb-dichotomy-P0HS}.

\begin{proposition}\label{Pro:hyp-dichotomy}
	Let $\varphi\in\U(\UD)$ be hyperbolic.
	\begin{Ourlist}
		\item\label{IT_PR-hyper-auto} If $\varphi$ is an automorphism, then there exists a unique continuous one-parameter group $(\phi_t)_{t\in\Real}\subset\U(\UD)$ with $\phi_1=\varphi$ and, in this case, $\Zen(\varphi)=\{\phi_t:t\in\R\}$.

\medskip

		\item\label{IT_PR-hyper-non-auto} If $\varphi$ is not an automorphism, then exactly one of the following alternatives hold:\smallskip
		\begin{itemize}
			\item[{\rm (i)}] either $\id_\UD$ is an isolated point of~$\Zen(\varphi)$ and, in this case, $\varphi$ is not embeddable,
			\item[{\rm (ii)}] or there exists a unique continuous one-parameter semigroup $(\phi_t)\subset\U(\UD)$ with $\phi_1=\varphi$ and, in this case $\Zen(\varphi)=\{\phi_t:t\ge0\}$.
		\end{itemize}
	\end{Ourlist}
\end{proposition}

\begin{proof} First of all, recall that by Corollary~\ref{CR_uniquesemigroup} there can exist \textit{at most one} continuous one-paremeter (semi)\,group ${(\phi_t)\subset\U(\UD)}$ satisfying ${\phi_1=\varphi}$.\medskip

Further, denote by $h$ the Koenigs function of~$\varphi$.

\StepP{\ref{IT_PR-hyper-auto}} If $\varphi$ is a hyperbolic automorphism, then by Theorem~\ref{TH_isomorphism}, $\Zen(\varphi)=\Tmap_\varphi(\Real)$. It is easy to see that the functions $\phi_{t}:=\Tmap_\varphi(t)={h^{-1}\circ (z\mapsto z+t)\circ h}$ form a one-parameter group ${(\phi_t)_{t\in\Real}\subset\U(\UD)}$ and ${\phi_{1}=\varphi}$.

\StepP{\ref{IT_PR-hyper-non-auto}} If $\phi_1=\varphi$ for a continuous one-parameter semigroup~$(\phi_t)$, then $\{\phi_t:t\ge0\}\subset\mathcal Z(\varphi)$.
Taking into account that $\phi_t\to\id_\UD$ in $\U(\UD)$ as $t\to0^+$, we easily see that
if $\id_\UD$ is an isolated point of~$\mathcal Z(\varphi)$, then the alternative~(i) must hold.

Suppose now that $\id_\UD$ is not an isolated point of~$\mathcal Z(\varphi)$. By Theorem~\ref{TH_isomorphism},
$\Zen(\varphi)$ is the image of $\mathcal A_\varphi\subset[0,+\infty)$ w.r.t. the homeomorphic map
$$
 \Tmap_\varphi:\mathcal A_\varphi\,\ni\,b~\mapsto~ h^{-1}\circ(z\mapsto z+b)\circ h\,\in\,\U(\UD).
$$
Therefore, ${0=\Tmap^{-1}_\varphi(\id_\UD)}$ is an accumulation point of~$\mathcal A_{\varphi}$, i.e. $\mathcal A_\varphi$ contains a sequence $(b_n)\subset(0,+\infty)$ converging to~$0$.
Moreover, by Theorem~\ref{PropA_varphi}, $\mathcal A_\varphi$ is closed w.r.t. taking sums and also topologically closed as subset of~$\C$.
It follows that ${\mathcal A_\varphi=[0,+\infty)}$ because it contains the set
$E:={\{kb_n:k\in\Natural_0,\,n\in\Natural\}
}$, which is dense in ${[0,+\infty)}$. Indeed,
for each $\varepsilon>0$, $E$ contains an $\varepsilon$-net for ${[0,+\infty)}$: simply choose ${n_\varepsilon\in\Natural}$ such that $b_{n_\varepsilon}\in{(0,\varepsilon)}$ and observe that ${\{kb_{n_\varepsilon}:k\in\Natural_0\}}{\subset E}$.
Thus, $\Zen(\varphi)={\{\phi_t:t\ge0\}}$, where $\phi_t:=\Tmap_\varphi(t)$'s form a one-parameter semigroup $(\phi_t)$ in~$\UD$ and clearly, ${\phi_1=\varphi}$.
\end{proof}

Theorems~\ref{PropA_varphi} and~\ref{TH_isomorphism} show that given a hyperbolic self-map ${\varphi\in\U(\UD)\setminus\Aut}$, the set
$\Smap_\varphi\big(\Zen(\varphi)\big)=\,\,\mathcal A_{\varphi}$ is a closed subsemigroup of ${\big[[0,+\infty),+\big]}$. We conclude the section by proving a converse to this fact.

\begin{proposition} \label{PROP_reciprocal}
Let $\mathcal{A}$ be a closed additive subsemigroup of $[0,+\infty)$ that contains $\N_0$. Then there exists  a hyperbolic self-map $\varphi\in{\U(\UD)\setminus\Aut}$ such that ${\mathcal A_{\varphi}=\mathcal{A}}$.
\end{proposition}
\begin{proof}
For each $b\in \mathcal A$, consider $\Gamma_{b}:=\{-b+iy: 0\leq y\leq 1/2\}$ and $$\Omega~:=~\big\{w: \, 0<\Im w<1\big\}\,~\big\backslash~\bigcup\nolimits_{b\in \mathcal A} \Gamma_{b}.$$
Since $\mathcal{A}$ is closed, we have that $\Omega $ is open. Clearly, for any ${c\in \R}$ the inclusion ${\Omega+c\subset \Omega}$ holds if and only if ${(\C\setminus\Omega)-c\subset\C\setminus\Omega}$, which by construction is equivalent to ${\mathcal A+c\subset\mathcal A}$. Since $\mathcal{A}$ is a semigroup containing $0$, the latter inclusion is further equivalent to ${c\in\mathcal A}$. Thus, ${\{c\in\Real:\Omega+c\subset\Omega\}}=\mathcal A$. In particular, it follows that $\Omega+1\subset\Omega$, because by the hypothesis $1\in\mathcal A$.

By the hypothesis $\mathcal A\subset[0,+\infty)$. As a consequence, $\bigcup_{n\in\Natural}\Omega-n$ coincides with the strip $S:={\{w: \, 0<\Im w<1\}}$. By construction, $\C\setminus\Omega$ does not have bounded connected components, so $\Omega$ is simply connected. Therefore, up to a translation, $\Omega$ is the Koenigs domain of a hyperbolic self-map $\varphi\in\U(\UD)$ given by  ${\varphi:=h^{-1}\circ (z\mapsto z+1)\circ h}$, where $h$ is a conformal map of~$\UD$ onto~$\Omega$.  For this self-map, we have $\mathcal A_\varphi={\{c\in\Real:\Omega+c\subset\Omega\}}=\mathcal A$, and it only remains to notice that ${\varphi\not\in\Aut}$ because ${\Omega\neq S}$.
\end{proof}

%%%%%%%%%%%%%%%%%%%%%%%%%%%%%%%%%%%%%%%%%%%%%%%%%%%%%%%%%%%%%%%%%%%%%%%%%%%%%%%%%%%%%%%%%%%%%%%%%%
%%%%%%%%%%%%%%%%%%%%%%%%%%%%%%%%%%%%%%%%%%%%%%%%%%%%%%%%%%%%%%%%%%%%%%%%%%%%%%%%%%%%%%%%%%%%%%%%%%
%%%%%%%%%%%%%%%%%%%%%%%%%%%%%%%%%%%%%%%%%%%%%%%%%%%%%%%%%%%%%%%%%%%%%%%%%%%%%%%%%%%%%%%%%%%%%%%%%%
%%%%%%%%%%%%%%%%%%%%%%%%%%%%%%%%%%%%%%%%%%%%%%%%%%%%%%%%%%%%%%%%%%%%%%%%%%%%%%%%%%%%%%%%%%%%%%%%%%
%%%%%%%%%%%%%%%%%%%%%%%%%%%%%%%%%%%%%%%%%%%%%%%%%%%%%%%%%%%%%%%%%%%%%%%%%%%%%%%%%%%%%%%%%%%%%%%%%%
%%%%%%%%%%%%%%%%%%%%%%%%%%%%%%%%%%%%%%%%%%%%%%%%%%%%%%%%%%%%%%%%%%%%%%%%%%%%%%%%%%%%%%%%%%%%%%%%%%
\section{Parabolic self-maps of positive hyperbolic step}\label{S_para-positive}
%%%%%%%%%%%%%%%%%%%%%%%%%%%%%%%%%%%%%%%%%%%%%%%%%%%%%%%%%%%%%%%%%%%%%%%%%%%%%%%%%%%%%%%%%%%%%%%%%%
%%%%%%%%%%%%%%%%%%%%%%%%%%%%%%%%%%%%%%%%%%%%%%%%%%%%%%%%%%%%%%%%%%%%%%%%%%%%%%%%%%%%%%%%%%%%%%%%%%
%%%%%%%%%%%%%%%%%%%%%%%%%%%%%%%%%%%%%%%%%%%%%%%%%%%%%%%%%%%%%%%%%%%%%%%%%%%%%%%%%%%%%%%%%%%%%%%%%%
We start by proving our result on simultaneous linearization, i.e Theorem~\ref{TH_simultaneous}, for the remaining~--- and the most complicated~---  case of a parabolic self-map $\varphi\in\U(\UD)$ of positive hyperbolic step.
Replacing, if necessary,  $\varphi$ with $z\mapsto \overline{\varphi(\bar z)}$, throughout this section we will suppose, without loss of generality that the canonical model of~$\varphi$ is of the form ${(\UH,h,z\mapsto z+1)}$.

For parabolic self-maps of positive hyperbolic step, Theorem~\ref{TH_simultaneous} is contained in Proposition~\ref{PR_simultaneous-para-positive}. In the proof of the latter we make use of a lemma (Lemma~\ref{Lem:intertwining-exp}) given in the Appendix.

\begin{proof}[\proofof{Proposition~\ref{PR_simultaneous-para-positive}}] \label{proofTH_simultaneousparpositive} Let $\varphi$ be a parabolic self-map of $\D$ of positive hyperbolic step. Without loss of generality we may assume that the base space of its canonical model ${(S,h,z\mapsto z+1)}$ is ${S=\UH}$. As before, denote ${\Omega:=h(\UD)}$.

Firstly, for each fixed $\psi\in\mathcal Z(\varphi)$, we will construct a univalent function ${h_*=h_{\varphi,\psi}}$ satisfying system~\eqref{EQ_two-Abel-in-remark} for a suitable constant $c_{\varphi,\psi}$. Secondly, we will show how this construction leads to the common solution $h_{\varphi,\Delta}$.

By Theorem~\ref{Thm:general}, for any ${\psi\in\mathcal Z(\varphi)}$ there exists $g\in\U(\UH)$ with ${g(\Omega)\subset\Omega}$ such that ${g\circ h}={h\circ\psi}$ and
$
   {g(w+1)=g(w)+1}
$ for all ${w\in\UH}$.
  By Lemma~\ref{Lem:intertwining-exp}\,\ref{IT_periodic-univalent},  there exists $f\in\U(\UD)$ with ${f(0)=0}$
such that
\begin{equation}\label{EQ_f-and-g}
 \exp\big(2\pi i\, g(w)\big)=f\big(e^{2\pi i\, w}\big), \quad\text{for all~$w\in\H$.}
\end{equation}

By the Schwarz Lemma, $|f'(0)|\le 1$ and if ${|f'(0)|=1}$, then there is ${c\in \R}$ such that ${g(z)=z+c}$ for all ${z\in\UH}$, and as a result, ${h_*=h_{\varphi,\psi}:=h}$, ${c_{\varphi,\psi}:=c}$ satisfy~\eqref{EQ_two-Abel-in-remark}. Therefore, we may assume that  $0{<|f'(0)|<1}$. In this case, by Theorem~\ref{Thm:model}, $f$ admits a holomorphic model of the form ${(\C,h_0, w \mapsto \lambda w)}$ with ${\lambda:=f'(0)\in\UD}$.
This means that $h_0$ is univalent in~$\UD$, $h_0(0)=0$, and
\begin{equation}\label{EQ_Schr-h_0}
 h_0\circ f = \lambda h_0.
\end{equation}

By Lemma~\ref{Lem:intertwining-exp}\,\ref{IT_periodic-univalent-converse}, applied with $f$ replaced by~$h_0$, we have $h_0(e^{2\pi i w})={\exp\big(2\pi i F(w)\big)}$ for all ${w\in\UH}$, where ${F:\UH\to\C}$ is univalent and satisfies $F(w+1)={F(w)+1}$. Taking into account~\eqref{EQ_f-and-g} and~\eqref{EQ_Schr-h_0}, for all ${w\in\UH}$ we obtain
$$
 \exp\big(2\pi i F(g(w))\big)=h_0\big(e^{2\pi i g(w)}\big)=h_0\big(f(e^{2\pi i w})\big)=\lambda h_0\big(e^{2\pi i w}\big)=\lambda\exp\big(F(w)\big).
$$
Hence, $F\circ g=F+c$ for a suitable constant $c\in\C$ satisfying ${e^{2\pi i c}=\lambda}$. Thus,
\begin{align*}
F\circ h\circ\varphi &= F\circ (h+1)=(F\circ h)+1,\quad \text{and}\\
F\circ h\circ\psi    &= F\circ g\circ h = (F\circ h)+c,
\end{align*}
i.e. ${h_*=h_{\varphi,\psi}:=F\circ h}$, ${c_{\varphi,\psi}:=c}$ satisfy~\eqref{EQ_two-Abel-in-remark}.

\medskip
We also note that since $F$ is univalent in~$\UH$, since $F(\UH)+1=F(\UH)$, and since
$$
  Q\Big(\bigcup\nolimits_{n\in\Natural}F(\UH)-nc\Big)~=~\bigcup\limits_{n\in\Natural}Q\big(F(\UH)-nc\big)%
    ~=~ \bigcup\limits_{n\in\Natural}\lambda^{-n}h_0(\UD^*)~=~\C^*,~{}~{}~ {Q(w):=e^{2\pi i w}},
$$
we have that $(\C,H,z\mapsto z+1)$, where $H:=c^{-1}F$, is a holomorphic model for $g:\UH\to\UH$. In particular, $g$ is parabolic of zero hyperbolic step.

\smallskip

Now let $\Delta\subset\Zen(\varphi)$ be such that any two elements of~$\Delta$ commute. If for any $\psi\in\Delta$, the function $g:=h\circ\psi\circ h^{-1}$ is of the form~${z\mapsto z+c}$ for some constant ${c=c_\psi\in\Real}$, then as we have seen above, for every ${\psi\in\Delta}$ the same univalent function ${h_{\varphi,\Delta}:=h}$ satisfy system~\eqref{EQ_simultaneous} if we set ${c_{\varphi,\psi}:=c_\psi}$; i.e. in this case, the conclusion of Proposition~\ref{PR_simultaneous-para-positive} holds with $\beta_\Delta:=\id_\UH$.

So assume now that there exists $\psi\in\Delta$ of the form $\psi={h^{-1}\circ g\circ h}$, where ${g\in\Zen_\UH(z\mapsto z+1)}$ and ${g(w)-w}$ is not a real constant. Consider an arbitrary $\widetilde\psi\in\Delta$. Then ${\widetilde\psi=h^{-1}\circ \tilde g\circ h}$ for some ${\tilde g\in\Zen_\UH(z\mapsto z+1)}$. At the same time, $\psi$ and $\widetilde\psi$ commute, and hence, ${\tilde g\in\Zen_\UH(g)}$. As we have seen above, $g$ is parabolic of zero hyperbolic step and, up to an additive constant, $H$ is the Koenigs function of~$g$. Therefore, according to Proposition~\ref{PR_abelian}, which we proved in Sect.\,\ref{S_para-zero}, $\tilde g={H^{-1}\circ(\zeta\mapsto\zeta+\tilde c)\circ H}$ for a suitable constant ${\tilde c\in\C}$. It follows that~${F\circ h\circ\widetilde\psi}={F\circ \tilde g\circ h}={F\circ h+c\tilde c}$. This means that $h_{\varphi,\Delta}:=F\circ h$ is the desired common solution to~\eqref{EQ_simultaneous}, i.e. the conclusion of Proposition~\ref{PR_simultaneous-para-positive} holds with ${\beta_\Delta:=F}$.
\end{proof}

\begin{remark}
As we have shown, see Theorems~\ref{PropA_varphi}, \ref{Thm:0HS} and~\ref{TH_isomorphism}, if $\varphi\in\U(\UD)$ is hyperbolic or parabolic of zero hyperbolic step, then ${\Zen(\varphi)=\Tmap_\varphi(\mathcal A_\varphi)}$. However, for parabolic self-maps of \textit{positive} hyperbolic step, the situation is different. By Proposition~\ref{positivo1}, which we state and prove below, ${\psi\in\Zen(\varphi)}$ belongs to~$\Tmap_\varphi(\mathcal A_\varphi)$ if and only the corresponding function $F$ in representation~\eqref{EQ_represent-P-PHS} is constant, and the latter is not always the case, as numerous examples in Section~\ref{S_examples} illustrate.
\end{remark}

In order to proceed, we need to introduce some more notation.
Consider a map $\varphi \in \mathrm{Hol}(\mathbb{D},\mathbb{C})$ with a
boundary fixed point $\sigma\in \partial \mathbb{D}$. We say that $\varphi $ is
of angular-class of order $p\in \mathbb{N}$ at $\sigma,$ and we denote it by $%
\varphi \in C_{A}^{p}(\sigma),$ if
\begin{equation*}
	\varphi (z)=\sigma+\sum_{k=1}^{p}\frac{a_{k}}{k!}(z-\sigma)^{k}+\gamma (z),\quad \quad
	\text{ }z\in \mathbb{D},
\end{equation*}%
where $a_{1},...,a_{p}\in \mathbb{C}$ and $\gamma \in \mathrm{Hol}(\mathbb{D},\mathbb{C})$ with
\begin{equation*}
	\angle \lim_{z\rightarrow b}\frac{\gamma (z)}{(z-\sigma)^{p}}=0.
\end{equation*}%
It is clear that the numbers $a_{1},...,a_{p}$ appearing in the above
expression are necessarily unique. They are called angular derivatives of higher orders. Correspondingly, we will use the notation  ${\varphi^{(k)}_A}(\sigma):=a_k$. Following the tradition we will keep omitting the subscript ``$A$'' for the angular derivative of the first order.

Further, we define $C_A^{\,\infty}(\sigma):=\bigcap_{p\in\Natural}C_A^p(\sigma)$.
Given a map $\phi\in\Hol(\H)$ with $${\angle \lim_{w\to\infty}\phi(w)=\infty},$$ we say that $\phi\in C^p_A(\infty)$ for some $p\in\Natural\cup\{\infty\}\,$ if ${C^{-1}\circ \phi\circ C}\in  C_{A}^{p}(1)$, where $C(z):=i\frac{1+z}{1-z}$, ${z\in\D}$.

\begin{proposition}\label{positivo1} Let $\varphi\in\U(\D)$ be a parabolic self-map of positive hyperbolic step with Denjoy\,--\,Wolff point $\tau\in\partial\D$ and with canonical model $(\UH,h,z\mapsto z+1)$.  Then, $\psi\in\Zen(\varphi)$ if and only if it admits the following representation
\begin{equation}\label{EQ_represent-P-PHS}
\psi~=~h^{-1}\,\circ\, \big(w \mapsto  w+F(e^{2\pi i w})\big)\,\circ\, h,
\end{equation}
where $F\in\Hol(\UD,\UH\cup\Real)$ is such that $g(w):=w+F(e^{2\pi i w})$ is univalent in~$\UH$ and maps ${\Omega:=h(\D)}$ into itself.

Moreover, every $\psi\in\Zen(\varphi)\setminus\{\id_\UD\}$ is a parabolic self-map of $\D$ with Denjoy\,--\,Wolff point~$\tau$, and the (uniquely defined) corresponding function~$g$  is a parabolic self-map of~$\UH$ with Denjoy\,--\,Wolff point at~$\infty$ and it belongs to $C^{\,\infty}_A(\infty)$.
\end{proposition}

\begin{proof} Thanks to Theorem~\ref{Thm:general}, in order to prove the representation~\eqref{EQ_represent-P-PHS} it is sufficient to check that $\mathcal B_1:=\{g\in\Zen_\H(w\mapsto w+1):g(\Omega)\subset\Omega\}$ coincides with the class $\mathcal B_2$ formed by all univalent functions ${g:\UH\to\C}$ mapping $\Omega$ into itself and having the form ${g(w)=w+F(e^{2\pi i w})}$, ${w\in\UH}$, where $F$ is holomorphic in~$\UD$ with ${\Im F\ge0}$.

It is easy to see that $\mathcal B_2\subset \mathcal B_1$.
To show that $\mathcal B_1\subset \mathcal B_2$, fix $g\in\mathcal B_1$ and denote $f_g:=g-\id_\H$. Then $f_g(w+1)=f_g(w)$ for every $w\in\H$.  By Lemma~\ref{Lem:intertwining-exp}\,\ref{IT_periodic-function},
there exists $F_g\in \Hol(\D^*)$ such that $f_g(w)=F_g\big(e^{2\pi i\, w}\big)$ for all ${w\in\H}$.
It is known, see e.g.~ \cite[\S26]{Valiron}, that for any ${g\in\Hol(\UH)}$, the angular limit $\anglim_{w\to\infty} g'(w)$ exists finitely and coincides with the angular derivative $g'(\infty):={\anglim_{z\to\infty}g(w)/w}$, which in turn is equal to ${\inf_{w\in\UH}\Im g(w)/\Im w}$. Since ${g\in\mathcal B_1}$, then  in addition we have ${\int_w^{w+1}g'(\zeta)\di\zeta}={g(w+1)-g(w)}{=1}$ for any ${w\in\UH}$. Hence, ${g'(\infty)=1}$ and, as a consequence,  ${f_{g}(\H)\subset \H\cup\Real}$. In particular, it follows that $0$~is a removable singularity of~$F_g$, i.e. in reality $F_g$ is holomorphic in~$\UD$. This shows that ${g\in\mathcal B_2}$. Thus, we have proved that ${\mathcal B_1=\mathcal B_2}$.

Given ${g\in\mathcal B_2}$, for any $k\in\Natural\,$ we have
$$
\anglim_{w\to\infty}w^k\big(\,g(w)\,-\,\big(w+F_{g}(0)\big)\,\big)=\anglim_{w\to\infty}w^k\big(F_g(e^{2\pi i w})-F_g(0)\big)=0.
$$
Clearly, it follows that $g$ is parabolic with the Denjoy\,--\,Wolff point at~$\infty$ and that ${g\in C^{\,\infty}_A(\infty)}$.

The fact that every $\psi\in\Zen(\varphi)\setminus\{\id_\UD\}$ is parabolic provided that $\varphi$ is parabolic and that $\psi$ and $\varphi$ have the same Denjoy\,--\,Wolff point is known, see \cite[Theorem~6]{Behan}, \cite[Corollary~4.1]{Cowen-comm}, and \cite[Section~4.10]{Abate2}. Below we give another proof of this fact (restricted to univalent self-maps $\varphi$ and~$\psi$) based on Theorem~\ref{TH_simultaneous} and Theorem~\ref{TH_ch-para-hyper}, which will be proved in the Appendix. For a map ${h:\UD\to\Complex}$ and a number ${c\in\C^*}:={\C\setminus\{0\}}$ we denote $S[h;c]:=\bigcup_{n\in\Natural}\big(h(\UD)-nc\big)$; in the special case ${c=1}$ we will write $S[h]$ instead of $S[h;1]$.

Let $\psi\in\Zen(\varphi)\setminus\{\id_\UD\}$. By Theorem~\ref{TH_simultaneous} applied to $\Delta:=\{\psi\}$, for a suitable $c_{\varphi,\psi}\in\C$ the system~\eqref{EQ_simultaneous} has a univalent solution, which we denote by $h_{\varphi,\psi}$. Clearly, $c_{\varphi,\psi}\neq0$ because $\psi\neq\id_\UD$. Since $\varphi$ is parabolic, by Theorem~\ref{TH_ch-para-hyper}\,\ref{IT_(D)}, $S[h_{\varphi,\psi}]$ contains a horizontal half-plane.  Moreover, adding to $h_{\varphi,\psi}$ a suitable constant and replacing, if necessary, the functions $\varphi$, $\psi$, and $h_{\varphi,\psi}$ with their transforms of the form $f~\mapsto~\big(z\mapsto\overline{f(\bar z)}\big)$, we may assume that ${\UH\subset S[h_{\varphi,\psi}]}$. In view of Theorem~\ref{TH_ch-para-hyper}\,\ref{IT_(D)} and Remark~\ref{RM_ch-para-hyper}, to show that $\psi$ is parabolic, it suffices to prove that $S[h_{\varphi,\psi};c_{\varphi,\psi}]$ contains a half-plane. Denote ${\Omega:=h_{\varphi,\psi}(\UD)}$ and notice that by~\eqref{EQ_simultaneous},
\begin{equation}\label{EQ_invar-wrt-two-transl}
 \Omega+1\subset\Omega\quad \text{and}\quad \Omega+c_{\varphi,\psi}\subset\Omega.
\end{equation}
Consider the following four cases.

\StepC{1a}{$c_{\varphi,\psi}\in(-\infty,0)\setminus\mathbb{Q}$} According to~\eqref{EQ_invar-wrt-two-transl}, ${\Omega+c\subset\Omega}$ for all~$c$ from the set $\{{n+c_{\varphi,\psi}k}:{n,k\in\N_0}\}$, which is dense in~$\Real$. Taking into account that ${\Omega\neq\C}$ and that ${S[h_{\varphi,\psi}]\supset\UH}$, we easily conclude that $\Omega$ is a half-plane containing~$\UH$. In this case, $\varphi$ and $\psi$ are parabolic automorphisms with the same Denjoy\,--\,Wolff point ${\tau=h_{\varphi,\psi}^{-1}(\infty)}$.

\StepC{1b}{$c_{\varphi,\psi}\in(-\infty,0)\cap\mathbb{Q}$} As in the previous case,  ${\Omega+c\subset\Omega}$ for all $c\in\{{n+c_{\varphi,\psi}k}:{n,k\in\N_0}\}=:E$. It is easy to see that ${E\supset\Z}$. As a result, ${w\mapsto w+1}$ is an automorphism of~$\Omega$ and hence ${\varphi\in\Aut}$. Moreover, writing $c_{\varphi,\psi}=-p/q$ with ${p,q\in\N}$, we see that $\psi^{\circ q}$ is the inverse of~$\varphi^{\circ p}$. In particular, $\psi$ is a parabolic automorphism with the same Denjoy\,--\,Wolff point as~$\varphi$.

\StepC{2}{$c_{\varphi,\psi}>0$} Essentially the same argument as in Remark~\ref{RM_absorption-continuous} shows that if ${w\in\UH}$, then there exists ${m(w)\in\N}$ such that $R_w:={\{w+m(w)+t:t\ge0\}}\subset\Omega$. It follows that ${S[h_{\varphi,\psi};c]\supset\UH}$ for any~${c>0}$. In particular, $\psi$ is parabolic.

Now fix ${w\in\UH}$ and let ${w_0:=w+m(w)}$, ${z_0:=h_{\varphi,\psi}^{-1}(w_0)}$. Since $R_w$ is a slit in~$\Omega$, the restriction of $h_{\varphi,\psi}^{-1}$ to~$R_w$ has a limit at~$\infty$; see e.g. \cite[Theorem~1 in \S{}II.3]{Goluzin}. Denote this limit by~$\tau_2$. On the one hand, it follows that $\psi^{\circ n}(z_0)={h^{-1}(w_0+nc_{\varphi,\psi})}$ tends to~$\tau_2$ as ${n\to+\infty}$ and hence, $\tau_2$ is the Denjoy\,--\,Wolff point of~$\psi$.
On the other hand,
 ${h^{-1}(w_0+n)}={\varphi^{\circ n}(z_0)\to\tau}$ as ${n\to+\infty}$, and hence $\tau_2=\tau$, as desired.

\StepC{3}{$c_{\varphi,\psi}\in\Complex\setminus\Real$} Fix some ${w_0\in\UH}$ with ${\Im w_0>\max\{-\Im c_{\varphi,\psi},\,0\}}$. Then the compact set $\Pi:=\{w_0+t+s c_{\varphi,\psi}:t,s\in[0,1]\}$ lies in~$\UH$. Since ${\UH\subset S[h_{\varphi,\psi}]}$, there exists ${n_0\in\N}$ such that $\Pi+n_0\subset\Omega$. Because of~\eqref{EQ_invar-wrt-two-transl}, it further follows that $A:=\{w_0+n_0+t+s c_{\varphi,\psi}:t,s\ge0\}$ lies in~$\Omega$. As a consequence $S[h_{\varphi,\psi}; c_{\varphi,\psi}]$ contains the half-plane $\{w:\Im\alpha(w-w_0-n_0)>0\}$, where ${\alpha:=c_{\varphi,\psi}^{-1}}$ if ${\Im c_{\varphi,\psi}<0}$, and ${\alpha:=-c_{\varphi,\psi}^{-1}}$  otherwise. Hence, $\psi$ is parabolic.

Moreover, if we fix some $w_o\in A$, then $R_1:={\{w_o+t:t\ge0\}}$ and $R_2:={\{w_o+s c_{\varphi,\psi}:s\ge0\}}$ are two slits in~$\Omega$ and hence, the limits ${\tau_j:=\lim_{R_j\ni w\to\infty}h_{\varphi,\psi}^{-1}(w)}$, ${j=1,2}$, do exist; see e.g. \cite[Theorem~1 in \S{}II.3]{Goluzin}. On the one hand, by the same argument as in the previous case, $\tau_1$ and $\tau_2$ are the Denjoy\,--\,Wolff points of $\varphi$ and~$\psi$, respectively. On the hand, the slits~$R_1$ and $R_2$ define the same accessible boundary point of~$\Omega$ and hence, ${\tau_2=\tau_1}$; see again \cite[Theorem~1 in \S{}II.3]{Goluzin}.

\smallskip
The proof is now complete.
\end{proof}

Next we are going to show that in spite of the difference between parabolic self-maps of zero hyperbolic step and those of positive hyperbolic step, the definition of the homomorphism ${\Smap_\varphi:[\Zen(\varphi),\circ]\to[\C,+]}$ given for the former case in Sect.\,\ref{S_para-zero} can be extended to the latter case. This fact follows from representation~\eqref{EQ_represent-P-PHS}, and as we will see in the proof of Theorem~\ref{Thm:PHS-S},
\begin{equation}\label{EQ_Smap_F(0)}
 \Smap_\varphi(\psi)=F(0).
\end{equation}

\begin{remark}\label{RM_left-inverse}
 Formula~\eqref{EQ_Smap_F(0)} implies that $\Smap_\varphi$ is a \textit{left} inverse of $\Tmap_\varphi$.
\end{remark}

\begin{theorem} \label{Thm:PHS-S}
Let $\varphi\in\U(\UD)$ be parabolic of positive hyperbolic step with Denjoy\,--\,Wolff point ${\tau\in\partial\UD}$ and with canonical model ${(\UH,h,z\mapsto z+1)}$. Denote $\Omega:=h(\D)$. Then, for every ${\psi\in\Zen(\varphi)}$, there exists
\begin{equation}\label{EQ_Slim-for_PHS}	\Smap_{\varphi}(\psi):=\angle\lim_{z\to\tau}\frac{\psi(z)-z}{\varphi(z)-z}=\angle\lim_{z\to\tau}h^{\prime}(z)(\psi(z)-z)\in  \UH\cup\Real.
\end{equation}
	Moreover, the following statements hold.
	\begin{Ourlist}
		\item\label{IT_PHS-Siff-id} $\Smap_{\varphi}(\psi)=0$ if and only if $\psi=\id_\UD$.\smallskip
		\item\label{IT_PHS-homom} $\Smap_{\varphi}(\psi_1\circ\psi_2)=\Smap_{\varphi}(\psi_1)+\Smap_{\varphi}(\psi_2)$, for every $\psi_1,\psi_2\in\Zen(\varphi)$.\smallskip
		\item\label{IT_PHS-cont} $\Smap_{\varphi}$ is a continuous map from $\Zen(\varphi)$ into $\UH\cup\Real$;\smallskip
\item\label{IT_PHS-closed} The image of the map~$\Smap_\varphi$, $\mathcal A^*_\varphi:=\Smap_\varphi\big(\Zen(\varphi)\big)$ is a closed set in~$\C$.\smallskip
        \item\label{IT_PHS-real} $\mathcal A_\varphi^\R:=\mathcal A_\varphi^*\cap\Real=\mathcal A_\varphi\cap\Real$ and moreover, each $b\in\mathcal A_\varphi^\R$ has exactly one preimage~$\psi_b$ under~$\Smap_\varphi$ given by ${\psi_b=\Tmap_\varphi(b)}$.
	\end{Ourlist}
\end{theorem}

\begin{proof}
Fix $\psi\in\Zen(\varphi)\setminus\{\id_\UD\}$. By Proposition~\ref{positivo1}, we find that $\psi$ is parabolic with Denjoy\,\,--Wolff point at~$\tau$. Recall that $h$ is the Koenigs map of~$\varphi$. Hence  ${h(\varphi(z))-h(z)=1}$ and by representation~\eqref{EQ_represent-P-PHS}, ${h(\psi(z))-h(z)}=F(e^{2\pi i h(z)})$ for all ${z\in\UD}$. By \cite[Theorem~3]{Pom79}, we have  $\angle\lim_{z\to\tau}\Im h(z)=+\infty$. Therefore,
  $$
\angle\lim_{z\to\tau}F(e^{2\pi i h(z)})=F(0).
  $$
Thus, appealing twice to Proposition~\ref{positivo2}, we deduce that
	\begin{equation*}
\anglim_{z\to\tau}\frac{\psi(z)-z}{\varphi(z)-z}~=~
   \anglim_{z\to\tau}\frac{1}{h^\prime(z)(\varphi(z)-z)}\,\cdot\,
   \anglim_{z\to\tau}\dfrac{h^\prime(z)(\psi(z)-z)}{F(e^{2\pi i h(z)})}\,\cdot\,
   \anglim_{z\to\tau}F(e^{2\pi i h(z)})~\,=\,~F(0).
	\end{equation*}
By a similar calculation, also $\anglim_{z\to\tau}h^\prime(z)(\psi(z)-z)\,=\,F(0)\,\in\,\UH\cup\Real$.

\StepP{\ref{IT_PHS-Siff-id}} If $\psi=\id_\UD$, then obviously the angular limit~\eqref{EQ_Slim-for_PHS} exists and equals zero. If $\psi\neq\id_\UD$, then $F$ in representation~\eqref{EQ_represent-P-PHS} cannot vanish identically in~$\UD$, and therefore, ${\Smap_{\varphi}(\psi)=F(0)\neq0}$.

\StepP{\ref{IT_PHS-homom}}
Let $\psi=\psi_1\circ\psi_2$, where $\psi_j\in\Zen(\varphi)$, $j=1,2$. Write representation~\eqref{EQ_represent-P-PHS} for~$\psi_j$:
$$
\psi_j~=~h^{-1}\,\circ\, \big(w \mapsto  w+F_j(e^{2\pi i w})\big)\,\circ\, h,\quad j=1,2.
$$
Then it is easy to see that $\big(h\circ\psi\circ h^{-1}\big)(w)=w+F(e^{2\pi i w})$ for all ${w\in h(\UD)}$, where $F$ is given by $$F(\zeta):=F_2(\zeta)+F_1\big(\zeta e^{2\pi F_2(\zeta)}\big)\quad\text{for all~$\zeta\in\UD$}.$$
It follows that $$\Smap_\varphi(\psi)=F(0)=F_2(0)+F_1(0)=\Smap_\varphi(\psi_2)+\Smap_\varphi(\psi_1).$$

\StepP{\ref{IT_PHS-cont} and \ref{IT_PHS-closed}} To see that the set $\mathcal A^*_\varphi:=\Smap_\varphi\big(\Zen(\varphi)\big)$ is closed, consider a sequence ${(\psi_n)\subset \Zen(\varphi)}$ such that
$$
c_n:=\Smap_\varphi(\psi_n)~\to~ c\quad\text{as~$~n\to+\infty$}
$$
for some $c\in\C$. We have to show that $c\in\mathcal A^*_\varphi$.

Write representation~\eqref{EQ_represent-P-PHS} for~$\psi_n$:
$$
 \psi_n~=~h^{-1}\,\circ\, \big(w \mapsto  w+F_n(e^{2\pi i w})\big)\,\circ\, h,\quad n\in\Natural.
$$
Recall that $F_n(\UD)\subset\UH\cup\Real$ for all $n\in\Natural$. Therefore, $F_n$'s form a normal family in~$\UD$.
Moreover, $F_n(0)=c_n$ for each~$n\in\Natural$. It follows that passing to a subsequence, we may suppose that $(F_n)$ converges locally uniformly in~$\UD$ to some ${F\in\Hol(\UD,\UH\cup\Real)}$.

Since for each $n\in\Natural$, $g_n(w):=w+F_n(e^{2\pi i w})$ is a univalent in~$\UH$ and since $(g_n)$ converges locally uniformly in~$\UH$ to $g(w):=w+F(e^{2\pi i w})$, it follows that $g$ is either univalent in~$\UH$ or constant. The latter alternative is, in fact, impossible because
$$
  g'(iy)=1+2\pi e^{-2\pi y} F'(e^{-2\pi y})~\to~1\quad\text{as~$~y\to+\infty$}.
$$

Furthermore, $g_n(\Omega)\subset\Omega:=h(\UD)$ for all $n\in\Natural$. Hence, $g(\Omega)\subset\overline{\Omega}$. Since $g$ is not constant, in fact we have ${g(\Omega)\subset\Omega}$. It follow that $\psi:={h^{-1}\circ g\circ h}$ is a univalent self-map of~$\UD$ commuting with~$\varphi$. Thus,
$$
 c=F(0)=\Smap_\varphi(\psi)\in\mathcal A_\varphi^*,
$$
as desired.

Now to check the continuity of $\Smap_{\varphi}$, we again consider a sequence ${(\psi_n)\subset \Zen(\varphi)}$. Suppose that $(\psi_n)$ converges locally uniformly in~$\UD$ to some ${\psi\in\Zen(\varphi)}$. Arguing as above, we see that $\psi$ admits representation~\eqref{EQ_represent-P-PHS} with $F:=\lim_{n\to+\infty} F_n$. It follows that
$$
\Smap_\varphi(\psi_n)=F_n(0)~\to~F(0)=\Smap_\varphi(\psi)\quad\text{as~$~n\to+\infty$},
$$
which means that $\Smap_{\varphi}$ is continuous.

\StepP{\ref{IT_PHS-real}} All we have to prove is that if ${\psi\in\Zen(\varphi)}$ and $b:={\Smap_\varphi(\psi)\in\Real}$, then ${\psi=\Tmap_\varphi(b)}$. To this end it is sufficient to apply Proposition~\ref{positivo1} and recall that, as we have proved above, $b=F(0)$. Since ${F(\UD)\subset\UH\cup\Real}$, it follows that $F\equiv b$ and hence the desired conclusion follows immediately from representation~\eqref{EQ_represent-P-PHS}.
\end{proof}

\begin{remark} \label{Rem:cinPHP} The ideas of the above proof allows us to identify the value $c_{\varphi,\psi}$ in Theorem~\ref{TH_simultaneous} and in system~\eqref{EQ_two-Abel-in-remark},  when $\varphi$ is parabolic of positive hyperbolic step. Indeed, take a univalent function $h_{\varphi,\psi}\in\Hol(\UD,\C)$ and a constant~$c_{\varphi,\psi}\in\C$ such that
\begin{equation*}
h_{\varphi,\psi}\circ\varphi=h_{\varphi,\psi}+1\quad\text{and}\quad
h_{\varphi,\psi}\circ\psi=h_{\varphi,\psi}+c_{\varphi,\psi}.
\end{equation*}
Then, by Proposition~\refeq{positivo2}, we deduce that
	$$
1=\angle \lim_{z\to\tau}\dfrac{h_{\varphi,\psi}(\psi(z))-h_{\varphi,\psi}(z)}{h_{\varphi,\psi}^\prime(z)(\psi(z)-z)}= \frac{ c_{\varphi,\psi}}{\angle \lim_{z\to\tau}h_{\varphi,\psi}^\prime(z)(\psi(z)-z)}
$$
and
$$
1=\angle \lim_{z\to\tau}\dfrac{h_{\varphi,\psi}(\varphi(z))-h_{\varphi,\psi}(z)}{h_{\varphi,\psi}^\prime(z)(\varphi(z)-z)}
= \frac{ 1}{\angle \lim_{z\to\tau}h_{\varphi,\psi}^\prime(z)(\varphi(z)-z)}.
$$
Therefore, similarly to the case of zero hyperbolic step,
\begin{equation}\label{EQ_c-phi-psi_PHS}
c_{\varphi,\psi}=\angle\lim_{z\to\tau}\frac{\psi(z)-z}{\varphi(z)-z}=\Smap_\varphi(\psi).
\end{equation}
\end{remark}

\medskip If $\varphi\in C_A^2(\tau)$, there exists yet another way to express the quantity~\eqref{EQ_c-phi-psi_PHS}, which is valid for any univalent parabolic self-map regardless of whether it of \textit{zero or positive} hyperbolic step.
\begin{corollary}\label{CR_parabolic-second-der}  Let $\varphi\in\U(\UD)$ be a parabolic self-map with Denjoy\,--\,Wolff point ${\tau\in\partial\UD}$ and assume that $\varphi\in C_A^2(\tau)$. Then ${\Zen(\varphi)\subset C_A^2(\tau)}$, and for any ${\psi\in\Zen(\varphi)}$,
\begin{equation}\label{EQ_when-the second-derivative exists}
	\psi''_A(\tau)=\varphi''_A(\tau)\Smap_{\varphi}(\psi).
	\end{equation}
In particular, for $\psi\in\Zen(\varphi)\setminus\{\id_\UD\}$, we have that ${\psi''_A(\tau)=0}$ if and only if ${\varphi''_A(\tau)=0}$.	
	\end{corollary}
\begin{proof}
Let $\psi\in\Zen(\varphi)$. By Theorems~\ref{Thm:0HS} and \ref{Thm:PHS-S}, the limit
$$
\Smap_\varphi(\psi)=\anglim_{z\to\tau}\frac{\psi(z)-z}{\varphi(z)-z}
$$
exists finitely.
Furthermore, the assumption $\varphi\in C^2_A(\tau)$ means that the limit
$$
\varphi_A''(\tau)=\anglim_{z\to\tau}\frac{2\big(\varphi(z)-z\big)}{(z-\tau)^2}
$$
exists finitely. It follows immediately that $\psi_A''(\tau)= \anglim_{z\to\tau}2\big(\psi(z)-z\big)/(z-\tau)^2$ does exists finitely as well, i.e. $\psi\in C^2_A(\tau)$, and that the formula~\eqref{EQ_when-the second-derivative exists} holds.

If additionally $\psi\neq\id_\UD$, then by Theorems~\ref{Thm:0HS} and \ref{Thm:PHS-S}, ${\Smap_\varphi(\psi)\neq0}$ and in this case, clearly, ${\psi''_A(\tau)=0}$ if and only if ${\varphi''_A(\tau)=0}$.
\end{proof}

\begin{remark}\label{RM_rigidity}
Recall that in case of a parabolic self-map~$\varphi$ of zero hyperbolic step, by Theorem~\ref{Thm:0HS}, $\Smap_\varphi=\Tmap_\varphi^{-1}$. Somewhat similarly, in case of positive hyperbolic step, by Theorem~\ref{Thm:PHS-S}\,\ref{IT_PHS-real},  the restriction of $\Smap_\varphi$ to $\Smap_\varphi^{-1}(\Real)$ is injective and its inverse is given by $\Tmap_\varphi$ restricted to $\mathcal A_\varphi\cap\Real$.
Bearing this in mind, we see that under the hypothesis of Corollary~\ref{CR_parabolic-second-der} the following statement holds: if $\psi\in\Zen(\varphi)$ and ${\psi_A''(\tau)=k\varphi_A''(\tau)\neq0}$ for some ${k\in\Z}$, then ${\psi=\varphi^{\circ k}}$.
  This gives an improvement of \cite[Theorem~2.4\,(6)]{BTV} and \cite[Theorem~1.2\,(1)-(2)]{Tauraso}, under the additional assumption that the commuting parabolic self-maps under consideration are univalent.
  \end{remark}

\begin{remark}
Assume that $\varphi\in\U(\UD)$ is a parabolic self-map with Denjoy\,--\,Wolff point ${\tau\in\partial\UD}$ and that $\varphi\in C_A^3(\tau)$ with $\varphi_A''(\tau)=0$.
By \cite[Theorem~7.1]{CDP}, $-\tau^2\varphi_A'''(\tau)>0$.
Moreover, by Corollary~\ref{CR_parabolic-second-der}, $\psi_A''(\tau)=0$ for all $\psi \in \Zen(\varphi)$. This, together with \eqref{EQ_c-phi-psi_PHS} and the very definition of the angular derivative of order $3$, implies that $\psi \in C_A^3(\tau)$ and
$$
\psi'''_A(\tau)= \Smap_\varphi(\psi)\, \varphi'''_A(\tau).
$$
In particular, it follows that, in this case, $$\mathcal A_\varphi=\mathcal A_\varphi^*\subset[0,+\infty).$$
Moreover, in the same manner as in Remark~\ref{RM_rigidity} we deduce that if, in addition to the above assumptions, ${\psi_A'''(\tau)=k\varphi_A'''(\tau)}$ for some integer~$k$, then ${\psi=\varphi^{\circ k}}$. This improves the rigidity results \cite[Theorem~2.4\,(6)]{BTV} and \cite[Theorem~1.2\,(3)]{Tauraso} for the case of univalent self-maps.
\end{remark}

Now we are ready to prove a dichotomy relating embeddability to the local structure of the centralizer. It is rather similar to Theorem~\ref{PR_emb-dichotomy-P0HS}; however, in contrast to the case of zero hyperbolic step, instead of $\mathcal A_\varphi$ we have to consider $\mathcal A^*_\varphi:=\Smap_\varphi\big(\Zen(\varphi)\big)$, which is not necessarily an isomorphic image of the centralizer. It is worth mentioning that \textit{a posteriori}, the result we state and prove below holds even with $\mathcal A_\varphi$ in place of~$\mathcal A^*_\varphi$, because ${\mathcal A_\varphi\subset \mathcal A^*_\varphi}$ and because  ${[0,+\infty)}\subset\mathcal A_\varphi$ if $\varphi$ is embeddable.

\begin{theorem}\label{TH_emb-dichotomy-P-PHS} Let $\varphi\in\U(\UD)$ be parabolic of positive hyperbolic step. Exactly one of the following alternatives holds:
	\begin{itemize}
		\item[\rm (i)] either there exist $\rho>0$ and $\delta>0$ such that
		$$\mathcal A^*_{\varphi }\,\cap\,\{c\in\C: 0<|c|<\rho, |\mathrm{Arg}(c)|<\delta\}~=~\emptyset,$$
		\item[\rm (ii)] or there exists a unique continuous one-parameter semigroup $(\phi_t)\subset\U(\UD)$ with $\phi_1=\varphi$.
	\end{itemize}
\end{theorem}

\begin{proof}
If (ii) holds, then $[0,+\infty)\subset \mathcal{A}_{\varphi}\subset\mathcal{A}^*_{\varphi}$ and, clearly, (i) does not hold.

Conversely, assume that (i) does not hold. By Theorem~\ref{Thm:PHS-S}, $\mathcal A^*_\varphi$ is closed w.r.t. taking sums and also topologically closed. Hence, arguing as in the proof of Theorem~\ref{PR_emb-dichotomy-P0HS}, we obtain the inclusion ${[0,+\infty)}\subset \mathcal A_\varphi^*$. Taking into account that by Theorem~\ref{Thm:PHS-S}\,\ref{IT_PHS-real}, ${\mathcal A_\varphi^*\cap\Real=\mathcal A_\varphi\cap\Real}$, this means that in fact, $\mathcal A_\varphi$ contains ${[0,+\infty)}$ and again as in the proof of Theorem~\ref{PR_emb-dichotomy-P0HS}, we conclude that (ii) must hold.
\end{proof}

We complete this section with two propositions, which are in particular useful for constructing non-trivial examples.
\begin{proposition}\label{PR_semigroup-lifitng}
The following two statements hold.
\begin{Ourlist}
  \item\label{IT_semigroup-from-disk} Let $(\phi_t)$ be a continuous one-parameter semigroup in~$\UD$ with Denjoy\,--\,Wolff point at~$0$. Then there exists a unique continuous one-parameter semigroup $(\Psi_t)$ in~$\UH$ such that
      \begin{align}\label{EQ_semigroup-periodicity}
                      \Psi_t(w+1)&=\Psi_t(w)+1\\
                   \label{EQ_semigroup-lifitng}
      \text{and}\quad \exp\big(2\pi i\,\Psi_t(w)\big)\,&=\,\phi_t(e^{2\pi i w})
      \end{align}
      for all~$t\ge0$ and all~$w\in\UH$.
      \smallskip
  \item\label{IT_semigroup-to-disk} Conversely, let  $(\Psi_t)$ be a continuous one-parameter semigroup in~$\UH$ satisfying~\eqref{EQ_semigroup-periodicity}. Then there exists a unique continuous one-parameter semigroup $(\phi_t)$ in~$\UD$ such that \eqref{EQ_semigroup-lifitng} holds. The semigroup~$(\phi_t)$ is elliptic with Denjoy\,--\,Wolff point at~$0$.
\end{Ourlist}\smallskip
Moreover, if $(\phi_t)$ and $(\Psi_t)$ are continuous one-parameter semigroups related as above, then the following three statements hold.\smallskip
\begin{Ourlist}\addtocounter{enumi}{2}
  \item\label{IT_lifting-generators} There exists a holomorphic function $p:\UD\to\C\,$  with~${\Re p\ge0}$ such that
      \begin{equation}\label{EQ_lifting-generators}
        G_\phi(z)=-zp(z)~\text{~for all~$~z\in\UD$}\quad\text{and}\quad G_\Psi(w)=\frac{i}{2\pi}p(e^{2\pi iw})~\text{~for all~$~w\in\UH$},
      \end{equation}
      where $G_\phi$ and $G_\Psi$ are the infinitesimal generators of $(\phi_t)$ and $(\Psi_t)$, respectively.
      \smallskip
  \item\label{IT_Koenigs-function-periodicity}  The Koenigs function $H$ of~$(\Psi_t)$ satisfy
 \begin{equation}\label{EQ_Koenigs-function-periodicity}
    H(w+1)=H(w)+\frac{2\pi}{ip(0)}\quad\text{for any~$~w\in\UH$}.
 \end{equation}
 \item\label{IT_Psi_t-type} The one-parameter semigroup $(\Psi_t)$ is parabolic with Denjoy\,--\,Wolff point at~$\infty$, and it is of zero hyperbolic step unless $p$ is an imaginary constant.
\end{Ourlist}
\end{proposition}
\begin{proof}
Suppose that $(\phi_t)$ is as in the hypothesis of~\ref{IT_semigroup-from-disk}. The existence of a unique one-parameter semigroup $(\Psi_t)\subset\Hol(\UH)$ satisfying~\eqref{EQ_semigroup-lifitng} is essentially by \cite[Proposition~2.1]{pavel}. The first equality in~\eqref{EQ_lifting-generators} is by the Berkson\,--\,Porta representation formula for infinitesimal generators in~$\UD$, see e.g. \cite[Theorem~10.1.10 on p.\,278]{BCD-Book}. Note that ${p\not\equiv0}$ thanks to our convention that all one-parameter semigroups we consider are non-trivial, see Sect.\,\ref{SS_one-param-semigr}. The second equality in~\eqref{EQ_lifting-generators} is again by \cite[Proposition~2.1]{pavel}. This equality implies, in particular, that $(\Psi_t)$ has no fixed point in~$\UH$ and that the trajectories of~$(\Psi_t)$ cannot converge to any point in~$\Real$. This proves that $(\Psi_t)$ is  non-elliptic and that its Denjoy\,--\,Wolff point is at~$\infty$. In turn, it follows that the Koenigs function~$H$ of~$(\Psi_t)$ satisfies ${H'=1/G_\Psi}$. Since by~\eqref{EQ_lifting-generators}, ${G_\Psi(w+1)=G_\Psi(w)}$ for all ${w\in\UH}$, it follows that $H(w+1)=H(w)+c_0$ for all ${w\in\UH}$ and some constant~$c_0\in\C$. Let ${w_0:=i(2\pi)^{-1}}$. Then, using the variable change $z=e^{2\pi i w}$, we have
$$
c_0=H(w_0+1)-H(w_0)=\int\limits_{w_0}^{w_0+1}\!\!\frac{1}{G_\Psi(w)}\,\di w~=\varointctrclockwise\limits_{|z|=1/e}\!\!\frac{2\pi}{ip(z)}\,\frac{\di z}{2\pi i z}~=~\frac{2\pi}{i p(0)}.
$$
This proves~\eqref{EQ_Koenigs-function-periodicity}, which in turn implies, in view of Abel's equation ${H\circ\Psi_t}={H+t}$ and univalence of~$H$, identity~\eqref{EQ_semigroup-periodicity}.\\

Furthermore, from~\eqref{EQ_Koenigs-function-periodicity} it follows that the base space~$S$ of the canonical model for~$(\Psi_t)$ contains all horizontal lines of the form $\Real+H(i)+2\pi k/\big(ip(0)\big)$ with ${k\in\Z}$. This shows that if $p$ is not an imaginary constant and hence ${\Re p>0}$, then $S$ cannot be a half-plane or a horizontal strip, and as a consequence, $(\Psi_t)$ is parabolic of zero hyperbolic step. If $p$ is an imaginary constant, then using~\eqref{EQ_lifting-generators} it is easy to see that $\Psi_t$'s are parabolic automorphisms of~$\UH$.
\medskip

It remains to prove~\ref{IT_semigroup-to-disk}. Let $(\Psi_t)$ be as in the hypothesis of~\ref{IT_semigroup-to-disk}. Then each $\Psi_t$ is univalent and satisfies~\eqref{EQ_semigroup-periodicity}. Hence by Lemma~\ref{Lem:intertwining-exp}\,\ref{IT_periodic-univalent}, there exists a family $(\phi_t)_{t\ge0}$ of univalent holomorphic functions in~$\UD$ with ${\phi_t(0)=0}$ for all ${t\ge0}$ and satisfying identity~\eqref{EQ_semigroup-lifitng}. Since ${\Psi_t(\UH)\subset\UH}$, identity~\eqref{EQ_semigroup-lifitng} implies that ${\phi_t(\UD^*)\subset\UD^*}$, so that ${\phi_t(\UD)\subset\UD}$ for all ${t\ge0}$. Since ${\Psi_0=\id_\UH}$, from~\eqref{EQ_semigroup-lifitng} for ${t:=0}$ it follows that ${\phi_0=\id_\UD}$. Furthermore, using~\eqref{EQ_semigroup-lifitng} two more times, we get
\begin{eqnarray*}
\phi_t\big(\phi_s(e^{2\pi i w})\big) &=& \phi_t\big(\exp\big(2\pi i\,\Psi_s(w)\big)\big)~=~\exp\big(2\pi i\,\Psi_t\big(\Psi_s(w)\big)\big)\\
&=& \exp\big(2\pi i\,\Psi_{t+s}(w)\big)~=~\phi_{t+s}(e^{2\pi i w})\quad\text{for all~$~w\in\UH$},
\end{eqnarray*}
and we immediately conclude that $\phi_t\circ\phi_s=\phi_{t+s}$ for all $\,{t,s\ge0}$.

Finally, since for any $w\in\UH$, ${\Psi_t(w)\to w}$ as ${t\to0^+}$, using~\eqref{EQ_semigroup-periodicity} it is easy to see that ${\phi_t(z)\to z}$ as  ${t\to0^+}$ for any ${z\in\UD^*}$. Thus recalling that $\phi_t(0)=0$ for all ${t\ge0}$ we may conclude that $(\phi_t)$ as a continuous one-parameter semigroup with Denjoy\,--\,Wolff point at~$0$, as desired.
\end{proof}

\begin{proposition}\label{PR_forExample}
Let $(\Psi_t)$ be a one-parameter semigroup in~$\UH$ with Denjoy\,--\,Wolff point at~$\infty$ and infinitesimal generator ${G:\UH\to\C}$. Suppose that ${\Psi_t(w+1)=\Psi_t(w)+1}$ for all~${t\ge0}$ and all~${w\in\UH}$. Then:
\begin{Ourlist}
\item  There exists finite non-zero angular limit
\begin{equation}\label{EQ_asIm_to_infty}
 G(\infty):=\lim_{\Im w\to+\infty}G(w),\quad\text{ with $~\Im G(\infty)\ge0$.}
\end{equation}

\item
Let $\varphi\in\U(\UD)$ be a self-map with canonical model ${(\UH,h,z\mapsto z+1)}$.
For every $t\ge0$ such that ${\Psi_t(\Omega)\subset\Omega:=h(\UD)}$ we have
\begin{equation}\label{EQ_forSmap}
 \Smap_{\varphi}(\psi_t)=tG(\infty),\quad  \text{ where }\ \psi_t:=h^{-1}\circ\Psi_t\circ h~\in~\Zen(\varphi).
\end{equation}
\end{Ourlist}
\end{proposition}
\begin{proof}
By Proposition~\ref{PR_semigroup-lifitng}, there exists a holomorphic function $p:\UD\to\C$ with ${\Re p\ge0}$ and such that $G(w)={ip(e^{2\pi i w})/(2\pi)}$ for all ${w\in\UH}$. This implies~\eqref{EQ_asIm_to_infty}, with $G(\infty)={ip(0)/(2\pi)}$.  In particular, identity~\eqref{EQ_Koenigs-function-periodicity} can be rewritten as ${G(\infty)H(w+1)}={G(\infty)H(w)+1}$, where $H$ stands for the Koenigs function of~$(\Psi_t)$. It follows that the univalent function $h_*:=G(\infty)H\circ h$ satisfies Abel's equation ${h_*\circ\varphi}=h_*+1$. Moreover, for every fixed ${t\ge0}$ such that ${\Psi_t(\Omega)\subset\Omega}$, we have
$$
 h_*\circ\psi_t=G(\infty)H\circ\Psi_t\circ h=G(\infty)(H+t)\circ h=h_*+G(\infty)t.
$$
Therefore, $h_*$  satisfies the system of Abel's equations~\eqref{EQ_two-Abel-in-remark} for ${\psi:=\psi_t}$ and $c_{\varphi,\psi}:=G(\infty)t$. This proves~\eqref{EQ_forSmap}, because by Remark~\ref{Rem:cinPHP}, the only value of $c_{\varphi,\psi}$ for which~\eqref{EQ_two-Abel-in-remark} admits a univalent solution is~$\Smap_\varphi(\psi)$.
\end{proof}

%%%%%%%%%%%%%%%%%%%%%%%%%%%%%%%%%%%%%%%%%%%%%%%%%%%%%%%%%%%%%%%%%%%%%%%%%%%%%%%%%%%%%%%%%%%%%%%%%%
%%%%%%%%%%%%%%%%%%%%%%%%%%%%%%%%%%%%%%%%%%%%%%%%%%%%%%%%%%%%%%%%%%%%%%%%%%%%%%%%%%%%%%%%%%%%%%%%%%
%%%%%%%%%%%%%%%%%%%%%%%%%%%%%%%%%%%%%%%%%%%%%%%%%%%%%%%%%%%%%%%%%%%%%%%%%%%%%%%%%%%%%%%%%%%%%%%%%%
%%%%%%%%%%%%%%%%%%%%%%%%%%%%%%%%%%%%%%%%%%%%%%%%%%%%%%%%%%%%%%%%%%%%%%%%%%%%%%%%%%%%%%%%%%%%%%%%%%
%%%%%%%%%%%%%%%%%%%%%%%%%%%%%%%%%%%%%%%%%%%%%%%%%%%%%%%%%%%%%%%%%%%%%%%%%%%%%%%%%%%%%%%%%%%%%%%%%%
%%%%%%%%%%%%%%%%%%%%%%%%%%%%%%%%%%%%%%%%%%%%%%%%%%%%%%%%%%%%%%%%%%%%%%%%%%%%%%%%%%%%%%%%%%%%%%%%%%
\section{Examples}\label{S_examples}
%%%%%%%%%%%%%%%%%%%%%%%%%%%%%%%%%%%%%%%%%%%%%%%%%%%%%%%%%%%%%%%%%%%%%%%%%%%%%%%%%%%%%%%%%%%%%%%%%%
%%%%%%%%%%%%%%%%%%%%%%%%%%%%%%%%%%%%%%%%%%%%%%%%%%%%%%%%%%%%%%%%%%%%%%%%%%%%%%%%%%%%%%%%%%%%%%%%%%
%%%%%%%%%%%%%%%%%%%%%%%%%%%%%%%%%%%%%%%%%%%%%%%%%%%%%%%%%%%%%%%%%%%%%%%%%%%%%%%%%%%%%%%%%%%%%%%%%%
We start with a rather obvious example, the centralizer of a parabolic automorphism.
\begin{example}\label{EX_parab-autom}
Let $\varphi$ be a parabolic automorphism of~$\UD$ with Denjoy\,--\,Wolff point~$\tau$. Then there exists ${a\in\Real\setminus\{0\}}$ such that the canonical holomorphic model for~$\varphi$ is $\big(\sgn(a)\UH,h,z\mapsto z+1\big)$, where $h(z):={ia(\tau+z)/(\tau-z)}$ for all ${z\in\UD}$. Without loss of generality we will restrict our consideration to the case ${a>0}$. Clearly, the Koenigs domain $\Omega:=h(\UD)$ coincides with the base space, e.g. with~$\UH$. Therefore, by Theorem~\ref{Thm:general},
$$
\Zen(\varphi)=\big\{h^{-1}\circ g\circ h: g\in\Zen_\UH(w\mapsto w+1)\big\}.
$$
Using Lemma~\ref{Lem:intertwining-exp} or \cite[Theorem~1.2.27]{Abate}, it is easy to see that there is a natural surjective map from $\Zen(\varphi)$ onto the subsemigroup $\U_0(\UD)$ of $\U(\UD)$ formed by of all $f\in\U(\UD)$ satisfying ${f(0)=0}$. This map is given by
$$
\Zen(\varphi)\,\ni\,\psi~\mapsto~f_\psi\in\U_0(\UD), \qquad f_\psi(e^{2\pi i w})=\exp\big(2\pi i\big(h\circ \psi\circ h^{-1}\big)(w)\big)\quad\text{for all~$w\in\UH$}.
$$
A simple calculation shows that this map is a \textit{semigroup homomorphism} and that the preimage of any ${f\in\U_0(\UD)}$ is a countable set.

Note also that for the additive semigroups $\mathcal A_\varphi$ and $\mathcal A_\varphi^*$ defined in Theorems~\ref{PropA_varphi} and~\ref{Thm:PHS-S}, in this example, we have $\mathcal A^*_\varphi=\mathcal A_\varphi={\UH\cup\R}$. Indeed, the latter of the equality signs can be easily checked using the very definition of~$\mathcal A_\varphi$. The former equality sign holds because $\mathcal A_\varphi\subset\mathcal A^*_\varphi\subset{\UH\cup\R}$ by Theorem~\ref{Thm:PHS-S} and Remark~\ref{RM_left-inverse}.

It is worth mentioning that although, being an automorphism, $\varphi$ is of positive hyperbolic step, its centralizer $\Zen(\varphi)$ contains a lot of parabolic self-maps of zero hyperbolic step, for example $h^{-1}\circ{(w\mapsto w+c)}\circ h$ for any ${c\in\UH}$.

Finally, it follows from Proposition~\ref{PR_semigroup-lifitng} that there is a one-to-one correspondence between continuous one-parameter semigroups contained in~$\Zen(\varphi)$ and those contained in~$\U_0(\UD)$, and that each semigroup ${(\psi_t)\subset\Zen(\varphi)}$ is of zero hyperbolic step, unless it consists of parabolic automorphisms (and hence contains~$\varphi$ or its inverse).
\end{example}

\begin{remark}\label{RM_quotient}
In the above example the centralizer is huge: in fact, the quotient semigroup $\Zen(\varphi)\,/\!\widesim{\varphi}$, where $\psi_1\widesim{\varphi}\psi_2$ if $\psi_1=\varphi^{\circ k}\circ\psi_2$ for a some ${k\in\Z}$, can be identified with $\U_0(\UD)$. In the next example, the situation is quite opposite.
\end{remark}

\begin{example}
For any $\varphi\in\U(\UD)$, obviously $\{\varphi^{\circ n}:n\in\N_0\}\subset\Zen(\varphi)$. This inclusion is not necessarily strict. Indeed,  by Proposition~\ref{PROP_reciprocal} applied with $\mathcal A:=\N_0$, there exists a hyperbolic self-map ${\varphi\in\U(\UD)}$ with ${\mathcal A_\varphi=\N_0}$. By Theorem~\ref{TH_isomorphism}, the centralizer of such a self-map coincides exactly with $\{\varphi^{\circ n}:n\in\N_0\}$.
\end{example}

In the next example we construct a \textit{non-embeddable} parabolic self-map~$\varphi\in\U(\UD)$ such that $\id_\UD$ is not isolated $\Zen(\varphi)$. Recall that by Theorem~\ref{TH_dichotomy-nonelliptic}, $\Zen(\varphi)$ must contain a non-trivial continuous one-parameter semigroup if $\id_\UD$ is an accumulation point of~$\Zen(\varphi)$. However, $\varphi$ does not have to be an element of this semigroup.

\begin{example}\label{EX_non-non}
For each $n\in \Z$ let $\Gamma_{n}:=\{n+iy: \, y\leq-(n+1)\}$, and let $h$ be a conformal map of~$\UD$ onto ${\Omega:=\C\setminus \bigcup_{n\in \Z}\Gamma_{n}}$ satisfying ${h(0)=0}$.  Define $\varphi:={h^{-1}\circ(h+1)}$. It is easy to see that $\varphi\in\U(\UD)$ and that ${(\C,h,z\mapsto z+1)}$ is the canonical holomorphic model for~$\varphi$. By Theorem~\ref{Thm:model}, $\varphi$ is parabolic of zero hyperbolic step.
Moreover, it is not difficult to check that the semigroup $\mathcal A_\varphi$, defined in Theorem~\ref{PropA_varphi}, coincides with $\big\{k+it:\,k\in\Z,~t\ge -k\big\}.$ In particular, ${\Omega+it}\subset\Omega$ for any ${t\ge0}$.
Therefore, $(\phi_{t}):={\big(h^{-1}\circ (z\mapsto z+it)\circ h\big)}$ is a continuous one-parameter semigroup in~$\UD$ contained in~$\mathcal Z(\varphi)$. It follows that the map $\id_\UD$ is not isolated in the centralizer of~$\varphi$.
At the same time, the inclusion ${\Omega+t\subset\Omega}$ fails for any non-integer ${t>0}$. Thus, $\varphi$ is  not embeddable; see Remark~\ref{RM_emb-ty_via_K-dom}.
\end{example}

In the next three examples we make use of the continuous one-parameter semigroup $(\Psi^*_t)$ in the upper half-plane~$\UH$ obtained, as explained in Proposition~\ref{PR_semigroup-lifitng}\,\ref{IT_semigroup-from-disk},  by lifting the continuous one-parameter semigroup in~$\UD$ generated by $G_0(z):={-z(1+z)/(1-z)}$.  The infinitesimal generator of~$(\Psi^*_t)$ is
\begin{equation*}%\label{EQ_Gstar}
G_*(w)=\frac{i}{2\pi}\,\frac{1+e^{2\pi i w}}{1-e^{2\pi i w}}=-\frac1{2\pi}\cot(\pi w),\quad w\in\UH,
\end{equation*}
and up to an additive constant, its Koenigs function $h_*$ is given by $h_*(w)=2\log(\cos\pi w)$.

\medskip

The centralizer of a parabolic automorphism, discussed in Example~\ref{EX_parab-autom}, is clearly non-abelian, e.g. because otherwise the quotient semigroup $\Zen(\varphi)\,/\!\widesim{\varphi}$ (see Remark~\ref{RM_quotient}) would be also abelian. In many respects, automorphisms constitute a very particular case, quite different from that of self-maps in~${\Hol(\UD)\setminus\Aut}$.
In the example below we construct a non-elliptic $\varphi\in\U(\UD)$ \textit{different from an automorphism} such that $\Zen(\varphi)$ is also non-abelian. In particular, it follows that ${\Zen(\varphi)\neq\Tmap_\varphi(\mathcal A_\varphi)}$ in this case. Recall that according to Theorems~\ref{Thm:0HS} and~\ref{TH_isomorphism}, such a situation can occur for a non-elliptic self-map $\varphi\in\U(\UD)$ only if $\varphi$ is parabolic of positive hyperbolic step.

\begin{example}[Non-commutative centralizer]\label{EX_Z-non-abelian}
Let $(\Psi^*_t)$, $G_*$, and $h_*$ be as above, and consider the continuous one-parameter semigroup $(\phi_t)$ in~$\UD$ whose Denjoy\,--\,Wolff is point at ${\tau=1}$ and whose Koenigs domain is $\Omega:={\UH\setminus\{w:\Re w\le0,\,\Im w\le1\}}$.
By symmetry w.r.t. the imaginary axis, $\Psi^*_t(Q_k)\subset Q_k:={\{w\in\UH:(-1)^k\Re w\ge0\}}$ for ${k=1,2}$ and all ${t\ge0}$. Moreover, since the Denjoy\,--\,Wolff point of~$(\Psi^*_t)$ is at~$\infty$, $\Psi^*_t(\UH_1)\subset\UH_1:={\{w:\Im w>1\}}$ for all~${t\ge0}$. It follows that ${\Psi^*_t(\Omega)\subset\Omega}$ for all~$t\ge0$.  Therefore, taking into account that by~\eqref{EQ_semigroup-periodicity}, ${\Psi^*_t(w+1)}={\Psi^*_t(w)+1}$ for all ${w\in\UH}$ and all ${t\ge0}$, and using Theorem~\ref{Thm:general}, we conclude that $\Zen(\phi_1)$ contains the whole continuous one-parameter semigroup $(\psi^*_t)$ given by $\psi^*_t:={h^{-1}\circ\Psi^*_t\circ h}$, ${t\ge0}$, where $h$ is the Koenigs map of~$(\phi_t)$.
Clearly, also $(\phi_t)\subset\Zen(\phi_1)$.

We have ${h\circ\phi_t}={h+t}$ and ${(h_*\circ h)\circ\psi_t^*}={h_*\circ\Psi_t^*\circ h}={(h_*\circ h)+t}$ for all ${t\ge0}$. It follows that the infinitesimal generators of the continuous one-parameter semigroups $(\phi_t)$ and $(\psi_t^*)$ are given by ${G_1=1/h'}$ and $G_2={1/(h_*\circ h)'}={(G_*\circ h)\,G_1}$, respectively. In particular, $G_2$ is not a constant multiple of~$G_1$. Therefore, by~\cite[Theorem~4]{conTauraso}, the semigroups $(\phi_t)$ and $(\psi^*_t)$ do not commute, i.e. there exist ${t>0}$ and ${s>0}$ such that ${\psi^*_t\circ\phi_s}\neq{\phi_s\circ\psi^*_t}$. (Compare with \cite[Theorem~5]{conReich} or \cite[Theorem~6.18]{EliShobook10}.) It immediately follows that $\Zen(\phi_1)$ is non-abelian.
\end{example}

\begin{example}\label{EX_again-non-abelian}
Let $(\Psi^*_t)$, $G_*$, and $h_*$ be as above. Consider now the self-map $\varphi\in\U(\UD)$ with the Denjoy\,--\,Wolff point at~${\tau=1}$ and the Koenigs domain  $$\Omega:=\UH\,\,\,\Big\backslash\,\bigcup_{k\in\N_0}~\Gamma_k,\quad\Gamma_k:=\{-k+iy:y\in[0,1]\}.$$
Denote by $h$ the Koenigs map of~$\varphi$ and let ${\psi^*_t:={h^{-1}\circ\Psi^*_t\circ h}}$ for ${t\ge0}$.
As a direct computation shows, $h_*(\Omega)$ is starlike at infinity. Therefore, all $\psi_t^*$'s are well-defined univalent holomorphic self-maps of~$\UD$. Moreover, similarly to the previous example, we may conclude that $(\psi_t^*)$ is a continuous one-parameter semigroup contained in~$\Zen(\varphi)$.

In contrast to the previous example, by Remark~\ref{RM_emb-ty_via_K-dom}, $\varphi$ is not embeddable.
At the same time, $\Zen(\varphi)$ contains the semigroup $(\psi_t)$ defined by $\psi_t:={h^{-1}\circ (z\mapsto z+it)\circ h}$, which does not commute with $(\psi_t^*)$. (The latter can be seen again by appealing to \cite[Theorem~4]{conTauraso}.) Therefore, also in this example, $\Zen(\varphi)$ is non-abelian.
\end{example}

The next example shows that for parabolic self-maps $\varphi\in\U(\UD)$ of positive hyperbolic step, $\mathcal A_\varphi$ can be a proper subset of ${\mathcal A_\varphi^*:=\Smap_\varphi(\Zen(\varphi))}$; see Theorems~\ref{PropA_varphi} and~\ref{Thm:PHS-S} for the definitions of~$\mathcal A_\varphi$ and $\Smap_\varphi$.

\begin{example}\label{EX_A-neq-Astar}
Once  again, let $(\Psi^*_t)$, $G_*$, and $h_*$ be as above, and consider the slits
$$
 \widetilde\Gamma_k:=h_*^{-1}\big(\{x\,+\,2\pi i(k+\tfrac13)\,:\,x\le0\}\big).
$$

Let ${\varphi\in\U(\UD)}$ be the self-map with the Denjoy\,--\,Wolff point at~${\tau=1}$ and the Koenigs domain $\Omega:=\UH\setminus\bigcup_{\,k\in\N_0}\widetilde\Gamma_k$.
Given a fixed $a\in \R$, for all $t>0$ we have
$$
q(t):=\tan\big(\Im h_{*}(a+it)\big)=-\tan(\pi a)\tanh(\pi t),
$$
which can be a finite constant only if $q\equiv0$.
Thus, $\widetilde\Gamma_k$'s do not lie on straight lines parallel to the imaginary axis. Therefore, the set $\mathcal A_\varphi$ does not contain points of the form ${w:=it}$ with small~${t>0}$. At the same time, as in the previous example, $\Zen(\varphi)$ contains the one-parameter semigroup $(\psi^*_t)$ defined by $\psi^*_t:={h^{-1}\circ\Psi^*_t\circ h}$, ${t\ge0}$, where $h$ stands for the Koenigs map of~$\varphi$. Since $G_*(\infty)=i/(2\pi)$, by Proposition~\ref{PR_forExample}, it follows that ${\{it:t\ge0\}}\subset\Smap_\varphi\big(\Zen(\varphi)\big)$. Therefore, ${\Smap_\varphi(\Zen(\varphi))\neq\mathcal A_\varphi}$.

It is worth mentioning that the slits $\Gamma_k$ in the previous example are backward orbits of the semigroup~$(\Psi^*_t)$ landing at the pole of the infinitesimal generator within a finite time. In this example, the slits $ \widetilde\Gamma_k$ are backward orbits tending to boundary repelling fixed points of~$(\Psi^*_t)$ in the $\alpha$-limit, i.e. as~$t\to-\infty$. (For more details on backward dynamics of continuous one-parameter semigroups, interested readers are referred to~\cite{BCKWRD1,BCKWRD2} or to~\cite[Sect.\,13]{BCD-Book}).
\end{example}

\begin{remark}\label{RM_examples-not-iso}
 In each of the above three examples, the semigroup homomorphisms $${\Tmap_\varphi:\mathcal A_\varphi\to\Zen(\varphi)}\qquad\text{ and }\qquad{\Smap_\varphi:\Zen(\varphi)\to\mathcal A_\varphi^*}:=\Smap\big(\Zen(\varphi)\big)$$ are not isomorphisms. Indeed, in these examples, $\Smap_\varphi$ is not injective, while the injective homomorphism~$\Tmap_\varphi$ fails to be surjective.
\end{remark}

\begin{remark}\label{RM_PHS-commute-with-PHS}
Moreover, in each of the above three examples, the elements of the semigroup $(\psi_t^*)$, in contrast to~$\varphi$, are of zero hyperbolic step. Indeed, the orbits of $(\Psi_t^*)$ are tangent at the Denjoy\,--\,Wolff point~$\infty$ to the imaginary axis. Since in each of the three examples, the map $1/h$ is conformal and hence isogonal at~$\tau$, see e.g. \cite[Sect.\,4.3]{Pommerenke:BB}, all the trajectories of $(\psi_t^*)$ approach~$\tau$ along the radial direction. This is only possible if $(\psi_t^*)$ is of zero hyperbolic step, see e.g. \cite[Sect.\,17.5]{BCD-Book}.
\end{remark}

%%%%%%%%%%%%%%%%%%%%%%%%%%%%%%%%%%%%%%%%%%%%%%%%%%%%%%%%%%%%%%%%%%%%%%%%%%%%%%%%%%%%%%%%%%%%%%%%%%
%%%%%%%%%%%%%%%%%%%%%%%%%%%%%%%%%%%%%%%%%%%%%%%%%%%%%%%%%%%%%%%%%%%%%%%%%%%%%%%%%%%%%%%%%%%%%%%%%%
%%%%%%%%%%%%%%%%%%%%%%%%%%%%%%%%%%%%%%%%%%%%%%%%%%%%%%%%%%%%%%%%%%%%%%%%%%%%%%%%%%%%%%%%%%%%%%%%%%
%%%%%%%%%%%%%%%%%%%%%%%%%%%%%%%%%%%%%%%%%%%%%%%%%%%%%%%%%%%%%%%%%%%%%%%%%%%%%%%%%%%%%%%%%%%%%%%%%%
%%%%%%%%%%%%%%%%%%%%%%%%%%%%%%%%%%%%%%%%%%%%%%%%%%%%%%%%%%%%%%%%%%%%%%%%%%%%%%%%%%%%%%%%%%%%%%%%%%
%%%%%%%%%%%%%%%%%%%%%%%%%%%%%%%%%%%%%%%%%%%%%%%%%%%%%%%%%%%%%%%%%%%%%%%%%%%%%%%%%%%%%%%%%%%%%%%%%%
\section{Appendix}
%%%%%%%%%%%%%%%%%%%%%%%%%%%%%%%%%%%%%%%%%%%%%%%%%%%%%%%%%%%%%%%%%%%%%%%%%%%%%%%%%%%%%%%%%%%%%%%%%%
%%%%%%%%%%%%%%%%%%%%%%%%%%%%%%%%%%%%%%%%%%%%%%%%%%%%%%%%%%%%%%%%%%%%%%%%%%%%%%%%%%%%%%%%%%%%%%%%%%
%%%%%%%%%%%%%%%%%%%%%%%%%%%%%%%%%%%%%%%%%%%%%%%%%%%%%%%%%%%%%%%%%%%%%%%%%%%%%%%%%%%%%%%%%%%%%%%%%%
For a given non-elliptic $\varphi\in\U(\UD)$, let $h\in\Hol(\UD,\C)$ be a univalent solution to Abel's equation
$
         h\circ\varphi=h+1
$
and let
$$
S[h]~:=~\bigcup_{n\in\Natural}\, h(\UD)-n.
$$
Then, clearly, $S[h]$ is a simply connected domain, $S[h]+1=S[h]$, and $(S[h],h,z\mapsto z+1)$ is a holomorphic model for~$\varphi$ but, in principal, quite different from the canonical one. Nevertheless, as the theorem below shows, it is possible, similarly to the canonical models, to distinguish between hyperbolic and parabolic self-maps looking at~$S[h]$.

\begin{theorem}\label{TH_ch-para-hyper}
	Let $\varphi\in\U(\UD)$ be non-elliptic, and let $h$ and $S[h]$ be as above. The following statements hold:
	\begin{Ourlist}
		\item\label{IT_(A)} $\varphi$ is parabolic of zero hyperbolic step if and only if $S[h]=\C$.\smallskip
		\item\label{IT_(B)} $\varphi$ is parabolic of positive hyperbolic step with canonical model $(\H,h_0,z\mapsto z+1)$ if and only if $S[h]\neq\C$ and there exists $a\in\R$ such that $\H+ia\subset S[h]$. \smallskip
		\item\label{IT_(C)} $\varphi$ is parabolic of positive hyperbolic step with canonical model $(-\H,h_0,z\mapsto z+1)$ if and only if $S[h]\neq\C$ and there exists $a\in\R$ such that $-\H+ia\subset S[h]$.		
	\end{Ourlist}

In particular:
\begin{Ourlist}\setcounter{enumi}{3}
 \item\label{IT_(D)} $\varphi$ is parabolic if and only if $S[h]$ contains a horizontal half-plane.
\end{Ourlist}
\end{theorem}

\begin{remark}\label{RM_ch-para-hyper}
Clearly, if $S[h]$ contains a half-plane which is not horizontal, then ${S[h]=\C}$. Therefore, assertion~\ref{IT_(D)} in the above theorem is equivalent to the following statement:
\textit{$\varphi$ is parabolic if and only if $S[h]$ contains a half-plane (not necessarily horizontal).}
\end{remark}

In the proof we make use of the following lemma, which is also used in Section~\ref{S_para-positive}.

\begin{lemma}\label{Lem:intertwining-exp}
The following statements hold.
\begin{OurlistM}
 \item\label{IT_periodic-function}  Let $S\subset\C$ be a domain satisfying ${S+1=S}$ and let $F$ be a holomorphic function in~$S$ such that $F(w+1)=F(w)$ for all ${w\in S}$. Then there exists a holomorphic function~$f$ in ${D(S):=\exp\big(2\pi i S\big)}$ such that
    \begin{equation}\label{Eq:intertwining-exp-bis}
       F(w)=f\big(e^{2\pi i\, w}\big)\qquad\text{for all~$~w\in S$}.
    \end{equation}
    \smallskip
\item\label{IT_periodic-univalent}
   Let $g:\UH \to \C$   be univalent such that
   $
   {g(w+1)=g(w)+1}
   $
 for all ${w\in\UH}$. Then, there exists $f\in\U(\UD,\C)$ such that $f(0)=0$ and
    \begin{equation}\label{Eq:intertwining-exp}
       e^{2\pi i\, g(w)}=f\big(e^{2\pi i\, w}\big)\qquad\text{for all~$~w\in\H$.}
   \end{equation}
   \smallskip
 \item\label{IT_periodic-univalent-converse} Conversely, if $f:\UD\to\C$ is a univalent function with ${f(0)=0}$, then there exists a univalent function ${F:\UH\to\C}$ such that
    \begin{equation}\label{EQ_periodic-univalent-converse}
        F(w+1)=F(w)+1\quad\text{and}\quad f(e^{2\pi i w})=\exp\big(2\pi i F(w)\big) \quad\text{for all ${w\in\UH}$.}
     \end{equation}
 \end{OurlistM}
\end{lemma}\vskip-0.5ex
%\begin{proof}
\StepP{\ref{IT_periodic-function}} The formula $w\mapsto C(w):={e^{2\pi i w}}$ defines a map from $S$ onto $D(S)$. By the hypothesis, ${F(w_1)=F(w_2)}$ for any ${w_1,w_2\in S}$ satisfying ${C(w_1)=C(w_2)}$. Therefore,  there exists a function ${f:D(S)\to\C}$ such that ${F=f\circ C}$, i.\,e. such that~\eqref{Eq:intertwining-exp-bis} holds. To show that $f$ is holomorphic, it suffices to notice that ${C:S\to D(S)}$ is surjective and locally univalent, and hence if ${z_0\in D(S)}$, then there exists a point ${w_0\in S}$ and its neighbourhood ${V\subset S}$ such that $C(w_0)=z_0$ and $C$ maps $V$ conformally onto some neighbourhood ${U:=C(V)}$ of~$z_0$. It follows that ${f|_U=F\circ (C|_V)^{-1}}$ is holomorphic in~$U$. Since $z_0$ in this argument is an arbitrary point of~$D(S)$, we are done with~\ref{IT_periodic-function}.

\StepP{\ref{IT_periodic-univalent}} Consider the function $F:\H\to \C$ given by $F(w):=e^{2\pi i g(w)}$, ${w\in \H}$. Clearly,
$
{F(w+1)}={e^{2\pi i g(w+1)}}={e^{2\pi i g(w)}}=F(w)
$ for any ${w\in\UH}$.
 By~\ref{IT_periodic-function}, there exists $f: \D^*\to \C$ such that  \eqref{Eq:intertwining-exp-bis} holds, which by the choice of~$F$ immediately implies~\eqref{Eq:intertwining-exp}.
 Let us check that $f$ is univalent in~$\UD^*$. Take $z_1, z_2\in \D^*$ with ${f(z_1)=f(z_2)}$. Then there are $w_j\in \H$ such that  $e^{2\pi i w_j}=z_j$, $j=1,2$. By~\eqref{Eq:intertwining-exp},  ${e^{2\pi i\, g(w_1)}=e^{2\pi i\, g(w_2)}}$ and hence, there is $k\in \Z$ such that $g(w_1)={g(w_2)+k}={g(w_2+k)}$. Since $g$ is univalent, it follows that ${w_1=w_2+k}$. Thus, $z_1=e^{2\pi i w_1}=e^{2\pi i (w_2+k)}=z_2$. Therefore, $f$ is univalent.

Recall that as a consequence of Picard's Great Theorem, univalent functions cannot have isolated essential singularities.
Hence, for $f$, the point  $z=0$ is a removable singularity or a pole.
Moreover, for any $\varepsilon\in(0,1)$, with the notation $w_0(\varepsilon):=i(2\pi)^{-1}\log(1/\varepsilon)$, we have
$$
\frac{1}{2\pi i}\varointctrclockwise\limits_{|z|=\varepsilon}\frac{f'(z)}{f(z)}\di z~=\int\limits_{w_0(\varepsilon)}^{w_0(\varepsilon)+1}\!\!\!\!\!e^{2\pi i w}\frac{f'(e^{2\pi w})}{f(e^{2\pi w})}\di w~=\int\limits_{w_0(\varepsilon)}^{w_0(\varepsilon)+1}\!\!\!\!\!g'(w)\di w~=~1.
$$
Thus, $z=0$ is a simple zero of~$f$, i.e. $f(0)=0$. Clearly, with this extension $f$ is univalent in~$\D$.

\StepP{\ref{IT_periodic-univalent-converse}} The function $q(w):={e^{2\pi i w} f'(e^{2\pi i
w})/f(e^{2\pi i w})}$ is holomorphic in~$\UH$. Let $F$ be one of its antiderivatives. It is easy to see that
$\Phi(w):=\exp\big(2\pi i F(w)\big)/f(e^{2\pi i w})$ is constant in~$\UH$. (Simply calculate its derivative.) Taking into
account that $F$ is defined up to an additive constant, we can normalize it in such a way that
\begin{equation}\label{EQ_Mlm_g-F}
f(e^{2\pi i w})=\exp\big(2\pi i F(w)\big)\quad\text{for all~$w\in\UH$}.
\end{equation}

Moreover, for any ${w\in \UH}$, we have
\begin{equation}\label{EQ_Mlm_F-periodic+1}
 F(w+1)-F(w)~=~\int_w^{w+1}\frac{e^{2\pi i \zeta} f'(e^{2\pi i \zeta})}{f(e^{2\pi i
 \zeta})}\di\zeta=\frac{1}{2\pi i}\varointctrclockwise_{|z|=\exp(-2\pi \Im w)}\frac{f'(z)}{f(z)}\di z~=~1.
\end{equation}

It remains to show that $F$ is univalent in~$\UH$.  Suppose ${F(w_1)=F(w_2)}$ for some ${w_1,w_2\in
\UH}$. Then taking into account that $f$ is univalent, by \eqref{EQ_Mlm_g-F} we must have ${w_1=w_2+k}$ for some
${k\in\Z}$.  With the help of~\eqref{EQ_Mlm_F-periodic+1}, we deduce that ${F(w_1)=F(w_2+k)=F(w_2)+k}$. Hence,
${k=0}$ and~${w_1=w_2}$.
\qed
%\end{proof}

\bigskip

\begin{proof}[\proofof{Theorem~\ref{TH_ch-para-hyper}}]\mbox{~}\nopagebreak

\StepP{\ref{IT_(A)}} If $S[h]=\C$, then, by Theorem~\ref{Thm:model}, $\varphi$ is parabolic of zero hyperbolic step. Assume now that $S[h]\neq\C$. Then $S[h]$ is a hyperbolic domain, and hence for any $z_0\in\UD$, we have
	\begin{eqnarray*}
		\lim_{n\to+\infty}\rho_{\D}\big(\varphi^{\circ n}(z_0),\varphi^{\circ(n+1)}(z_0)\big)&\geq&
		\lim_{n\to+\infty}\rho_{S[h]}\big(h(\varphi^{\circ(n+1)}(z_0)),h(\varphi^{\circ n}(z_0))\big)\\
		&=&\lim_{n\to+\infty}\rho_{S[h]}\big(h(z_0)+n+1,h(z_0)+n\big)\\&=&
		\rho_{S[h]}\big(h(z_0),h(z_0)+1\big)>0,
	\end{eqnarray*}
	where we have used that $z\mapsto z+1$ is an isometry in $S[h]$ endowed with its hyperbolic distance~$\rho_{S[h]}$.
Hence, $\varphi$ is not parabolic of zero hyperbolic step.

\StepP{\ref{IT_(B)}} Assume that $\varphi$ is parabolic of positive hyperbolic step with canonical model $(\H,h_0,z\mapsto z+1)$. Then, by Theorem~\ref{Thm:uniqness}, there exists a biholomorphism $\eta$ from $\H$ onto $S[h]$ such that ${\eta(z+1)}={\eta(z)+1}$ for all ${z\in\H}$. In particular, $S[h]\neq\C$. According to Lemma~\ref{Lem:intertwining-exp}\,\ref{IT_periodic-univalent}, there exists a holomorphic function $f:\D\to\C$ with $f(0)=0$ and such that $e^{2\pi i \eta(z)}=f(e^{2\pi i z})$ for all ${z\in\H}$.  Clearly, $f(\UD^*)$ contains a punctured neighbourhood of~$0$, i.e.
$$
\{e^{2\pi i z}:\Im z>r\}\subset f(\D^*)=\{e^{2\pi i \eta(z)}:z\in\H\}
$$
for some $r>0$.
Since $\eta(z+1)=\eta(z)+1$ for all ${z\in\H}$, we conclude that ${\{z:\Im z>r\}}\subset\eta(\H)= S[h]$, i.e. $S[h]$ contains some horizontal upper half-plane.

Assume now that $S[h]\neq \C$ and contains a horizontal upper half-plane.
Replacing $h$ with $h+ia$ for a suitable $a\in\Real$, we may suppose that ${\UH\subset S[h]}$.
Since $\varphi$ is non-elliptic,
its canonical holomorphic model is of the form ${(S_I,h_0,z\mapsto z+1)}$, where $S_I:=\{z:\Im z\in I\}$ with $I\in\big\{\Real,{(0,+\infty)},{(-\infty,0)}\big\}$ or with ${I\subset\Real}$ being a bounded open interval.

By Theorem~\ref{Thm:uniqness}, there exists a biholomorphism $\eta$ of~$S[h]$ onto~$S_I$ satisfying ${\eta(w+1)}={\eta(w)+1}$ for all ${w\in S[h]}$. Since ${S[h]\neq\C}$, it immediately follows that also ${S_I\neq\C}$, i.e. ${I\neq\Real}$. Denote ${J:=\{e^{-2\pi y}:y\in I\}}$. Arguing as above, with the help of  Lemma~\ref{Lem:intertwining-exp}\,\ref{IT_periodic-univalent} applied to ${g:=\eta|_\UH}$, we conclude that the set $\{\zeta:|\zeta|\in J\}$, which is the image of $S_I$ under $z\mapsto e^{2\pi i z}$, contains a punctured neighbourhood of~$0$. It follows that $I$ is not bounded from above. Thus, ${I=(0,+\infty)}$, i.e. the base space $S_I$ of the canonical holomorphic model for~$\varphi$ is $\UH$, as desired.

\medskip
\noindent\textsc{Proof of~\ref{IT_(C)}} is similar to that of~\ref{IT_(B)}.
In turn, assertion~\ref{IT_(D)} follows  from \ref{IT_(A)}, \ref{IT_(B)}, and~\ref{IT_(C)}.
\end{proof}

\end{document}